\documentclass{amsart}
\usepackage{amssymb, amsmath, latexsym, amscd, graphicx, color}
\usepackage[all]{xy}
\usepackage[dvipdfm]{hyperref}
\usepackage[all]{hypcap}
\newtheorem{theorem}{Theorem}[section]
\newtheorem{lemma}[theorem]{Lemma}
\newtheorem{cor}[theorem]{Corollary}
\newtheorem{definition}[theorem]{Definition}
\newtheorem{example}[theorem]{Example}
\newtheorem{remark}[theorem]{Remark}
\newtheorem{proposition}[theorem]{Proposition}
\def\pagenumber{1}

\begin{document}
\setcounter{page}{\pagenumber}
\newcommand{\T}{\mathbb{T}}
\newcommand{\R}{\mathbb{R}}
\newcommand{\Q}{\mathbb{Q}}
\newcommand{\N}{\mathbb{N}}
\newcommand{\Z}{\mathbb{Z}}
\newcommand{\tx}[1]{\quad\mbox{#1}\quad}
\parindent=0pt
\def\SRA{\hskip 2pt\hbox{$\joinrel\mathrel\circ\joinrel\to$}}
\def\tbox{\hskip 1pt\frame{\vbox{\vbox{\hbox{\boldmath$\scriptstyle\times$}}}}\hskip 2pt}
\def\circvert{\vbox{\hbox to 8.9pt{$\mid$\hskip -3.6pt $\circ$}}}
\def\IM{\hbox{\rm im}\hskip 2pt}
\def\ES{\vbox{\hbox to 8.9pt{$\big/$\hskip -7.6pt $\bigcirc$\hfil}}}
\def\TR{\hbox{\rm tr}\hskip 2pt}
\def\GRAD{\hbox{\rm grad}\hskip 2pt}
\def\bull{\vrule height .9ex width .8ex depth -.1ex}
\def\VLLA{\hbox to 25pt{\leftarrowfill}}
\def\VLRA{\hbox to 25pt{\rightarrowfill}}
\def\DU{\mathop{\bigcup}\limits^{.}}
\setbox2=\hbox to 25pt{\rightarrowfill}
\def\DRA{\vcenter{\copy2\nointerlineskip\copy2}}
\def\RANK{\hbox{\rm rank}\hskip 2pt}

\font\rsmall=cmr7 at 7truept \font\bsmall=cmbx7 at 7truept
\newfam\Slfam
\font\tenSl=cmti10 \font\neinSl=cmti9 \font\eightSl=cmti8
\font\sevenSl=cmti7 \textfont\Slfam=\tenSl
\scriptfont\Slfam=\eightSl \scriptscriptfont\Slfam=\sevenSl
\def\Sl{\fam\Slfam\tenSl}

\def\CODIM{\hbox{\rm codim}\hskip 2pt}
\def\CODIMS{\hbox{\rsmall codim}\hskip 2pt}
\def\SRA{\hskip 2pt\hbox{$\joinrel\mathrel\circ\joinrel\to$}}
\def\tbox{\hskip 1pt\frame{\vbox{\vbox{\hbox{\boldmath$\scriptstyle\times$}}}}\hskip 2pt}
\def\circvert{\vbox{\hbox to 8.9pt{$\mid$\hskip -3.6pt $\circ$}}}
\def\IM{\hbox{\rm im}\hskip 2pt}
\def\ES{\vbox{\hbox to 8.9pt{$\big/$\hskip -7.6pt $\bigcirc$\hfil}}}
\def\TR{\hbox{\rm tr}\hskip 2pt}
\def\TRS{\hbox{\rsmall tr}\hskip 2pt}
\def\DIV{\hbox{\rm div}\hskip 2pt}
\def\DIVS{\hbox{\rsmall div}\hskip 2pt}
\def\GRAD{\hbox{\rm grad}\hskip 2pt}
\def\GRADS{\hbox{\rsmall grad}\hskip 2pt}
\def\bull{\vrule height .9ex width .8ex depth -.1ex}
\def\VLLA{\hbox to 25pt{\leftarrowfill}}
\def\VLRA{\hbox to 25pt{\rightarrowfill}}
\def\DU{\mathop{\bigcup}\limits^{.}}
\setbox2=\hbox to 25pt{\rightarrowfill}
\def\DRA{\vcenter{\copy2\nointerlineskip\copy2}}
\def\RANK{\hbox{\rm rank}\hskip 2pt}
\def\ROT{\hbox{\rm rot}\hskip 2pt}
\def\ROTS{\hbox{\rsmall rot}\hskip 2pt}
\def\NOTSUBSET{\hskip 2pt\hbox{$\subset\hskip -9pt\backslash$}\hskip 2pt}
\def\SRA{\hskip 2pt\hbox{$\joinrel\mathrel\circ\joinrel\to$}}
\def\SLA{\hskip 2pt\hbox{$\leftarrow\joinrel\mathrel\circ$}}

\vskip -1cm

\title[Quantum Extended Crystal Super PDE's]{QUANTUM EXTENDED CRYSTAL SUPER PDE's}
\author{Agostino Pr\'astaro}
\maketitle
\vspace{-.5cm}

{\footnotesize
\begin{center}
Department SBAI - Mathematics, University of Rome ''La Sapienza'', Via A.Scarpa 16,
00161 Rome, Italy. \\
E-mail: {\tt agostino.prastaro@uniroma1.it}
\end{center}
\vspace{.5cm}

\vspace{.5cm} {\bsmall ABSTRACT.} We generalize our geometric theory on extended crystal PDE's and their stability, to the category
$\mathfrak{Q}_S$ of quantum supermanifolds. By using algebraic topologic techniques, obstructions to the
existence of global quantum smooth solutions for such equations are
obtained. Applications are given to encode quantum dynamics of nuclear nuclides,
identified with graviton-quark-gluon plasmas, and study their stability. We prove that such quantum dynamical
systems are encoded by suitable quantum extended crystal Yang-Mills super PDE's. In this way stable
nuclear-charged plasmas and nuclides are characterized as suitable stable quantum solutions of such quantum Yang-Mills
super PDE's. An existence theorem of local and global solutions with mass-gap, is given for quantum super Yang-Mills
PDE's, $\widehat{(YM)}$, by identifying a suitable constraint, $\widehat{(Higgs)}\subset \widehat{(YM)}$,
{\em Higgs quantum super PDE}, bounded by a quantum super partial differential relation
$\widehat{(Goldstone)}\subset \widehat{(YM)}$, {\em quantum Goldstone-boundary}. A global solution
$V\subset\widehat{(YM)}$, crossing the
quantum Goldstone-boundary acquires (or loses) mass. Stability properties of such solutions are
characterized.\footnote{Work partially supported by Italian grants MIUR ''PDE's Geometry and
Applications''.}} \vspace{.5cm}

\vspace{.5cm}
{\bsmall AMS (MOS) MS CLASSIFICATION. 57R90, 53C99,
81Q99.}

{\rsmall KEY WORDS AND PHRASES. Integral bordisms in quantum super PDE's.
Existence of local and global solutions in quantum super PDE's.
Conservation laws. Quantum (super)gravity. (Un)stable quantum super PDE's.
(Un)stable quantum solutions. (Un)stable quantum extended crystal super PDE's. Quantum singular super PDE's.}

\vskip 0.5cm

\section{\bf Introduction}

The mathematical heritage of the last century is essentially that
Physics is Geometry and nothing else. This is, in fact, the main
issue by the Einstein's General Relativity theory \cite{AEIN},
(previously supported also by the Maxwell's theory of
electromagnetism \cite{MAXW1, MAXW2}). Unfortunately this message
was not well understood at the quantum level! The principal
motivation was, we believe, that Mathematics was not ready to extend
such a philosophy in the formulation of a geometric quantum theory.
In fact, at the beginning of the last century the Ricci-Curbastro's
tensor calculus \cite{RICCURB1, RICCURB2, SCHOUT} was just enough
developed to allow to Einstein his formulation of general
gravitation. Instead, in order to formulate a ''quantum general
relativity'', it was necessary yet to build a new geometric theory
of quantum PDE's!

At the beginning of last century, however, was just well understood
that the concept of ''mass'' is synonymous of concentred energy. In
fact, the General Relativity Theory proved that big masses are able to deform space-time geometry.

Nowadays we can state that also at quantum level, high concentration of energy modifies geometry and
produces noncommutative geometry. Thus the mass-energy, is nothing
else that a property of the geometry involved to describe
''particles''. This becomes conjectured also after the first middle
of the last century, thanks to the famous formula $E=mc^2$ and
nuclear energy production experiments. To this purpose, let us
recall the well known J. A. Wheeler's slogans: ''{\em mass without
mass, charge without charge, field without field}'' and his
pioneering works on the ''geometrodynamics''. (See, e.g.,
Refs.\cite{WHE1, WHE2, WHE3, MIS-THO-WHE}.)

Our PDE's Algebraic Topology, developed in the category of
(non)commutative manifolds, aims to follow this philosophy, building
a new mathematics just able to allow a full geometrization of
Physics.

In some previous works we have characterized PDE's as extended
crystals, in the sense that their integral bordism groups can be
considered as extensions of suitable crystallographic subgroups. For
such structures a geometric formulation of stability theory has been
developed also.  (See Refs.\cite{PRA24, PRA25, PRA26, PRA27, PRA28, PRA29}.)

Aim of the present paper is to extended to
quantum (super) PDE's, i.e., PDE's built in the category
$\mathfrak{Q}_S$ of quantum supermanifolds, above formulations, and to
study in some details quantum
supergravity Yang-Mills PDE's (quantum SG-Yang-Mills PDE's). This
type of equations have been previously introduced by us in some
recent works and appears very useful to encode quantum dynamics
unifying, just at quantum level, gravity with the other fundamental
forces of Nature, i.e., electromagnetic, weak and strong
forces.\cite{PRA9, PRA10, PRA13, PRA14, PRA15, PRA16, PRA20, PRA21,
PRA22, PRA23, PRA30, PRA31, PRA32, PRA37, PRA38}. These equations extend,
at quantum level, some superclassical ones, well known in literature
about supergravity. (See, e.g., Refs.\cite{DEW, G-S-W, VNI, W-B,
WES, WIT}.) In fact supergravity, as has been usually considered, is
a classical field theory, that, in some sense comes from a
generalization of Charles Ehresmann and \`Elie Cartan's differential
geometry \cite{SHA}.\footnote{It is well known that the mathematical
foundation of all gauge theories is to ascribe to C.Ehresmann, when
he was a E. Cartan's student \cite{EHR}.} Then classical
supergravity requires to be quantized. But in this way one discards
nonlinear phenomena. In fact this quantization is obtained by means
of so-called quantum propagators, that are just associated to
linearizations of classical PDE's. Our formulation, instead, works
directly on noncommutative manifolds ({\em quantum supermanifolds}),
and the quantization is not more necessary. In fact, whether it is
performed in this noncommutative framework, it can bee seen as a
linear approximation of a more general nonlinear integration. In
some previous papers this important aspect has been carefully
proved. (See Refs.\cite{PRA23, PRA32}.)

This paper, after Introduction, splits in two more sections.
Section 2. Here we characterize quantum super PDE's like extended
crystals, in the sense that their integral bordism groups can be
considered as extensions of crystallographic subgroups. This
approach generalizes our previous one for commutative PDE's, and
allows us to identify an algebraic topologic obstruction to the
existence of global smooth solutions for PDE's in the category
$\mathfrak{Q}_S$. Furthermore, for such solutions we study their
stability properties from a geometric point of view. Section 3. Here we
consider ''quantum gravity'' in the category $\mathfrak{Q}_S$, and
encoded by suitable quantum Yang-Mills equations ({\em quantum
SG-Yang-Mills PDE's}), say $\widehat{(YM)}$. In this way we are able
to characterize quantum (super)gravity like a secondary object,
associated to some geometric fundamental objects (fields), solutions
of $\widehat{(YM)}$. Then mass properties of such solutions are
directly pointed-out, without the necessity to assume symmetry
breaking Higgs-mechanisms. However, we recognize a constraint in $\widehat{(YM)}$, that gives a pure quantum geometrodynamic
mechanism able to justify mass acquisition (or loss) to a quantum solution of $\widehat{(YM)}$.
Furthermore, nuclear particles and
nuclides can be seen as suitable $p$-chain solutions of
$\widehat{(YM)}$, and their energy-thermodynamic contents and stability
properties characterized.

The main results of this paper are the following. Theorem
\ref{crystal-structure-quantum-super-pdes} characterizes the
crystal structure of quantum super PDE's. Corollary \ref{main7}
identifies the algebraic topologic obstruction to the existence of
global smooth solutions of PDE's in the category $\mathfrak{Q}_S$.
Theorem \ref{criteria-fun-stab} and Theorem
\ref{finite-stable-extended-crystal-PDE} characterize the
stability of quantum super PDE's and their solutions. Theorem
\ref{criterion-average-asymptotic-stability} gives a criterion
to average stability. Theorem \ref{dynamic-equation} and Theorem
\ref{sg-ym-pde-cartan-geometry} encode dynamic
for quantum (super)gravity and characterize the quantum Cartan
geometry induced by solutions of quantum SG-Yang-Mills PDE's,
shortly denoted by $\widehat{(YM)}$. Theorem
\ref{quantum-Levi-Civita-Higgs-ym} gives a criterion, founded
on the quantum Higgs-symmetry breaking mechanism, to recognize
solutions of $\widehat{(YM)}$ whose quantum Levi-Civita connections
induce zero covariant derivative on the corresponding quantum
metric. Theorem \ref{quantum-crystal-structure-ym} and Theorem
\ref{quantum-crystal-structure-ym[i]} characterize the quantum
crystal structure of $\widehat{(YM)}$. Theorem
\ref{local-mass-formula-theorem} gives a local mass-formula for
solutions of $\widehat{(YM)}$. Theorem
\ref{observed-objects-and-splitting-formulas-theorem}
identifies important quantum observed fields by means of a quantum
relativistic frame. Theorem \ref{stability-properties-ym}
characterizes the stability properties of $\widehat{(YM)}$ and its
solutions. Theorem \ref{observed-quantum-system-thermodynamics}
identifies thermodynamic functions and thermodynamic equations
associated to observed solutions of $\widehat{(YM)}$. Theorem \ref{existence-mass-gap-solutions}, and
Corollary \ref{goldstone-piece-characterization} prove existence of a formally integrable and
completely integrable quantum super PDE, $\widehat{(Higgs)}\subset\widehat{(YM)}$, where for
any point, initial condition, there exist solutions with mass-gap.\footnote{Here and in the following, talking about
quantum super PDE's, we simply say formally
integrable, (resp. completely integrable), instead of formally quantum superintegrable,
(resp. completely quantum superintegrable). (Compare with previous works on the same subjects.)} We call such a constraint {\em Higgs quantum super PDE}.
A global solution, $V\subset\widehat{(YM)}$, crossing the boundary, (denoted $\widehat{(Goldstone)}$),
of $\widehat{(Higgs)}$ in $\widehat{(YM)}$, {\em quantum Goldstone-boundary}, acquires (resp. loses) mass going inside,
(resp. outside), $\widehat{(Higgs)}$. So that the quantum Goldstone-boundary
can be considered as the quantum integral situs for mass-creation, or, vice versa, mass-destruction, according that one
 considers the solution going inside $\widehat{(Higgs)}$, or outgoing from  $\widehat{(Higgs)}$. The stability of
 such solutions are studied and identified the corresponding stabilized quantum extended crystal super PDE,
where all the smooth solutions have mass-gap and are stable at finite-times.

\section{\bf EXTENDED CRYSTAL SUPER PDE's STABILITY IN {\boldmath$\mathfrak{Q}_S$}}
\vskip 0.5cm

In this section we extend to quantum super PDE's our previous
results on the algebraic topological crystal characterization of
commutative PDE's.\footnote{For general informations on Algebraic (Co)homology see, e.g., the following
Refs.\cite{ABE, ATI, BOARD, BOU, C-R-R, GOL-GUIL, HIR, LEV1, LEV2,
M-M, MATH1, MATH2, MATH3, MATH4, MATH5, MCC2, MCCO, M-S, QUIL, ROT,
RUDY, STO, SULL, SWE, SWI, THO1, THO2, THO3, WAL1, WAL2}. For
general informations on crystallography, as used in this paper, see,
e.g., Refs.\cite{FED, HAH, PLE-PEST, RAGH, SCHO, SCHW, SUN}. For the
geometric theory of PDE's, see, e.g., Refs.\cite{B-C-G-G-G, GOL,
GOL-SPE, GRO, KRA-LYC-VIN, KUR1, KUR2, KUR3, KUR4, LIB, L-P, PAL,
PRA2, PRA3, PRA4, PRA5, PRA-REGGE, SPE1, SPE2, TAK}. For the
Algebraic Topology of PDE's, super PDE's, quantum PDE's and quantum
super PDE's see Refs.\cite{PRA6, PRA7, PRA8, PRA9, PRA10, PRA11,
PRA12, PRA13, PRA14, PRA15, PRA16, PRA16, PRA17, PRA18, PRA19,
PRA19, PRA20, PRA21, PRA22, PRA23, PRA24, PRA25, PRA26, PRA27,
PRA28, PRA29, PRA30, PRA31, PRA32, PRA33, PRA34, PRA35, PRA36, PRA37, PRA38}. See also the following
Refs.\cite{AG-PRA1, AG-PRA2, AG-PRA3, AG-PRA4, PRA-RAS1, PRA-RAS2,
PRA-RAS3}, where interesting applications of the PDE's Algebraic
Topology are given.} Furthermore, we shall consider the stability of
quantum super PDE's in the framework of the geometric theory of
quantum super PDE's. We will follow the line just drawn in some our
previous papers on this subject for commutative PDE's, where we have
interpreted stability of PDE's on the ground of their integral
bordism groups and related the quantum bordism of PDE's to Ulam
stability too.

Here and in the following we shall denote the boundary $\partial V$
of a compact quantum supermanifold $V$, of dimension $m|n$, with
respect to a quantum superalgebra $A$, split in the form $\partial
V=N_0\bigcup P\bigcup N_1$, where $N_0$ and $N_1$ are two disjoint
$(m-1|n-1)$-dimensional quantum sub-supermanifolds of $V$, that are
not necessarily closed, and $P$ is another $(m-1|n-1)$-dimensional
quantum sub-supermanifold of $V$. For example, if $V=\hat
D{}^{m|n}\times \hat D{}^{1|1}$, where $\hat D{}^{r|s}\subset \hat
S{}^{r|s}$ is the $(r|s)$-dimensional quantum superdisk, contained
in the $(r|s)$-dimensional quantum supersphere, one has that $\dim
V=(m+1|n+1)$ and  $N_0=\hat D{}^{m|n}\times\{0\}$, $N_1=\hat
D{}^{m|n}\times\{1\}$, $P=\partial \hat D{}^{m|n}\times \hat
D{}^{1|1}\cong \hat S{}^{m-1|n-1}\times \hat
D{}^{1|1}$.\footnote{Recall that a {\em$(m|n)$-dimensional quantum
supersphere}, $\hat S^{m|n}$, over a quantum superalgebra $A$, is
the Alexandrov compactification of $A^{m|n}$, i.e., $\hat
S^{m|n}=A^{m|n}\bigcup\{\infty\}$. A {\em$(m|n)$-dimensional quantum
superdisk} over a quantum superalgebra $A$, is a connected compact
sub-supermanifold  $\hat D^{m|n}\subset \hat S^{m|n}$, of dimension
$m|n$ over $A$, such that $\partial\hat D^{m|n}\cong \hat
S^{m-1|n-1}$, i.e., with boundary $\partial\hat D^{m|n}$
diffeomorphic to $\hat S^{m-1|n-1}$. For details on such quantum
supermanifolds see \cite{PRA31}.} Therefore $\dim N_0=\dim N_1=\dim
P=(m|n)$. Note that as $\hat D{}^{m|n}\times \hat D{}^{1|1}=\hat
D{}^{m+1|n+1}$, we can also write $\partial V=\partial\hat
D{}^{m+1|n+1}=\hat S{}^{m|n}=\hat S{}^{m-1|n-1}\times\hat
S{}^{1|1}$, since

\begin{equation}
\begin{array}{ll}
\partial V&=\partial(\hat D{}^{m|n}\times \hat D{}^{1|1})
      =(\partial \hat D{}^{m|n})\times \hat D{}^{1|1}\bigcup \hat
      D{}^{m|n}\times\partial\hat D{}^{1|1}\\
      &=\hat S{}^{m-1|n-1}\times \hat D{}^{1|1}\bigcup \hat
      D{}^{m|n}\times\hat S{}^{0|0}.\\
      \end{array}
    \end{equation}
Therefore, $\partial V$ is obtained by means of the quantum
surgering removing $\hat S{}^{m-1|n-1}\times \hat
D{}^{1|1}\subset\hat S{}^{m|n}$. (For details on quantum surgering
see \cite{PRA31, PRA32}.) Of course if $V=\hat S{}^{m|n}\times \hat
D{}^{1|1}$, then $P=\emptyset$, hence $\partial V=\hat
S{}^{m|n}\times\{0\}\DU S{}^{m|n}\times\{1\}$.

This example shows that if $V$ is a solution of a quantum super PDE,
then it can be obtained by propagating an initial Cauchy
hypersurface $X\subset V$, $\dim X=(m|n)$, by means of an integrable
full quantum vector field $\zeta:V\to \widehat{TV}\equiv
Hom_Z(A;TV)$, $\partial\phi=\zeta$, where $\phi:A\times V\to V$.
(See Fig.1.)
\begin{figure}[h]
\centerline{\includegraphics[width=5cm]{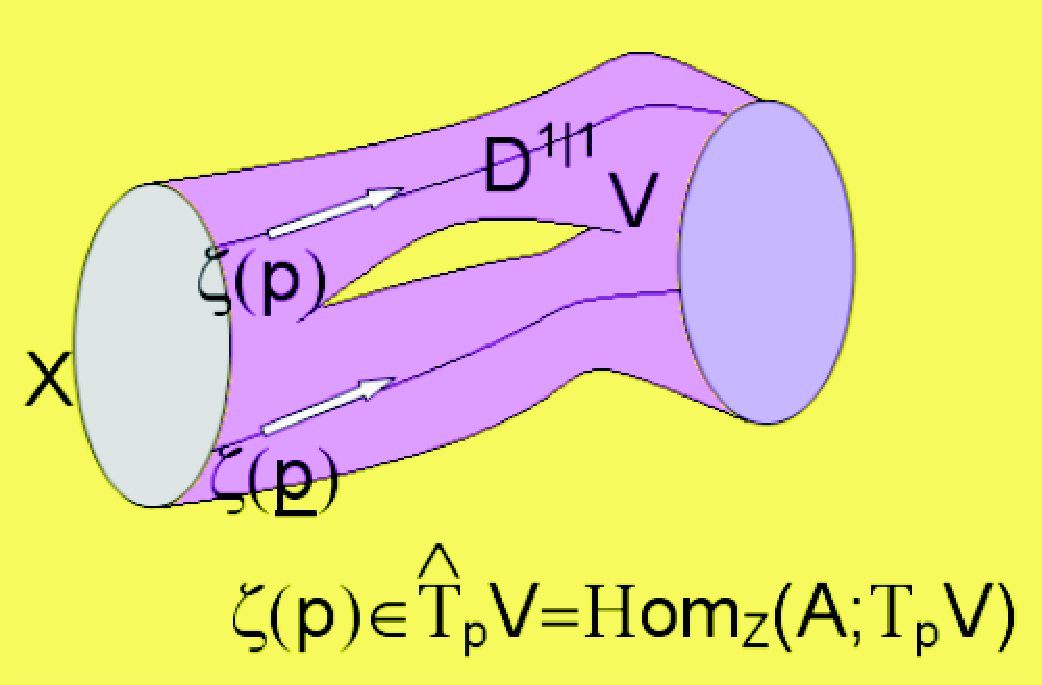}} \caption{Quantum
solution $V$, of dimension $(m+1|n+1)$ over a quantum superalgebra
$A$, propagating $X$, $\dim
X=(m|n)$.\label{quantum-supermanifold-solution}}
\end{figure}

Let us emphasize also, that in some cases solutions can be obtained
also by flows of integrable vector fields $\zeta:V\to TV$, i.e.,
$\zeta=\partial\psi$, with $\psi:\mathbb{R}\times V\to V$. For
example if $V=\hat D{}^{m|n}\times I$, with
$I\equiv[0,1]\subset\mathbb{R}$, then $\dim V=m+1|n$, and $\partial
V=M_0\bigcup P\bigcup M_1$, with $P=\hat S{}^{m-1|n-1}\times I$,
$M_0=\hat D{}^{m|n}\times\{0\}$, $M_1=\hat D{}^{m|n}\times\{1\}$. So
we get $\dim P=m|n-1$, $\dim M_0=\dim M_1=m|n$. Therefore $\partial
V$ has some components ($M_0$ and $M_1$) that have only the even
dimension dropped by $1$, with respect to $V$, and other one ($P$)
where also the odd dimension drops by $1$.

Let us also recall that with the term {\em quantum solutions} we
mean integral bordisms relating Cauchy quantum hypersurfaces of
$\hat E_{k+s}$, contained in $J^{k+s}_{m|n}(W)$, but not necessarily
contained into $\hat E_{k+s}$. (For details see refs.\cite{PRA21,
PRA22, PRA23, PRA31, PRA32}.)

\begin{definition}
We say that a quantum super PDE $\hat E_k\subset \hat J^k_{m|n}(W)$
is a {\em quantum extended $0$-crystal super PDE}, if its weak
integral bordism group $\Omega^{\hat E_k}_{m-1|n-1,w}$ is zero.
\end{definition}

\begin{theorem}{\em(Criterion to recognize quantum extended $0$-crystal super PDE's).}\label{main4}
Let $\hat E_k\subset \hat J^k_{m|n}(W)$ be a formally
integrable and completely integrable quantum super PDE
such that $W$ is contractible. If $m-1\not=0$ and $n-1\not=0$, then
$E_k$ is a quantum extended $0$-crystal super PDE.
\end{theorem}

\begin{proof}
In fact, one has the following isomorphisms, (see \cite{PRA22}):\footnote{Note that here solutions of $\hat E_k$ are in general considered $(m|n)$-dimensional integral quantum supermanifolds $V$ belonging to $C_{m|n}(\hat E_{k+h},A)$, $h\ge 0$, such that $TV\subset\mathbf{E}_{m|n}^{k+h}\subset T\hat E_{k+h}$, where $\mathbf{E}_{m|n}^{k+h}$ is the Cartan distribution of $\mathbf{E}_{m|n}^{k+h}$. We have denoted such spaces also with the symbol $\bar C_{m|n}(\hat E_{k+h},A)$. (See \cite{PRA22}.)}
\begin{equation}
    \begin{array}{ll}
     \Omega^{\hat E_k}_{m-1|n-1,w}& \cong H_{m-1|n-1}(W;A) \\
      & \cong \left(A_0\bigotimes_{\mathbb{K}}H_{m-1}(W;\mathbb{K})\right)\bigoplus
\left(A_1\bigotimes_{\mathbb{K}}H_{n-1}(W;\mathbb{K})\right).
    \end{array}
\end{equation}

Thus, when $W$ is contractible, and $m-1\not=0$, $n-1\not=0$, one
has $H_{m-1}(W;\mathbb{K})=H_{n-1}(W;\mathbb{K})=0$, hence we get
$\Omega^{\hat E_k}_{m-1|n-1,w}=0$.
\end{proof}

\begin{theorem}{\em(Crystal structure of quantum super PDE's).}\label{crystal-structure-quantum-super-pdes}
Let $\hat E_k\subset \hat J^k_{m|n}(W)$ be a formally
integrable and completely integrable quantum super
PDE. Then its integral bordism group $\Omega_{m-1|n-1}^{\hat E_k}$
is an extension of some crystallographic subgroup $G\triangleleft
G(d)$. We call $d$ the {\em crystal dimension} of $\hat E_k$ and
$G(d)$ its {\em crystal structure} or {\em crystal group}.
\end{theorem}

\begin{proof}
The proof generalizes and extends one given in \cite{PRA25} for PDE's in the category of commutative manifolds, and announced in \cite{PRA26} too. The first step is to
note that there is a relation between lower dimensions integral
bordisms in a commutative PDE. (In the following we shall assume that all PDE's are formally integrble and completely integrable.)

\begin{lemma}{\em(Relations between
lower dimensions integral bordisms in commutative
PDE's).}\label{commutative-PDE-lower-order-relation-bordism} Let
$E_k\subset J^k_n(W)$ be a PDE on the fiber bundle $\pi:W\to M$,
$\dim W=m+n$, $\dim M=n$. Let ${}^{S}C_p(E_k)$ be the set of all
compact $p$-dimensional admissible integral smooth manifolds of
$E_k$. The disjoint union gives an addition on ${}^{S}C_p(E_k)$ with
$\varnothing$ as the zero element. Let us consider the homomorphisms
$\partial_p:{}^{S}C_p(E_k)\to {}^{S}C_{p-1}(E_k)$ that associates to
any element $a\in{}^{S}C_p(E_k)$ its boundary $\partial
a=\partial_p(a)$. So we obtain the chain complex {\em(\ref{integral-smooth-bordism-chain-complex})}of
abelian groups {\em(integral smooth bordisms chain complex)}:
\begin{equation}\label{integral-smooth-bordism-chain-complex}
\xymatrix{{}^{S}C_{n}(E_k)\ar[r]^{\partial_n}&{}^{S}C_{n-1}(E_k)\ar[r]^{\partial_{n-1}}&{}^{S}C_{n-2}(E_k)\ar[r]^(.6){\partial_{n-2}}&
\cdots\ar[r]^(.4){\partial_1}&{}^{S}C_{0}(E_k).\\}
\end{equation}
Then the $p$-bordism groups $\Omega_p^{E_k}$, $0<p<n$, can be
represented by means of the homology of the chain complex
{\em(\ref{integral-smooth-bordism-chain-complex})}.
\end{lemma}
\begin{proof}
Let us denote by
$\left\{{}^{S}C_{\bullet}(E_k),\partial_\bullet\right\}$ the chain
complex in (\ref{integral-smooth-bordism-chain-complex}). Then, we
can build the exact commutative diagram (\ref{commutative-diagram-integral-smooth-bordism-chain-complex}).
\begin{equation}\label{commutative-diagram-integral-smooth-bordism-chain-complex}
\xymatrix{&&0\ar[d]&0\ar[d]&&\\
&0\ar[r]&{}^{S}B_{\bullet}(E_k)\ar[d]\ar[r]&{}^{S}Z_{\bullet}(E_k)\ar[d]\ar[r]&{}^{S}H_{\bullet}(E_k)\ar[r]&0\\
&&{}^{S}C_{\bullet}(E_k)\ar[d]\ar@{=}[r]&{}^{S}C_{\bullet}(E_k)\ar[d]&&\\
0\ar[r]&\Omega_\bullet^{E_k}\ar[r]&{}^{S}Bor_{\bullet}(E_k)\ar[d]\ar[r]&{}^{S}Cyc_{\bullet}(E_k)\ar[d]\ar[r]&0&\\
&&0&0&&\\}
\end{equation}
where

$$\scalebox{0.9}{$\left\{\begin{array}{l}
{}^{S}B_{\bullet} (E_k)=\ker(\partial|_{\bullet});\hskip
2pt {}^{S}Z_{\bullet} (E_k)=\IM(\partial_{\bullet});\\
{}^{S}H_{\bullet} (E_k)={}^{S}Z_{\bullet} (E_k)/{}^{S}B_{\bullet} (E_k),\\
b\in[a]\in {}^{S}Bor_{\bullet}(E_k)\Rightarrow a-b=\partial c;\hskip
2pt
                 c\in {}^{S}C_{\bullet}(E_k);\hskip 2pt
b\in[a]\in {}^{S}Cyc_{\bullet}(E_k)\Rightarrow \partial(a-b)=0;\\
b\in[a]\in \Omega_{\bullet}^{E_k}\Rightarrow
                 \left\{\begin{array}{l}
                          \partial a=\partial b=0\\
                          a-b=\partial c,\quad c\in {}^{S}C_{\bullet}(E_k)\\
                          \end{array}
                                \right\}.\\
\end{array}\right.$}$$
Then from
(\ref{commutative-diagram-integral-smooth-bordism-chain-complex}) it
follows directly that $\Omega_{p}^{E_k}\cong {}^{S}H_p(E_k)$,
$0<p<n$.
\end{proof}

\begin{lemma}{\em(Relations between integral bordisms groups in commutative PDEs).}\label{relation-between-integral-bordisms-groups-in-commutative-PDEs}
One has the following canonical isomorphism:
\begin{equation}\label{isomorphism-relation-between-integral-bordisms-groups-in-commutative-PDEs}
    \mathbb{Z}\bigotimes_{\Omega_\bullet^{E_k}}\mathbb{Z}{}^{S}Bor_{\bullet}(E_k)\cong\mathbb{Z}{}^{S}Cyc_{\bullet}(E_k).
\end{equation}
\end{lemma}
\begin{proof}
Follows directly from the extension of groups given at the bottom of
the commutative exact diagram
(\ref{commutative-diagram-integral-smooth-bordism-chain-complex})
and some properties between extension of groups. (See, e.g.,
\cite{PRA9}.)
\end{proof}

\begin{lemma}{\em(Relations between integral bordisms groups in commutative PDEs-2).}\label{relation-between-integral-bordisms-groups-in-commutative-PDEs-2}
If $H^2({}^{S}Cyc_{\bullet}(E_k),\Omega_\bullet^{E_k})=0$ one has
the following canonical isomorphism:
\begin{equation}\label{isomorphism-relation-between-integral-bordisms-groups-in-commutative-PDEs-2}
    {}^{S}Bor_{\bullet}(E_k)\cong\Omega_\bullet^{E_k}\times{}^{S}Cyc_{\bullet}(E_k).
\end{equation}
\end{lemma}

\begin{proof}
Follows directly from the extension of groups given at the bottom of
the commutative exact diagram
(\ref{commutative-diagram-integral-smooth-bordism-chain-complex})
and some properties between extension of groups. (See, e.g.,
\cite{PRA9}.)
\end{proof}

\begin{lemma}{\em(Integral ringoid of PDE).}\label{integral-ringoid-pde}
A {\em ringoid} is a structure $(A,+,\cdot)$, where $A$ is a set and
$+$ is a binary operation such that $(A,+)$ is an abelian additive
group with zero $0\in A$; $\cdot$ is a partially binary operation,
i.e., it is defined only for some couples $(a,b)\in A\times A$, such
that it is associative, and distributive with respect to $+$, i.e.,
if $a\cdot b$ and $a\cdot c$ are defined, then it is defined also
$a\cdot(b+c)=a\cdot b+a\cdot c$. A {\em graded ringoid} is a set
$A=\bigoplus_nA_n$, where each $A_n$ is an abelian additive group
and there is a partial binary operation $\cdot$, associative, and
distributive with respect to $+$, such that if $a\in A_n$, $b\in
A_m$, then $a\cdot b\in A_{m+n}$, whenever it is defined.

Let $E_k\subset J^k_n(W)$ be a PDE, with $\pi:W\to M$ a fiber
bundle, $\dim W=m+n$, $\dim M=n$. Then the integral bordism groups
$\Omega_p^{E_k}$, $0\le p\le n-1$, identify a graded ringoid
$\Omega_\bullet^{E_k}$, that we call {\em integral ringoid} of
$E_k$, that is an extension of a graded ringoid contained in the
nonoriented bordism ring $\Omega_\bullet$. One has the commutative diagram {\em(\ref{ringoid-hmomorphisms})}.
\begin{equation}\label{ringoid-hmomorphisms}
\scalebox{0.8}{$\xymatrix{&&0\ar[d]&&\\
0&{Hom_{ringoid}(\overline{K}^{E_k}_\bullet;\mathbb{R})}\ar[l]&{Hom_{ringoid}(\Omega^{E_k}_\bullet;\mathbb{R})}\ar[l]\ar[d]&
{Hom_{ringoid}({}^{(n-1)}\Omega_\bullet;\mathbb{R})}\ar[dl]\ar[l]&0\ar[l]\\
&&{\mathbf{H}_\bullet(E_k)}&&}$}
\end{equation}
where $\mathbf{H}_\bullet(E_k)\equiv\bigoplus_{0\le p\le
n-1}\mathbf{H}_p(E_k)$, that allows us to represent differential
$p$-conservation laws of order $k$ by means of ringoid homomorphisms
$\Omega^{E_k}_\bullet\to\mathbb{R}$.
\end{lemma}

\begin{proof}
Set
\begin{equation}\label{integral-ringoid}
    \Omega_\bullet^{E_k}\equiv\bigoplus_{0\le p\le
    n-1}\Omega_p^{E_k}.
\end{equation}
Each $\Omega_p^{E_k}$ are additive abelian groups, with addition
induced by disjoint union, $\sqcup$. Furthermore, there is a natural
product induced by the cartesian product, i.e.,
$[X_1]\cdot[X_2]=[X_1\times X_2]\in\Omega_{p_1+p_2}^{E_k}$, for
$[X_i]\in\Omega_{p_i}^{E_k}$, $0\le p_i\le n-1$, $i=1,2$, $0\le
p_1+p_2\le n-1$. This product it is not always defined for all couples $(X_1,X_2)$ of
closed admissible integral manifolds $X_i$,  $i=1,2$, but only for
ones such that $X_1\times X_2$ is a closed integral admissible
manifold. Therefore $\Omega_\bullet^{E_k}$ is a graded ringoid. Set
${}^{(n-1)}\Omega_\bullet\equiv\bigoplus_{0\le p\le
    n-1}\Omega_p$. It has in a natural way a graded ringoid structure, with
    respect the same operations with respect to which
    $\Omega_\bullet$ is a graded ring. Furthermore, for any $0\le p\le
    n-1$, one has the exact sequence (\ref{exact-sequence-integral-bordism-bordism-group}), (see proof of
    Theorem 4.2 in \cite{PRA12}).
\begin{equation}\label{exact-sequence-integral-bordism-bordism-group}
\xymatrix{0\ar[r]&\overline{K}_p^{E_k}\ar[r]&\Omega_p^{E_k}\ar[r]&\Omega_p\ar[r]&0\\}
\end{equation}
As a by-product one has also the following exact commutative
diagram:
\begin{equation}
\xymatrix{&&&0\ar[d]&\\
0\ar[r]&\overline{K}_\bullet^{E_k}\ar[r]&\Omega_\bullet^{E_k}\ar[r]&{}^{(n-1)}\Omega_\bullet\ar[r]\ar[d]&0\\
&&&\Omega_\bullet\\}
\end{equation}
where $\overline{K}_\bullet^{E_k}\equiv\bigoplus_{0\le p\le
    n-1}\overline{K}_p^{E_k}$.

A full $p$-conservation law is any function
$f:\Omega_p^{E_k}\to\mathbb{R}$, $0\le p\le n-1$. These, identify
elements of $\mathbf{H}_\bullet(E_k)\equiv\bigoplus_{0\le p\le
n-1}\mathbf{H}_p(E_k)$ in a natural way. In
$\mathbf{H}_\bullet(E_k)$ are contained also ones identified by
means of differential conservation laws of order $k$, belonging to the following quotient space, ({\em space of characteristic integral $q$-forms on $E_k$}): $\mathfrak{I}(E_k)^\bullet\equiv\bigoplus_{0\le q\le
n-1}\mathfrak{I}(E_k)^q$, with
\begin{equation}\label{p-conservation-laws-order-k}
    {\frak I}(E_k)^{q}
\equiv\frac{\Omega^q(E_k)\cap
d^{-1}(C\Omega^{q+1}(E_k))}{d\Omega^{q-1}(E_k)\oplus\{C\Omega^q(E_k)\cap
d^{-1}(C\Omega^{q+1}(E_k))\}}.
\end{equation}
Here, $\Omega^q(E_k)$ is the space of smooth $q$-differential forms
on $E_k$ and $C\Omega^q(E_k)$ is the space of Cartan $q$-forms on
$E_k$, that are zero on the Cartan distribution $\mathbf{E}_k$ of
$E_k$. Therefore, $\beta\in C\Omega^q(E_k)$ iff
$\beta(\zeta_1,\cdots,\zeta_q)=0$, for all $\zeta_i\in
C^\infty(\mathbf{E}_k)$.\footnote{The {\em space of conservation laws} of $E_k$, ${\frak C}ons(E_k)$, can be
identified with the spectral term $E_1^{0,n-1}$ of the spectral
sequence associated to the filtration induced in the graded algebra
$\Omega^\bullet(E_\infty)\equiv\oplus_{q\ge 0}\Omega^q(E_\infty)$,
by the subspaces $C\Omega^q(E_\infty)\subset\Omega^q(E_\infty)$.
(For abuse of language we shall call ''conservation laws of
$k$-order'', characteristic integral $(n-1)$-forms too. Note that
$C\Omega^0(E_k)=0$. See also Refs.\cite{PRA9, PRA11, PRA13, PRA20}.)} Any
$[\alpha]\in\mathfrak{I}(E_k)^\bullet$ identifies a ringoid
homomorphism $f[\alpha]:\Omega_\bullet^{E_k}\to\mathbb{R}$. More
precisely one has
$f[\alpha]([X_1]+[X_2])=<\alpha,X_1>+<\alpha,X_2>$, for $[\alpha]\in
{\frak I}(E_k)^{p}$, $[X_1],[X_2]\in \Omega_p^{E_k}$, and
$f[\alpha]([X_1]\cdot[X_2])=<\alpha_1,X_1><\alpha_2,X_2>$, for
$[\alpha]=[\alpha_1]+[\alpha_2]\in {\frak I}(E_k)^{p_1}\oplus {\frak
I}(E_k)^{p_2}$, $[X_1]\in \Omega_{p_1}^{E_k}$, $[X_2]\in
\Omega_{p_2}^{E_k}$.
\end{proof}

The next step is to extend above results to PDE's in the category
$\mathfrak{Q}_S$ of quantum supermanifolds.

\begin{lemma}{\em(Relations between lower order integral bigraded-bordisms in quantum super PDE's).}\label{quantum-PDE-lower-order-relation-bigraded-bordism}
Let $\hat E_k\subset \hat J^k_{m|n}(W)$ be a quantum super PDE on
the fiber bundle $\pi:W\to M$, $\dim_B W=(m|n,r|s)$, $\dim_A M=m|n$,
$B=A\times E$, $E$ a quantum superalgebra that is also a $Z$-module,
with $Z=Z(A)$ the center of $A$, assumed Noetherian. Let ${}^{S}C_{p|q}(\hat E_k)$ be
the set of all compact $p|q$-dimensional, (with respect to $A$),
admissible integral smooth manifolds of $\hat E_k$, $0\le p\le m$,
$0\le q\le n$. The disjoint union gives an addition on
${}^{S}C_{p|q}(\hat E_k)$ with $\varnothing$ as the zero element. Let
us consider the homomorphisms $\partial_{p|q}:{}^{S}C_{p|q}(\hat
E_k)\to {}^{S}C_{p-1|q-1}(\hat E_k)$ that associates to any element
$a\in{}^{S}C_{p|q}(\hat E_k)$ its boundary $\partial
a=\partial_{p|q}(a)$. So we obtain the chain complex {\em(\ref{integral-smooth-bigraded-bordism-chain-complex})} of
abelian groups {\em(integral smooth bigraded-bordisms chain
complex)} of $\hat E_k\subset \hat J^k_{m|n}(W)$.
\begin{equation}\label{integral-smooth-bigraded-bordism-chain-complex}
\begin{array}{l}
\xymatrix@C=45pt{{}^{S}C_{m|n}(\hat E_k)\ar[r]^(0.45){\partial_{m|n}}&
{}^{S}C_{m-1|n-1}(\hat E_k)\ar[r]^(0.5){\partial_{m-1|n-1}}&{}^{S}C_{m-2|n-2}(\hat E_k)\ar[r]^(0.6){\partial_{m-2|n-2}}&\cdots}\\
\xymatrix@C=40pt{\cdots\ar[r]^(0.3){\partial_{r|r}}&{}^{S}C_{m-r|n-r}(\hat E_k)}\\
\end{array}
\end{equation}
where $r=\min\{m,n\}$. Then the $p|q$-integral bordism groups
$\Omega_{p|q}^{\hat E_k}$, $(m-r)<p<m$, $(n-r)<q<n$, can be
represented by means of the homology of the chain complex
{\em(\ref{integral-smooth-bigraded-bordism-chain-complex})}.

One has the following canonical isomorphism:
\begin{equation}\label{isomorphism-relation-between-integral-bigraded-bordisms-groups-in-quantum-PDEs}
    \mathbb{Z}\bigotimes_{\Omega_{\bullet|\bullet}^{\hat E_k}}\mathbb{Z}{}^{S}Bor_{\bullet|\bullet}(\hat E_k)
    \cong\mathbb{Z}{}^{S}Cyc_{\bullet|\bullet}(\hat E_k).
\end{equation}
Furthermore, if $H^2({}^{S}Cyc_{\bullet|\bullet}(\hat
E_k),\Omega_{\bullet|\bullet}^{\hat E_k})=0$ one has the following
canonical isomorphism:
\begin{equation}\label{isomorphism-relation-between-integral-bigraded-bordisms-groups-in-quantum-PDEs-2}
    {}^{S}Bor_{\bullet|\bullet}(\hat E_k)\cong\Omega_{\bullet|\bullet}^{\hat E_k}\times{}^{S}Cyc_{\bullet|\bullet}(\hat E_k).
\end{equation}
\end{lemma}

\begin{proof}
The proof can be conduced similarly to the ones for Lemma
\ref{commutative-PDE-lower-order-relation-bordism}, Lemma
\ref{relation-between-integral-bordisms-groups-in-commutative-PDEs}
and Lemma
\ref{relation-between-integral-bordisms-groups-in-commutative-PDEs-2}.
\end{proof}

Similarly we can prove the following lemma concerning the total
analogue of the complex (\ref{integral-smooth-bigraded-bordism-chain-complex}) too.

\begin{lemma}{\em(Relations between lower order integral total-bordisms in quantum super PDE's).}\label{quantum-PDE-lower-order-relation-total-bordism}
Let $\hat E_k\subset \hat J^k_{m|n}(W)$ be a quantum super PDE on
the fiber bundle $\pi:W\to M$, $\dim_B W=(m|n,r|s)$, $\dim_A M=m|n$,
$B=A\times E$, $E$ a quantum superalgebra that is also a $Z$-module,
with $Z=Z(A)$ the center of $A$. Let ${}^{S}C_{p}(\hat E_k)$, $0\le
p\le m+n$, be the set of all compact $u|v$-dimensional, (with
respect to $A$), admissible integral quantum smooth manifolds of $\hat E_k$,
such that $u+v=p$. The disjoint union gives an addition on
${}^{S}C_{p}(\hat E_k)$ with $\varnothing$ as the zero element. Thus
we can write
\begin{equation}\label{total-smooth-integral-bordisms-quantum-pde}
{}^{S}C_{p}(\hat E_k)=\bigoplus_{u,v; u+v=p}{}^{S}C_{u|v}(\hat
E_k)={}^{Tot,S}C_{p}(\hat E_k).
\end{equation}

Let us consider the homomorphisms $\partial_{p}:{}^{S}C_{p}(\hat
E_k)\to {}^{S}C_{p-1}(\hat E_k)$ that associates to any element
$a\in{}^{S}C_{p}(\hat E_k)$ its boundary $\partial
a=\partial_{p}(a)$, i.e., one has:
\begin{equation}\label{total-smooth-integral-bordisms-quantum-pde-morphisms}
\begin{array}{ll}
  \partial_pa&=\partial_p(a_{p|0},a_{p-1|1},a_{p-2|2},\cdots,a_{0|p})\\
  &=(\partial_{p|0}a_{p|0},\partial_{p-1|1}a_{p-1|1},\partial_{p-2|2}a_{p-2|2},\cdots,\partial_{0|p}a_{0|p})\\
  &\in\bigoplus_{u,v; u+v=p-1}{}^{S}C_{u|v}(\hat
E_k)={}^{S}C_{p-1}(\hat E_k).
\end{array}
\end{equation}
One has $\partial_{p-1}\circ\partial_p=0$. So we get the chain complex {\em(\ref{integral-smooth-total-bordism-chain-complex})} of abelian groups {\em(integral smooth
bigraded-bordisms chain complex)} of $\hat E_k\subset \hat
J^k_{m|n}(W)$.
\begin{equation}\label{integral-smooth-total-bordism-chain-complex}
\xymatrix{{}^{S}C_{n}(\hat
E_k)\ar[r]^{\partial_{n}}&{}^{S}C_{n-1}(\hat
E_k)\ar[r]^(0.5){\partial_{n-1}}&{}^{S}C_{n-2}(\hat
E_k)\ar[r]^(0.6){\partial_{n-2}}&
\cdots\ar[r]^(0.3){\partial_{1}}&{}^{S}C_{0}(\hat E_k).\\}
\end{equation}
Then the $p$-integral total bordism groups $\Omega_{p}^{\hat E_k}$,
$0<p<m+n$, can be represented by means of the homology of the chain
complex {\em(\ref{integral-smooth-total-bordism-chain-complex})}.

One has the following canonical isomorphism:
\begin{equation}\label{isomorphism-relation-between-integral-total-bordisms-groups-in-quantum-PDEs}
    \mathbb{Z}\bigotimes_{\Omega_{\bullet}^{\hat E_k}}\mathbb{Z}{}^{S}Bor_{\bullet}(\hat E_k)
    \cong\mathbb{Z}{}^{S}Cyc_{\bullet}(\hat E_k).
\end{equation}
Furthermore, if $H^2({}^{S}Cyc_{\bullet}(\hat
E_k),\Omega_{\bullet}^{\hat E_k})=0$ one has the following canonical
isomorphism:
\begin{equation}\label{isomorphism-relation-between-integral-total-bordisms-groups-in-quantum-PDEs-2}
    {}^{S}Bor_{\bullet}(\hat E_k)\cong\Omega_{\bullet}^{\hat E_k}\times{}^{S}Cyc_{\bullet}(\hat E_k).
\end{equation}
\end{lemma}
\begin{proof}
The proof is similar to the one of Lemma
\ref{quantum-PDE-lower-order-relation-bigraded-bordism}.
\end{proof}

\begin{lemma}{\em(Integral ringoid of PDE's in $\mathfrak{Q}_S$ and quantum conservation laws).}
Let $\hat E_k\subset \hat J^k_{m|n}(W)$ be a PDE in the category
$\mathfrak{Q}_S$ as defined in Lemma
\ref{quantum-PDE-lower-order-relation-total-bordism}. Then,
$\Omega^{\hat E_k}_\bullet\equiv\bigoplus_{0\le p\le
m+n}\Omega^{\hat E_k}_p$, has a natural structure of graded ringoid,
with respect to the (partial) binary operations similar to the
commutative case. We call $\Omega^{\hat E_k}_\bullet$ the {\em
integral ringoid} of $\hat E_k$. Furthermore, quantum conservation
laws of order $k$, $\hat f\in Map(\Omega^{\hat E_k}_{p|q},B_k)\equiv
\mathbf{H}_{p|q}(\hat E_k)$, can be projected on their classic
limits $\hat f\mapsto\hat f_C\equiv c\circ\hat f\in Map(\Omega^{\hat
E_k}_{p|q},\mathbb{K})\equiv \mathbf{H}_{p|q}(\hat E_k)_C$. By
passing to the corresponding total spaces, we get the exact commutative diagram {\em(\ref{quantum-conservation-laws-classic-limit})}.
\begin{equation}\label{quantum-conservation-laws-classic-limit}
\xymatrix{0\ar[d]&0\ar[d]&\\
\mathbf{H}_{p|q}(\hat E_k)\ar[d]\ar[r]&\mathbf{H}_{p|q}(\hat E_k)_C\ar[d]\ar[r]&0\\
\mathbf{H}_\bullet(\hat E_k)\ar[r]&\mathbf{H}_\bullet(\hat
E_k)_C\ar[r]&0\\}
\end{equation}
Moreover, graded ringoid homomorphisms $\hat h\in
Hom_{ringoid}(\Omega_\bullet^{\hat E_k},\mathbb{K})$, can be
identified by means of classic limit quantum conservation laws of
$\hat E_k$. One has the exact commutative diagram {\em(\ref{ringoid-homomorphisms-quantum-conservation-laws})}.
\begin{equation}\label{ringoid-homomorphisms-quantum-conservation-laws}
\xymatrix{0\ar[r]&{}^{R}\mathbf{H}_\bullet(\hat E_k)\ar[r]\ar[d]&\mathbf{H}_\bullet(\hat E_k)\ar[d]\\
0\ar[r]&Hom_{ringoid}(\Omega_\bullet^{\hat
E_k},\mathbb{K})\ar[r]\ar[d]&\mathbf{H}_\bullet(\hat E_k)_C\ar[d]\\
&0&0\\}
\end{equation}
that defines a subalgebra ${}^{R}\mathbf{H}_\bullet(\hat E_k)$ of
$\mathbf{H}_\bullet(\hat E_k)$, whose elements we call {\em rigid
quantum conservation laws}, and whose classic limit can be
identified with ringoid homomorphisms $\Omega_\bullet^{\hat
E_k}\to\mathbb{K}$. In particular, quantum conservation laws arising
by full quantum differential form classes
\begin{equation}\label{quantum-differential-conservation-laws}
\left\{\begin{array}{ll}
         [\alpha]& \in \bigoplus_{p,q\ge 0 }\hat{\frak I}(\hat
E_k)^{p|q} \\
         & \hat{\frak I}(\hat
E_k)^{p|q}\equiv{{\widehat{\Omega}^{p|q}(\hat E_k)\cap
d^{-1}(C\widehat{\Omega}^{p+1|q+1}(\hat
E_k))}\over{d\widehat{\Omega}^{p-1|q-1}(E_k)\oplus\{C\widehat{\Omega}^{p|q}(\hat
E_k)\cap d^{-1}(C\widehat{\Omega}^{p+1|q+1}(\hat E_k)))\}}}
       \end{array}
\right.
\end{equation}
belong to ${}^{R}\mathbf{H}_\bullet(\hat E_k)$.
\end{lemma}
\begin{proof}
The proof follows directly from above lemmas. (For details on spaces
$\hat{\frak I}(\hat E_k)^{p|q}$ see Refs.\cite{PRA15, PRA20,
PRA22}.)
\end{proof}

Let us, now, denote $\mathop{\Omega}\limits_c{}_{p|q}^{\hat E_k}$
(or $\mathop{\Omega}\limits_c{}_{p+q}^{\hat E_k}$), the classic
limit of integral $(p|q)$-bordism group of $\hat E_k$, i.e., the
$(p+q)$-bordism group of classic limits of integral quantum supermanifolds
$N\subset\hat E_k$, such that $\dim_AN=p|q$.\footnote{Let us recall that a {\em
$(p|q)$-dimensional integral quantum supermanifold} $V$ of $\hat
E_k$, $0\le p\le m$, $0\le q\le n$, with boundary $\partial  V$,
(or eventually with $\partial  V=\hskip 2pt\varnothing$), we mean an
element $ V\in C_{p|q}(\hat E_{k+h},A)$, $h\ge 0$, such that $
TV\subset\mathbf{E}_{m|n}^{k+h}$. So, if $V=\sum_i a^iu_i+\sum_j b^jv_i$,
$ a_i\in A_0$, $ b_j\in A_1$, one has $\partial
V=\sum_i(-1)^ia^i\partial_iu+\sum_j(-1)^jb^j\partial_jv$. (For more details see \cite{PRA22}.} Furthermore, let us
denote by $\mathop{\Omega}\limits_c{}_{\widehat{p+q}}^{\hat E_k}$
the classic limit of total integral $(p+q)$-bordism group of $\hat
E_k$, i.e., the $(p+q)$-bordism group of classic limits of integral
supermanifolds $N\subset \hat E_k$, such that $\dim_AN=u|v$, with
$u+v=p+q$. One has the exact commutative diagram (\ref{commutative-exact-diagram-relation-integral-bord-classic-limit-total-integral-bord}).

\begin{equation}\label{commutative-exact-diagram-relation-integral-bord-classic-limit-total-integral-bord}
\xymatrix{0\ar[r]&\Omega_{p|q}^{\hat
E_k}\ar[d]\ar[r]&\Omega_{p+q}^{\hat E_k}\ar[d]\\
0\ar[r]&\mathop{\Omega}\limits_c{}_{p+q}^{\hat E_k}\ar[d]\ar[r]&\mathop{\Omega}\limits_c{}_{\widehat{p+q}}^{\hat E_k}\ar[d]\\
&0&0\\}
\end{equation}

Taking into account Theorem 3.6 in \cite{PRA22} we get a relation
between $\Omega_{m-1|n-1}^{\hat E_k}$,
$\mathop{\Omega}\limits_c{}_{p|q}^{\hat E_k}$ and the bordism group
$\Omega_{m+n-2}$. In fact, we can see that there is a relation
between integral bordism groups in quantum super PDEs and Reinhart
integral bordism groups of commutative manifolds. More precisely,
let $N_0, N_1\subset\hat E_k\subset\hat J^k_{m|n}(W)$ be closed
admissible integral quantum supermanifolds of a quantum super PDE
$\hat E_k$, of dimension $(m-1|n-1)$ over $A$, such that $N_0\DU
N_1=\partial V$, for some admissible integral quantum supermanifold
$V\subset \hat E_k$, of dimension $(m|n)$ over $A$. Then $(N_0)_C\DU
(N_1)_C=\partial V_C$ iff $(N_0)_C$ and $(N_1)_C$ have the same
Stiefel-Whitney and Euler characteristic numbers. In fact, by
denoting $ \Omega_p^\uparrow$ the Reinhart $p$-bordism groups and
$\Omega_p$ the $p$-bordism group for closed smooth finite
dimensional manifolds respectively, one has the exact
commutative diagram (\ref{Reinhart-bordism-groups-relation}).
\begin{equation}\label{Reinhart-bordism-groups-relation}
\xymatrix{0\ar[r]&K^{\hat
E_k}_{m-1|n-1;m+n-2}\ar[r]&\Omega_{m-1|n-1}^{\hat E_k}\ar[r]&
\mathop{\Omega}\limits_c{}_{m+n-2}^{\hat E_k}\ar[d]\ar[r]\ar[dr]&0&\\
&0\ar[r]&K^\uparrow_{m+n-2}\ar[r]&\Omega^\uparrow_{m+n-2}\ar[r]&
\Omega_{m+n-2}\ar[r]& 0\\}
\end{equation}

This has as a consequence that if $N_0\DU N_1=\partial V$, then
$(N_0)_C\DU (N_1)_C=\partial V_C$ iff $(N_0)_C$ and $(N_1)_C$ have
the same Stiefel-Whitney and Euler characteristic
numbers.\footnote{Note that for $p+q=3$ one has $K_3^\uparrow=0$,
hence one has $\Omega_3^\uparrow=\Omega_3$.}

From above exact commutative diagram one has that
$\Omega_{m-1|n-1}^{\hat E_k}$ is an extension of a subgroup of
$\Omega_{m+n-2}$.

Let us consider, now, the following lemmas.

\begin{lemma}{\em\cite{PRA25}}\label{bordism-groups-crystallography}
Bordism groups, $\Omega_p$, relative to smooth manifolds can be
considered as extensions of some crystallographic subgroup
$G\triangleleft G(d)$.
\end{lemma}

\begin{lemma}
If the group $G$ is an extension of $H$, any subgroup
$\widetilde{G}\vartriangleleft G$ is an extension of a subgroup
$\widetilde{H}\vartriangleleft H$.
\end{lemma}

\begin{proof}
In fact $\widetilde{G}$ is an extension of
$p(\widetilde{G})\vartriangleleft H$, with respect to the following
short exact sequence:
$\xymatrix{0\ar[r]&K\ar[r]&G\ar[r]^{p}&H\ar[r]&0\\}$.
\end{proof}

Therefore by using above two lemmas, we get also that
$\Omega_{m-1|n-1}^{\hat E_k}$ is an extension of some
crystallographic subgroup $G\triangleleft G(d)$.
\end{proof}

The theorem below relates the integrability properties of a quantum
super PDE to crystallographic groups. Let us first give the
following definition.

\begin{definition}
We say that a quantum super PDE $\hat E_k\subset \hat J^k_{m|n}(W)$
is an {\em extended crystal quantum super PDE}, if conditions of
Theorem \ref{crystal-structure-quantum-super-pdes} are verified.
Then, for such a PDE $\hat E_k$ are defined its crystal group $G(d)$
and crystal dimension $d$.
\end{definition}

In the following we relate crystal structure of quantum super PDE's
to the existence of global smooth solutions for smooth boundary
value problems, by identifying an algebraic-topological obstruction.

\begin{theorem}\label{obstruction-smooth-solutions}
Let $B_k$ be the model quantum superalgebra of $\hat J^k_{m|n}(W)$,
$k\ge 0$. (See \cite{PRA21, PRA22}.) We denote also by $
B_\infty=\lim_k B_k$.\footnote{We also adopt the notation $B_k(A)$
and $B_\infty(A)$, whether it is necessary to specify the starting
original quantum super algebra $A$.} Let $\hat E_k\subset \hat
J^k_{m|n}(W)$ be a formally integrable and completely
integrable quantum super PDE. Then, in the algebra
$\mathbf{H}_{m-1|n-1}(\hat E_k)\equiv Map(\Omega_{m-1|n-1}^{\hat
E_k};B_k)$, {\em Hopf quantum superalgebra} of $\hat E_k$, there is
a quantum sub-superalgebra, {\em (crystal Hopf quantum superalgebra)}
of $\hat E_k$.\footnote{Recall that with the term {\em quantum Hopf superalgebra} we mean an extension
$\xymatrix{A\ar[r]&C\equiv A\otimes_{\mathbb{K}}H\ar[r]&D\ar[r]&D/C\ar[r]&0\\}$, where $H$ is an Hopf $\mathbb{K}$-algebra and $A$ is a quantum superalgebra. (For more details on generalized Hopf algebras, associated to PDE's, see Refs.\cite{PRA10, PRA11, PRA22}.)} On such an algebra we can represent the quantum
superalgebra $B^{G(d)}$ associated to the quantum crystal supergroup
$G(d)$ of $\hat E_k$. (This justifies the name.) We call {\em
quantum crystal conservation superlaws} of $\hat E_k$ the elements
of its quantum Hopf crystal superalgebra. Then, the obstruction to
find global smooth solutions of $\hat E_k$, for integral boundaries
with orientable classic limit, can be identified with the quotient
$\mathbf{H}_{m-1|n-1}(\hat E_\infty)/B_\infty^{\Omega_{m+n-2}}$.
\end{theorem}

\begin{proof}
Let $N_0, N_1\subset \hat E_k$ be two respectively initial and
final, closed compact Cauchy data of $\hat E_k$. Then there exists a
weak, (resp. singular, resp. smooth) solution $V\subset \hat E_k$,
such that $\partial V=N_0\DU N_1$, iff $X\equiv N_0\DU
N_1\in[0]\in\Omega_{m-1|n-1,w}^{\hat E_k}$, (resp.
$X\in[0]\in\Omega_{m-1|n-1,s}^{\hat E_k}$, resp.
$X\in[0]\in\Omega_{m-1|n-1}^{\hat E_k}$). Let $X_C$ be orientable,
then $X$ is the boundary of a smooth solution, iff $X$ has zero all
the integral characteristic quantum supernumbers, i.e.,
$<\alpha,X>=0$, $\forall\alpha\in\mathbf{H}_{m-1|n-1}(\hat
E_\infty)=Map(\Omega_{m-1|n-1}^{\hat E_\infty},B_\infty)$. Taking
into account the following short exact sequence: $0\to
\Omega_{m-1|n-1,w}^{\hat E_k}\to\widetilde{\Omega}_{m+n-2}$, where
$\widetilde{\Omega}_{m+n-2}\vartriangleleft G(d)$, for some
crystallographic group $G(d)$, we get also the following short exact
sequence: $\mathbf{H}_{m-1|n-1}(\hat E_\infty)\leftarrow
B_\infty^{\widetilde{\Omega}_{m+n-2}}\leftarrow 0$. So
$B_\infty^{\widetilde{\Omega}_{m+n-2}}$ can be identified with a
subalgebra of $\mathbf{H}_{m-1|n-1}(\hat E_\infty)$. Then the
obstruction to find smooth solutions can be identified with the
quotient $\mathbf{H}_{m-1|n-1}(\hat E_\infty)/
B_\infty^{\widetilde{\Omega}_{m+n-2}}$. Taking into account that
$\widetilde{\Omega}_{m+n-2}\vartriangleleft\Omega_{m+n-2}\vartriangleleft
G(d)$, we can also represent $B^{\widetilde{\Omega}_{m+n-2}}$ with
$B_\infty^{\Omega_{m+n-2}}$, or with $B_\infty^{G(d)}$. Thus, it is
justified also call $B_\infty^{\widetilde{\Omega}_{m+n-2}}$ as
crystal quantum superlaws algebra of $\hat E_k$.
\end{proof}

\begin{definition}
We define {\em crystal obstruction} of $\hat E_k$ the above quotient
of algebras, and put: $ cry(\hat E_k)\equiv
\mathbf{H}_{m-1|n-1}(\hat E_\infty)/B_\infty^{\Omega_{m+n-2}}$. We
call {\em quantum $0$-crystal super PDE} a quantum super PDE $\hat
E_k\subset \hat J^k_{m|n}(W)$ such that $cry(\hat E_k)=0$.
\end{definition}

\begin{remark}
A quantum extended $0$-crystal super PDE $\hat E_k\subset \hat
J^k_{m|n}(W)$ does not necessitate to be a quantum $0$-crystal super
PDE. In fact $\hat E_k$ is an extended $0$-crystal quantum super PDE
if $\Omega_{m-1|n-1,w}^{\hat E_k}=0$. This does not necessarily
implies that $\Omega_{m-1|n-1}^{\hat E_k}=0$. In fact, the different
types of integral bordism groups of PDE's in the category
$\mathfrak{Q}_S$, are related by the following proposition.
\end{remark}

\begin{proposition}{\em(Relations between integral bordism groups)}{\em\cite{PRA22}}
The different types of integral bordism groups for a quantum super
PDE, are related by the exact commutative diagram reported in
{\em(\ref{relations-between-integral-bordism-groups-commutative-diagram})}.

\begin{equation}\label{relations-between-integral-bordism-groups-commutative-diagram}
\xymatrix{
&0\ar[d]&0\ar[d]&0\ar[d]&\\
0\ar[r]&K^{\hat E_k}_{m-1|n-1,w/(s,w)}\ar[d]\ar[r]&
K^{\hat E_k}_{m-1|n-1,w}\ar[d]\ar[r]&K^{\hat E_k}_{m-1|n-1,s,w}\ar[d]\ar[r]&0\\
0\ar[r]&K^{\hat E_k}_{m-1|n-1,s}\ar[d]\ar[r]&
\Omega^{\hat E_k}_{m-1|n-1}\ar[d]\ar[r]&\Omega^{\hat E_k}_{m-1|n-1,s}\ar[d]\ar[r]&0\\
&0\ar[r]&\Omega^{\hat E_k}_{m-1|n-1,w}\ar[d]\ar[r]&\Omega^{\hat E_k}_{m-1|n-1,w}\ar[d]\ar[r]&0\\
&&0&0&}
\end{equation}

One has the canonical isomorphisms:
\begin{equation}\label{canonical-isomorphisms-integral-bordism-groups-comm-pde}
\left\{   \begin{array}{l}
     K^{\hat E_k}_{m-1|n-1,w/(s,w)}\cong K^{\hat E_k}_{m-1|n-1,s}\\
\Omega^{\hat E_k}_{m-1|n-1}/K^{\hat E_k}_{m-1|n-1,s}\cong
\Omega^{\hat E_k}_{m-1|n-1,s}\\
\Omega^{\hat E_k}_{m-1|n-1,s}/K^{\hat
E_k}_{m-1|n-1,s,w}\cong\Omega^{\hat E_k}_{m-1|n-1,w}\\
\Omega^{\hat E_k}_{m-1|n-1}/K^{\hat
E_k}_{m-1|n-1,w}\cong\Omega^{\hat E_k}_{m-1|n-1,w}.\\
   \end{array}\right.
\end{equation}
\end{proposition}

\begin{cor}\label{main7}
Let $\hat E_k\subset \hat J^k_{m|n}(W)$ be a quantum $0$-crystal
super PDE. Let $N_0, N_1\subset \hat E_k$ be two closed initial and
final Cauchy data of $\hat E_k$ such that $X\equiv N_0\DU
N_1\in[0]\in\Omega_{m-1|n-1}$, and such that $X_C$ is orientable.
Then there exists a smooth solution $V\subset \hat E_k$ such that
$\partial V=X$.
\end{cor}

Let us, now, revisit some definitions and results about stability
of mappings and their relations with singularities of mappings,
adapting them to the category $\mathfrak{Q}_S$.

\begin{definition}
Let $X$, (resp. $Y$), be a quantum supermanifold of dimension $m|n$,
(resp. $r|s$), with respect to a quantum superalgebra $A=A_0\oplus
A_1$, (resp. $B=B_0\oplus B_1$). We shall assume that the center
$Z=Z(A)$ of $A$, acts on $B$ that becomes a $Z$-module. Furthermore, we shall assume that $Z$ is Noetherian.\footnote{In
the following, whether it is not differently specified, $X$ and $Y$
are such quantum supermanifolds.} Let $f\in Q_w^\infty(X,Y)$. Then
$f$ is {\em stable} if there is a neighborhood $W_f\subset
Q_w^\infty(X,Y)$ of $f$, in the natural Whitney-type topology of
$Q_w^\infty(X,Y)$, such that every $W_f$ is contained in the orbit
of $f$, via the action of the group $\hat Diff(X)\times \hat
Diff(Y)$.\footnote{Here $\hat Diff(X)$ denotes the group of quantum
diffeomorphisms of a quantum super manifold $X$.} This is equivalent
to say that for any $f'\in W_f$ there exist quantum diffeomorphisms
$g:X\to X$ and $h:Y\to Y$ such that $h\circ f= f'\circ g$.
Furthermore, $f$ is called {\em infinitesimally stable} if there
exist a map $\zeta:X\to TY$, such that $\pi_Y\circ\zeta=f$, where
$\pi_Y:TY\to Y$ is the canonical map, and integrable vector fields
$\nu:Y\to TY$, $\xi:X\to TX$, such that $\zeta=T(f)\circ\xi+\nu\circ
f$. Thus the diagram {\em(\ref{infinitesimal-stability})}is commutative.
\begin{equation}\label{infinitesimal-stability}
    \xymatrix{TX\ar@/^1pc/[d]^{\pi_X}\ar[r]^(.4){T(f)}&TY\bigoplus TY\ar@/^1pc/[d]^{\pi_Y}\ar[r]^(0.4){+}&TY\ar[d]^{\pi_Y}\\
    X\ar[u]^{\xi}\ar[urr]^(.3){\zeta}\ar[r]_{f}&Y\ar[u]^(.3){\nu}\ar@{=}[r]&Y\\}
\end{equation}
\end{definition}

\begin{theorem}
Let $X$ be a compact quantum supermanifold and $f:X\to Y$ be quantum smooth. Then $f$ is
stable iff $f$ is infinitesimally stable. Furthermore, if $f$ is a
proper mapping, then does not necessitate assume that $X$ is
compact.\footnote{Recall that a map $f:X\to Y$ between topological
spaces is a {\em proper map} if for every compact subset $K\subset
Y$, $f^{-1}(K)$ is a compact subset of $X$.}
\end{theorem}

\begin{proof}
Note that the infinitesimal stability, requires existence of flows $g_t:X\to X$, $\partial g=\xi$, $h_t:Y\to Y$, $\partial h=\nu$, such that for the infinitesimal variation $\zeta$ of $f_t=h_t\circ f\circ g_t$ one has $\zeta=T(f)\circ\xi+\nu\circ f$.
In fact, one has the following lemma.

\begin{lemma}\label{Lie-derivative}
Let $(W,V,\pi_W;\mathbb{B})$ be a bundle of geometric objects in the
category $\mathfrak{Q}_S$ and in the intrinsic sense \cite{PRA1, PRA2}
\footnote{See also Refs.\cite{PRA9} for related
subjects.}. Let $\phi:\mathbb{R}\times V\to V$ be a one-parameter
group of $Q^\infty_w$ transformations of $V$, $\xi=\partial\phi$ its
infinitesimal generator and $s:V\to W$ a field of geometric objects,
i.e. a section of $\pi_W$. Then, $\phi$ induces a deformation
$\widetilde{s}$ of $s$ defined by means of the commutative
diagram {\em(\ref{infinitesimal-variation})}.
\begin{equation}\label{infinitesimal-variation}
\xymatrix{\mathbb{R}\times\mathbb{R}\times V\ar[d]_{(id_{\mathbb{R}},\phi)}\ar[r]^{\widetilde{s}}&W\\
\mathbb{R}\times V\ar[r]^{(id_{\mathbb{R}},s)}&\mathbb{R}\times W\ar[u]^{\mathop{\phi}\limits^{\overline{\circ}}{}_\lambda}\\}
\end{equation}
where $\mathop{\phi}\limits^{\overline{\circ}}{}_\lambda\equiv\mathbb{B}(\phi^{-1}_\lambda)$, $\forall\lambda\in\mathbb{R}$. One has $\widetilde{s}_{(0,0)}=s$. Then, for the infinitesimal variation of
$\widetilde{s}$ ({\em Lie derivative} of $s$ with respect to the integrable field $\xi$),  $\partial(\widetilde{s}\circ d):V\to s^*vTW$, one has:
\begin{equation}\label{infinitesimal variation}
    \begin{array}{ll}
      \partial(\widetilde{s}\circ d)& =\partial(s\circ\phi)+\partial(\mathop{\phi}\limits^{\overline{\circ}})\circ s \\
      & =T(s)\circ\xi+\nu\circ s.
    \end{array}
\end{equation}
\end{lemma}

\begin{proof}
This lemma can be proved by copying the intrinsic proof for the
commutative case given in \cite{PRA1}.
\end{proof}
In our case we can consider the following situation, with respect to Lemma \ref{Lie-derivative}, $W\equiv X\times Y$, $V\equiv X$, $\mathbb{B}(g_\lambda)=h_\lambda$ and $s=(id_X,f)$.

Furthermore, in the case that $X$ is compact, the proof follows the
same lines of the proof given by Mather for commutative manifolds
\cite{MATH2}.
\end{proof}

\begin{theorem}
Stable maps $f:X\to Y$ do not necessitate to be dense in
$Q_w^\infty(X,Y)$.
\end{theorem}

\begin{proof}
This is just a corollary of the corresponding theorem for commutative manifolds given by Thom-Levine \cite{LEV1, LEV2}.
\end{proof}

\begin{example}{\em(Submersions and stability).}
Let $X$ be a compact quantum supermanifold. Let $f:X\to Y$ be a
quantum differentiable mapping of maximum possible super-rank. If
$m\ge r>1$,  $n\ge s>1$, $f$ is a {\em quantum submersion} and it is
(infinitesimally) stable.
\end{example}

\begin{example}{\em(Immersions and stability).}
Let $X$ be a compact quantum supermanifold. Let $f:X\to Y$ be a quantum
differentiable mapping of maximum possible super-rank. If $m\le r$,
$n\le s$, $f$ is an {\em immersion} and if it is $1:1$ then it is
also stable. (Not all immersions are stable.)
\end{example}

\begin{definition}{\em(Singular solutions of quantum super PDE's).}
Let $\pi:W\to M$ be a fiber bundle, where $M$ is a quantum supermanifold of
dimension $(m|n)$ on the quantum superalgebra $A$ and $W$ is a
quantum supermanifold of dimension $(m|n,r|s)$ on the quantum
superalgebra $B\equiv A\times E$, where $E$ is also a $Z$-module,
with $Z=Z(A)$ the center of $A$, assumed to be Noetherian.

Let $E_k\subset JD^k(W)$ be a quantum super PDE. By using the
natural embedding $J\hat D^k(W)\subset \hat J^k_{m|n}(W) $, we can
consider quantum super PDEs $\hat E_k\subset J\hat D^k(W)$ like
quantum super PDEs $\hat E_k\subset \hat J^k_{m|n}(W) $, hence we
can consider solutions of $\hat E_k$ as $(m|n)$-dimensional, (over
$A$), quantum supermanifolds $V\subset\hat E_k$ such that $V$ can be
represented in the neighborhood of any of its points $q'\in V$,
except for a nowhere dense subset $\Sigma(V)\subset V$, of
dimension $\le (m-1|n-1)$, as $N^{(k)}$, where $N^{(k)}$ is the
$k$-quantum prolongation of a $(m|n)$-dimensional (over $A$) quantum
supermanifold $N\subset W$. In the case that $\Sigma(V)=\varnothing$,
we say that $V$ is a {\em regular solution} of $\hat E_k\subset \hat
J^k_{m|n}(W)$. Solutions $V$ of $\hat E_k\subset \hat J^k_{m|n}(W)$,
even if regular ones, are not, in general diffeomorphic to their
projections $\pi_k(V)\subset M$, hence are not representable by
means of sections of $\pi:W\to M$. $\Sigma(V)\subset V$ is the {\em
singular points set} of $V$. Then $V\setminus\Sigma(V)=\bigcup_rV_r$
is the disjoint union of connected components $V_r$. For every of
such components $\pi_{k,0}:V_r\to W$ is an immersion and can be
represented by means of $k$-prolongation of some quantum
supermanifold of dimension $m|n$ over $A$, contained in $W$. Whether
we consider $\hat E_k$ as contained in $J\hat D^k(W)$ then {\em
regular solutions} are locally obtained as image of $k$-derivative
of sections of $\pi:W\to M$. So we can (locally) represent such
solutions by means of mapping $f:M\to E_k$, such that $f=D^ks$, for
some section $s:M\to W$.

We shall also consider
solutions of $\hat E_k\subset \hat J^k_{m|n}(W)$, any subset $V\subset\hat E_k$,
that can be obtained as projections of ones of the previous type,
but contained in some $s$-prolongation $\hat E_{k+s}\subset \hat
J^{k+s}_{m|n}(W)$, $s>0$.

We define {\em weak solutions}, solutions $V\subset \hat E_k$, such that
the set $\Sigma(V)$ of singular points of $V$, contains also
discontinuity points, $q,q'\in V$, with
$\pi_{k,0}(q)=\pi_{k,0}(q')=a\in W$, or $\pi_{k}(q)=\pi_{k}(q')=p\in
M$. We denote such a set by $\Sigma(V)_S\subset\Sigma(V)$, and, in
such cases we shall talk more precisely of {\em singular boundary}
of $V$, like $(\partial V)_S=\partial V\setminus\Sigma(V)_S$.
However for abuse of notation we shall denote $(\partial V)_S$,
(resp. $\Sigma(V)_S$), simply by $(\partial V)$, (resp.
$\Sigma(V)$), also if no confusion can arise.
\end{definition}

\begin{definition}{\em(Stable solutions of quantum super PDE's).}
Let us consider a quantum super PDE $\hat E_k\subset J\hat D^k(W)$, and let us denote $\underline{Sol}(\hat E_k)$ the
set of regular solutions of $E_k$. This has a natural structure of locally convex manifold.
Let $f:X\to E_k$ be a
regular solution, where $X\subset M$ is a smooth $(m|n)$-dimensional
compact manifold with boundary $\partial X$. Then $f$ is {\em
stable} if there is a neighborhood $W_f$ of $f$ in
$\underline{Sol}(\hat E_k)$,
such that each $f'\in W_f$ is equivalent to $f$, i.e., $f$ is
transformed in $f'$ by some integrable vertical symmetries of $\hat E_k$.
\end{definition}
\begin{equation}\label{commutative-diagram}
\scalebox{0.9}{$\xymatrix{\hat E_k\hskip 3pt\ar@{^{(}->}[d]&vT\hat
E_k\hskip 3pt\ar@{^{(}->}[d]\ar[l]&(D^ks)^*vT\hat
E_k\hskip 3pt\ar@{^{(}->}[l]\hskip 3pt\ar@{^{(}->}[d]\ar@{=}[r]^(.6){\sim}&
\hat E_k[s]\ar@{^{(}->}[d]\\
J\hat D^k(W)\ar@/_2pc/[dd]_{\pi_k}\ar[d]^{\pi_{k,0}}&vTJ\hat
D^k(W)\ar[l]\ar[d]&(D^ks)^*vTJ\hat D^k(W)\hskip 3pt\ar@{^{(}->}[l]
\ar@{=}[r]^(.6){\sim}&J\hat D^k(\hat E[s])\ar[d]_{\bar\pi_{k,0}}\ar@/^2pc/[dd]^{\bar\pi_k}\\
W\ar[d]_{\pi}&vTW\ar[l]_{\pi'}&&\hat E[s]\hskip 3pt\ar@{^{(}->}[ll]\ar[d]^{\bar\pi}\\
M\ar@/^4pc/[uuu]^(.8){D^ks}\ar@/_1pc/[u]^{s}\ar@{=}[rrr]&&&M\ar@/^1pc/[u]^{\nu}\ar@/_4pc/[uuu]_(.8){D^k\nu}}$}
\end{equation}

\begin{theorem}\label{deformation}
Let $\hat E_k\subset J\hat D^k(W)$ be a $k$-order quantum super PDE
on the fiber bundle $\pi:W\to M$ in the category of quantum smooth
supermanifolds. Let $s:M\to W$ be a section, solution of $\hat E_k$,
and let $\nu:M\to s^*vTW\equiv \hat E[s]$ be an integrable solution
of the linearized equation $\hat E_k[s]\subset  J\hat D^k(\hat
E[s])$. Then to $\nu$ it is associated a flow
$\{\phi_\lambda\}_{\lambda\in J}$, where $J\subset \mathbb{R}$ is a
neighborhood of $0\in\mathbb{R}$, that transforms $V$ into a new
solution $\widetilde{V}\subset \hat E_k$.
\end{theorem}

\begin{proof}
Let $(x^\alpha,y^j)$ be fibered coordinates on $W$. Let
$\nu=\partial y_j(\nu^j):M\to s^*vTW$ a vertical vector field on $W$
along the section $s:M\to W$. Then $\nu$ is a solution of $\hat
E_k[s]$ iff the diagram (\ref{commutative-diagram}) is commutative.
Then $D^k\nu(p)$ identifies, for any $p\in M$, a vertical vector on
$\hat E_k$ in the point $q=D^ks(p)\in V=D^ks(M)\subset \hat E_k$. On
the other hand infinitesimal vertical symmetries on $E_k$ are
locally written in the form
\begin{equation}
\left\{\begin{array}{l} \zeta=\sum_{0\le|\alpha|\le k}\partial
y_j^\alpha(Y^j_\alpha),\quad 0=\zeta.F^I=<dF^I,\left(
\sum_{0\le r\le k}\partial y_j^{\alpha_1\cdots\alpha_r}(Y^j_{\alpha_1\cdots\alpha_r})\right)>\\
{}\\
Y^j_\alpha=Z^{(0)}_\alpha( Y^j),\quad
Z^{(0)}_\alpha=\partial x_\alpha+\partial y_sy^s_\alpha\\
{}\\
Y^j_{\alpha_1\cdots \alpha_ri}=Z^{(r)}_i(Y^j_{\alpha_1\cdots
\alpha_r}),\quad Z^{(r)}_i=Z_i^{(r-1)}+\partial y_s^{\gamma_1\cdots
\gamma_r}y^s_{\gamma_1\cdots
\gamma_ri}\\
\end{array}\right.
\end{equation}
where $Y^j_\alpha\in Q^\infty_w(U\subset J\hat
D^k(W);\mathop{\widehat{A}}\limits^{|\alpha|}(E))$,
$\mathop{\widehat{A}}\limits^{|\alpha|}(E)\equiv
Hom_Z(\mathop{\overbrace{A\otimes_Z\cdots\otimes_ZA}}\limits^{|\alpha|};E)$,
$0\le|\alpha|\le k$. $\partial y_j^\alpha(q)\in
Hom_Z(\mathop{\widehat{A}}\limits^{|\alpha|}(E);T_qJ\hat D^k(W))$,
$y^j_{\alpha_1\cdots\alpha_r}\in
Q^\infty_w(U;\mathop{\widehat{A}}\limits^{r}(E))$. Then we can see
that solutions of $\hat E_k[s]$ are vertical vector fields $\nu:M\to
s^*vTW\equiv \hat E[s]$, such that their prolongations
$D^k\nu=\zeta\circ D^ks$, for some vertical symmetry $\zeta$ of
$\hat E_k$.

Therefore, the flows of above integrable vertical vector fields, transform
regular solutions $V$ of $\hat E_k$ into new solutions of $\hat
E_k$. Solutions of the linearized equation $\hat E_k[s]$ give
initial conditions for the determination of such vertical flows.
\end{proof}
The following lemmas are also important to understand how the
structure of solutions of $\hat E_k[s]$ are related to the vertical
symmetries of $\hat E_k$. (For complementary informations on the contact structure of $\hat J^k_{m|n}(W)$, see \cite{PRA21}.)

\begin{lemma}{\em(Symmetries of horizontal $k$-order contact ideals).}\label{Horiz-symm}
Let  $\rceil:\hat J^k_{m|n}(W)\to \hat J^{k+1}_{m|n}(W)$  be a {\em
quantum $(k+1)$-connection} on $W$, i.e.,  a $Q^\infty_w$-section of
$\pi_{k+1,k}$. (The restriction of $\rceil$ to $\hat J^k(W)\subset
\hat J^k_{m|n}(W)$ is also called quantum $(k+1)$-connection). Let
$\widehat{\mathfrak{H}}_k(\rceil)$ be the {\em quantum horizontal
$k$-order contact ideal} of $\widehat{\Omega}^\bullet(\hat
J^k_{m|n}(W))$ given by
$\widehat{\mathfrak{H}}_k(\rceil)\equiv\rceil^*\widehat{\mathfrak{C}}_{k+1}(W)$,
where $\widehat{\mathfrak{C}}_{k+1}(W)$ is the contact ideal of
$\hat J^{k+1}_{m|n}(W)$. Locally one can write
$\widehat{\mathfrak{H}}_k(\rceil)=
<\omega^j,\dots,\omega^j_{\alpha_1\dots\alpha_{k-1}},\hat
H^j_{\alpha_1\dots\alpha_k}>$, where
$$\left\{\begin{array}{ll}
           \hat H^j_{\alpha_1\dots\alpha_k}& \equiv
            \rceil^*\omega^j_{\alpha_1\dots\alpha_k}
            =\rceil^*(dy^j_{\alpha_1\dots\alpha_k}-
            y^j_{\alpha_1\dots\alpha_k\beta}dx^\beta)\\
            &\\
            &=dy^j_{\alpha_1\dots\alpha_k}-
            \rceil^j_{\alpha_1\dots\alpha_k\beta}dx^\beta\in\widehat{\Omega}^1(\hat J^k_{m|n}(W)),\\
         \end{array}\right.$$
with $\rceil^j_{\alpha_1\dots\alpha_k\beta}\equiv
y^j_{\alpha_1\dots\alpha_k\beta}\circ\rceil\in
\widehat{\Omega}^0(J^k_{m|n}(W))$. $\widehat{\mathfrak{C}}_k(W)$ is
a subideal of $\widehat{\mathfrak{H}}$. Then the {\em quantum
horizontal $k$-order Cartan distribution}
$\mathbf{H}_k(\rceil)\subset TJ^k_{m|n}(W)$ (identified by a
$(k+1)$-connection $\rceil$) is the Cauchy characteristic
distribution associated to $\mathfrak{H}_k(\rceil)$.
$\mathfrak{d}(\mathbf{H}_k(\rceil))$ admits the following local
{\em(canonical basis):}\footnote{For a distribution
$\mathbf{E}\subset TX$ on a manifold $X$, we denote by
$\mathfrak{d}(\mathbf{E})$ the vector space of vector fields on $X$
belonging to $\mathbf{E}$.}

$$\left\{\begin{array}{ll}
\zeta_\alpha=&\partial x_\alpha+\partial y_jy^j_\alpha\\
&\\
&+\cdots+\partial
y_j^{\alpha_1\dots\alpha_{k-1}}y^j_{\alpha_1\dots\alpha_{k-1}\alpha}
+\partial
y_j^{\alpha_1\dots\alpha_k}\rceil^j_{\alpha_1\dots\alpha_k\alpha}.\\
\end{array}\right.$$
For any quantum $(k+1)$-connection $ \rceil$ on $W$, one has the
following direct sum decompositions:
\begin{equation}\label{splitting-quantum-Cartan-distribution}
   \left\{
\begin{array}{l}
 \mathbf{E}^k_{m|n}(W)_q
\cong\mathbf{H}_k(\rceil)_q \bigoplus Hom_Z(S^k(T_aN);\nu_a) \\
 \widehat{\Omega}^1(\hat J^k_{m|n}(W))\cong\widehat{\Omega}^1(\hat
J^k_{m|n}(W))_v \bigoplus\widehat{\mathfrak{H}}_k(\rceil)^1
\end{array}
\right.
\end{equation}
with $a\equiv\pi_{k,0(q)}\in W$, $\rceil(q)=[N]^{k+1}_a$, and
$\mathbf{H}_k(\rceil)_q\equiv T_qN^{(k)}$,
$\widehat{\mathfrak{H}}_k(\rceil)^1\equiv\widehat{\mathfrak{H}}_k(\rceil)\cap
\widehat{\Omega}^1(\hat J^k_{m|n}(W))$. The connection $\rceil$ is
{\em flat}, i.e., with zero curvature, iff the differential ideal
$\widehat{\mathfrak{H}}_k(\rceil)$ is closed, or equivalently, iff
$\mathbf{H}_k(\rceil)$ is involutive. If $\rceil$ is a flat quantum
$(k+1)$-connection on $W$, then one has the
following:\footnote{$\mathbf{C}_{har}(\widehat{\mathfrak{H}}_k(\rceil))$
denotes the characteristic distribution of
$\widehat{\mathfrak{H}}_k(\rceil)$, and
$\mathfrak{char}(\widehat{\mathfrak{H}}_k(\rceil))$ the
corresponding vector space of its vector fields. Furthermore,
$\mathfrak{s}(\widehat{\mathfrak{H}}_k(\rceil))$ denotes the vector
space of infinitesimal symmetries of the ideal
$\widehat{\mathfrak{H}}_k(\rceil)$. $\mathfrak{v}_k(\rceil)$ is the
vector space of vertical infinitesimal symmetries of the ideal
$\widehat{\mathfrak{H}}_k(\rceil)$.}

\begin{equation}\label{quantum-contact-structure-connection}
   \left\{
\begin{array}{l}
\overline{\widehat{\mathfrak{C}}}_k(W)\subset
\widehat{\mathfrak{H}}_k(\rceil) \quad\hbox{\rm as a closed subideal}\\
\mathbf{H}_k(\rceil)\cong\mathbf{C}_{har}(\widehat{\mathfrak{H}}_k(\rceil)); \quad
\mathfrak{char}(\widehat{\mathfrak{H}}_k(\rceil))
\subset\mathfrak{s}(\widehat{\mathfrak{H}}_k(\rceil)). \\
\end{array}
\right.
\end{equation}

$\hat J^k_{m|n}(W)$ is foliated by regular solutions $Z$ such that
$\widehat{\mathfrak{H}}_k(\rceil)|_Z=0$. The leaves of the foliation
are given in implicit form by the following equations:
$f^I(x^\alpha,y^j,\dots,y^j_{\alpha_1\dots\alpha_k})=\kappa^I\in
B_k$, $ 1\le I\le p+q$, $\dim J^k_{m|n}(W)-(p|q)=m|n$, where $f^I$
represent a complete independent system of primitive integrals of
the linear system of PDEs $(\zeta_\alpha.f)=0$, $1\le\alpha\le m+n$,
where $\zeta_\alpha$ is a basis (e.g., the canonical basis) of the
horizontal distribution $\mathbf{H}_k(\rceil)$.\footnote{A (local)
section $s$ of $\pi$ identifies a flat (local) $ (k+1)$-connection
$\rceil^j_{\alpha_1\dots\alpha_{k+1}}\equiv (\partial
x_{\alpha_1}\dots\partial x_{\alpha_{k+1}}.s^j)$.} Any
$\zeta\in\mathfrak{s}((\mathbf{H}_k(\rceil))$ has the following
local representation:
\begin{equation}\label{symmetries-horizontal-k-contact-ideal}
    \left\{
\begin{array}{ll}
  \zeta=& \zeta_\alpha(X^\alpha)+\partial y_j(Y^j)+
\partial y_j^\alpha(\zeta_\alpha.Y^j)+
\partial y_j^{\alpha_1\alpha_2}(\zeta_{\alpha_1}\zeta_{\alpha_2}.Y^j)\\
 &\\
 &+\cdots+
\partial y_j^{\alpha_1\dots\alpha_k}(\zeta_{\alpha_1}\dots\zeta_{\alpha_k}.Y^j),\\
\end{array}\right.
\end{equation}

for any choice of $x^\alpha\in Q^\infty_w(U\subset \hat
J^k_{m|n}(W),A)$, ${1\le \alpha\le m+n}$, and $Y^j\in
Q^\infty_w(U\subset \hat J^k_{m|n}(W),E)$, ${1\le j\le r+s}$, such
that
\begin{equation}\label{cond-symmetries-horizontal-k-contact-ideal}
\left\{
\begin{array}{ll}
(\zeta_{\alpha_1}\dots\zeta_{\alpha_k}.Y^j)=& (\partial
y_i.\rceil^j_{\alpha_1\dots\alpha_k})Y^i+ (\partial
y_i^\gamma.\rceil^j_{\alpha_1\dots\alpha_k})(\zeta_\gamma.Y^i)\\
&\\
&+\cdots+ (\partial y_i^{\gamma_1\dots\gamma_k}.
\rceil^j_{\alpha_1\dots\alpha_k})
(\zeta_{\gamma_k}\dots\zeta_{\gamma_1}.Y^i).\\
\end{array}\right.
\end{equation}

The space $\mathfrak{s}(\mathfrak{H}_k(\rceil))$ admits the following
direct sum decomposition:
$$\mathfrak{s}(\mathfrak{H}_k(\rceil))\cong\mathfrak{d}(\mathbf{H}_k(\rceil))
\bigoplus\mathfrak{v} _k(\rceil),$$ where $\mathfrak{v} _k(\rceil)$
is the collection of all vectors of the form
\begin{equation}\label{}
\left\{
\begin{array}{ll}
\xi=&\zeta-\zeta_\alpha(X^\alpha) =\partial y_j(Y^j)+
\partial y_j^\alpha(\zeta_\alpha.Y^j)+
\partial y_j^{\alpha\beta}(\zeta_\alpha\zeta_\beta.Y^j)\\
&\\
&+\dots+
\partial y_j^{\alpha_1\dots\alpha_k}(\zeta_{\alpha_1}\dots\zeta_{\alpha_k}.Y^j),\\
\end{array}\right.
\end{equation}

for any choice of $Y^j\in Q^\infty_w(U\subset \hat J^k_{m|n}(W),E)$,
${1\le j\le r+s}$, such that conditions
$(\ref{cond-symmetries-horizontal-k-contact-ideal})$ are satisfied.
${\frak s}(\mathfrak{H}_k(\rceil))$ is a Lie algebra that admits the
subalgebra $\mathfrak{d}(\mathbf{H}_k(\rceil))$ as an ideal.
\end{lemma}

The general local expression for the symmetries of the
$(m|n)$-dimensional involutive Cartan distribution
$\mathbf{E}_{m|n}^{\infty}(W)\subset T \hat J^\infty_{m|n}(W)$, can be
also obtained by equations
{\em(\ref{symmetries-horizontal-k-contact-ideal})} with all $k>0$,
and forgetting conditions
(\ref{cond-symmetries-horizontal-k-contact-ideal}).\footnote{In fact
the Cartan distribution on $\hat J^\infty_{m|n}(W)$ can be considered an
horizontal distribution induced by the canonical connection
identified by the local canonical basis $\zeta_\alpha=\partial
x_\alpha+\sum_{|\beta|\ge 0}y^j_{\alpha\beta}\partial y_j^\beta$
just generating $\mathbf{E}^{\infty}_{m|n}(W)$.} So we get the following
expression for $\zeta\in\mathfrak{s}(\mathbf{E}^\infty_{m|n}(W))$:
\begin{equation}\label{infty-contact-symmetries}
\left\{\begin{array}{l}
  \zeta=\partial_\alpha(X^\alpha)+\sum_{r\ge 0}\partial y_j^{\alpha_1\cdots\alpha_r}(Y^j_{\alpha_1\cdots\alpha_r})\\
  \partial_\alpha=\partial x_\alpha+\sum_{r\ge 0}\partial y_j^{\alpha_1\cdots\alpha_r}(y^j_{\alpha\alpha_1\cdots\alpha_r})\\
  Y^j_{\alpha_1\cdots\alpha_r}=(\partial_{\alpha_1}\cdots\partial_{\alpha_r}.Y^j),
  \quad Y^j\in Q^\infty_w(U\subset \hat J^\infty_{m|n}(W),E),
{1\le j\le r+s}.\\
\end{array}\right.
    \end{equation}
Then the canonical splitting
$T_q\hat J^\infty_{m|n}(W)\cong(\mathbf{E}^\infty_{m|n}(W))_q\bigoplus
vT_q\hat J^\infty_{m|n}(W)$, $q\in \hat J^\infty_{m|n}(W)$, gives the following
splitting in
$\mathfrak{s}(\mathbf{E}^\infty_{m|n}(W))=\mathfrak{d}(\mathbf{E}^\infty_{m|n}(W))\bigoplus
\mathfrak{v}_\infty$, $\zeta=\zeta_o+\zeta_v$, with
$\zeta_o=\partial_\alpha(X^\alpha)$ and $\zeta_v=\sum_{r\ge
0}\partial
y_j^{\alpha_1\cdots\alpha_r}(Y^j_{\alpha_1\cdots\alpha_r})$, where $Y^j_{\alpha_1\cdots\alpha_r}$
are given in (\ref{infty-contact-symmetries}).

\begin{definition}\label{fun-stable-PDE}
Let $\hat E_k\subset \hat J^k_{m|n}(W)$, where $\pi:W\to M$ is a
fiber bundle, in the category of quantum smooth supermanifolds. We
say that $\hat E_k$ is {\em functionally stable} if for any compact
regular solution $V\subset \hat E_k$, such that $\partial
V=N_0\bigcup P \bigcup N_1$ one has quantum solutions
$\widetilde{V}\subset \hat J^{k+s}_{m|n}(W)$, $s\ge 0$, such that
$\pi_{k+s,0}(\widetilde{N}_0\DU \widetilde{N}_1)=\pi_{k,0}(N_0\DU
N_1)\equiv X\subset W$, where $\partial
\widetilde{V}=\widetilde{N}_0\bigcup
\widetilde{P}\bigcup\widetilde{N}_1$.

We call the set $\Omega[V]$ of such solutions $\widetilde{V}$ the
{\em full quantum situs} of $V$. We call also each element
$\widetilde{V}\in \Omega[V]$ a {\em quantum fluctuation} of
$V$.\footnote{Let us emphasize that to $\Omega[V]$ belong also (non
necessarily regular) solutions $V'\subset E_k$ such that $N_0'\DU
N_1'=N_0\DU N_1$, where $\partial V'=N_0'\bigcup P'\bigcup N_1'$.}
\end{definition}

\begin{definition}\label{infinitesimal}
We call {\em infinitesimal bordism} of a regular solution $V\subset
\hat E_k\subset J\hat D^k(W)$ an element $\widetilde{V}\in\Omega[V]$,
defined in the proof of Theorem \ref{deformation}. We denote by
$\Omega_0[V]\subset \Omega[V]$ the set of infinitesimal bordisms of
$V$. We call $\Omega_0[V]$ the {\em infinitesimal situs} of $V$.

Let $\hat E_k\subset \hat J^k_{m|n}(W)$, where $\pi:W\to M$ is a fiber bundle, in
the category of quantum smooth supermanifolds. We say that a regular solution
$V\subset \hat E_k$, $\partial V=N_0\bigcup P \bigcup N_1$, is {\em
functionally stable} if the infinitesimal situs $\Omega_0[
V]\subset\Omega[V]$ of $V$ does not contain singular infinitesimal
bordisms.
\end{definition}

\begin{theorem}\label{main}
Let $\hat E_k\subset \hat J^k_{m|n}(W)$, where $\pi:W\to M$ is a fiber bundle, in
the category of quantum smooth supermanifolds. If $\hat E_k$ is formally integrable
and completely integrable, then it is functionally stable as well as
Ulam-extended superstable.

A regular solution $V\subset \hat E_k$ is stable iff it is functionally
stable.
\end{theorem}

\begin{proof}
In fact, if $\hat E_k$ is formally integrable and completely integrable,
we can consider, for any compact regular solution $V\subset \hat E_k$,
its $s$-th prolongation $V^{(s)}\subset (\hat E_k)_{+s}\subset
\hat J^{k+s}_{m|n}(W)$. Since one has the following short exact sequence

\begin{equation}
\xymatrix{\Omega_{m-1|n-1}^{(\hat E_k)_{+s}}\ar[r]&\Omega_{m-1|n-1}((\hat E_k)_{+s})\ar[r]&0\\}
\end{equation}

where $\Omega_{m-1|n-1}^{(\hat E_k)_{+s}}$, (resp.
$\Omega_{m-1|n-1}((\hat E_k)_{+s})$), is the integral bordism group,
(resp. quantum bordism group),\footnote{Here the considered bordism
groups are for admissible non-necessarily closed Cauchy
hypersurfaces.} we get that there exists a solution
$\widetilde{V}\subset \hat J^{k+s}_{m|n}(W)$ such that
\begin{equation}
\left\{\begin{array}{l}
\partial \widetilde{V}=\widetilde{N}_0\bigcup\widetilde{P}\bigcup\widetilde{N}_1;\quad
\partial V^{(s)}=N_0^{(s)}\bigcup P^{(s)}\bigcup N_1^{(s)}\\
\widetilde{N}_0=N_0^{(s)};\quad
\widetilde{N}_1=N_1^{(s)}.\\
\end{array}\right.
\end{equation}
Then, as a by-product we get also:
$\pi_{k+s,0}(\widetilde{N}_0\DU\widetilde{N}_1)=\pi_{k,0}(N_0\DU
N_1)\subset W$. Therefore, $\hat E_k$ is functionally stable.
Furthermore, $\hat E_k$ is also Ulam-extended superstable, since the
integral bordism group $\Omega_{m-1|n-1}^{\hat E_k}$ for smooth solutions and
the integral bordism group $\Omega_{m-1|n-1,s}^{\hat E_k}$ for singular
solutions, are related by the following short exact sequence:
\begin{equation}
\xymatrix{0\ar[r]&\hat K_{m-1|n-1,s}^{\hat
E_k}\ar[r]&\Omega_{m-1|n-1}^{\hat E_k}\ar[r]&
\Omega_{m-1|n-1,s}^{\hat E_k}\ar[r]&0.\\}
\end{equation}
This implies that in the neighborhood of each smooth solution there
are singular solutions.

Finally a regular solution $V\subset \hat E_k$ is stable iff the set
of solutions of the corresponding linearized equation $\hat E_k[V]$
does not contains singular solutions. But this is just the
requirement that $\Omega_0[V]$ does not contains singular solutions.
Therefore, $V$ is stable if it is functionally stable and vice
versa. More precisely if $f=D^ks:X\to \hat E_k$ is a stable solution
of $\hat E_k$, then there exists an open set $W_s\subset
\underline{Sol}(\hat E_k)$ such that for any $s'\in W_s$, $s'$ is
equivalent to $s$.\footnote{Recall that $Q_w^\infty(W)$ has a
natural structure of quantum smooth supermanifold modeled on locally
convex topological vector fields. $\underline{Sol}(\hat E_k)$ is a
closed submanifold of $\underline{Sol}(\hat E_k)\subset
Q_w^\infty(W)$. (For details see ref.\cite{PRA9}.)} Let us consider
the tangent space $T_s\underline{Sol}(\hat E_k)$. One has the
following isomorphism
\begin{equation}
T_s\underline{Sol}(\hat E_k)\cong\left\{\zeta\in
(Q_w^\infty)_0((D^ks)^*vT\hat E_k)\hskip 2pt|\hskip 2pt \exists \xi\in
T_sQ_w^\infty(W), \zeta=|_k\circ D^k\xi\right\}\cong\Omega_0[V]
\end{equation}
where $|_k$ is the canonical isomorphism $J\hat D^k(s^*vTW)\cong
(D^ks)^*vTJ\hat D^k(W)$, and $V=D^ks(X)\subset \hat E_k$. Since
$W_s$ is open in $\underline{Sol}(\hat E_k)$, one has also the
following isomorphism $T_{s'}W_s\cong T_s\underline{Sol}(\hat E_k)$.
Thus also to $s'$ there correspond vector fields $\zeta\in
T_{s'}W_s$ that must be regular ones, i.e., without singular points.
Therefore $\Omega_0[V]$ cannot contain singular solutions, hence $V$
is functionally stable. Vice versa, if $V$ is functionally stable,
then we can find an open neighborhood $W_s\subset
\underline{Sol}(\hat E_k)$ built by perturbing $V$ with all the
flows induced by the regular vector fields belonging to
$\Omega_0[V]$. This set is an open set of
$W_s\subset\underline{Sol}(\hat E_k)$ since its tangent space at any
of its point $s'$ is isomorphic to $T_{s'}\underline{Sol}(\hat
E_k)$, since this last is isomorphic to $\Omega_0[V]$. Furthermore,
any two of such points of such an open set are equivalent since they
can be related both to $s$ by local diffeomorphisms. Therefore, $V$
that is functionally stable, is also stable.
\end{proof}

\begin{remark}
Let us emphasize that  the definition of functionally stable quantum
super PDE interprets in pure geometric way the definition of Ulam
superstable functional equation just adapted to PDE's.\footnote{For
informations on the Ulam stability see Refs.\cite{HYE, PRA-RAS3,
ULA}.}
\end{remark}

\begin{definition}
We say that $\hat E_k\subset J\hat D^k(W)$ is a {\em stable extended
crystal quantum super PDE} if it is an extended crystal quantum super PDE that is functionally
stable and all its regular quantum smooth solutions are (functionally)
stable.

We say that $\hat E_k\subset J\hat D^k(W)$ is a {\em stabilizable
extended crystal quantum super PDE} if it is an extended crystal quantum super PDE and to $\hat E_k$
can be canonically associated a stable extended crystal quantum super PDE
${}^{(S)}\hat E_k\subset J\hat D^{k+s}(W)$. We call ${}^{(S)}\hat E_k$ just
the {\em stable extended crystal quantum super PDE of} $\hat E_k$.
\end{definition}

We have the following criteria for functional stability of solutions
of qunatum super PDE's and to identify stable extended crystal quantum super PDE's.
\begin{theorem}{\em(Functional stability criteria).}\label{criteria-fun-stab}
Let $\hat E_k\subset J\hat D^k(W)$ be a $k$-order formally integrable
and completely integrable quantum super PDE on the fiber bundle $\pi:W\to M$.

{\em 1)} If the symbol $\hat g_k=0$, then all the quantum smooth regular
solutions $V\subset \hat E_k\subset J\hat D^k(W)$ are functionally
stable, with respect to any non-weak perturbation. So $\hat E_k$ is a
stable extended crystal.

{\em 2)} If $\hat E_k$ is of finite type, i.e., $\hat g_{k+r}=0$,
for $r>0$, then all the quantum smooth regular solutions $V\subset
\hat E_{k+r}\subset J\hat D^{k+r}(W)$ are functionally stable, with
respect to any non-weak perturbation. So $\hat E_k$ is a
stabilizable extended crystal with stable extended crystal
${}^{(S)}\hat E_k=\hat E_{k+r}$.

{\em 3)} If $V\subset(\hat E_k)_{+\infty}\subset J\hat D^\infty(W)$ is a
smooth regular solution, then $V$ is functionally stable, with
respect to any non-weak perturbation. So any formally integrable end
completely integrable quantum super PDE $\hat E_k\subset J\hat D^k(W)$, is a
stabilizable quantum extended crystal PDE, with stable quantum extended crystal PDE
${}^{(S)}\hat E_k=(\hat E_k)_{+\infty}$.
\end{theorem}

\begin{proof}
We shall use the following lemmas.

\begin{lemma}\label{eq-prolongations}
Let $\hat E_k\subset J\hat D^k(W)$ be a formally integrable and
completely integrable quantum super PDE the fiber bundle $\pi:W\to M$. Then for
any quantum smooth regular solution $s:M\to W$, one has the following
canonical isomorphism: $(\hat E_k[s])_{+h}\cong((\hat E_k)_{+h})[s]$, $\forall
h\ge 1, \infty$.
\end{lemma}

\begin{proof}
In fact one has the following commutative diagram.

\begin{equation}\label{prolongations}
\left\{\begin{array}{ll}
  (\hat E_k[s])_{+h}& =J\hat D^h((D^ks)^*vT\hat E_k)\bigcap J\hat D^{k+h}(s^*vTW) \\
  & \cong (D^{k+h}s)^*vTJ\hat D^{h}(\hat E_k)\bigcap(D^{k+h}s)^*vTJ\hat D^{k+h}(W)\\
& \cong (D^{k+h}s)^*vT\left(J\hat D^{h}(\hat E_k)\bigcap J\hat D^{k+h}(W)\right)\\
&\cong(D^{k+h}s)^*vT((\hat E_k)_{+h})=((\hat E_{k})_{+h})[s].\\
\end{array}
\right.
\end{equation}

\end{proof}

\begin{lemma}\label{symbol-prolongations}
Let $\hat E_k\subset J\hat D^k(W)$ be a formally integrable and
completely integrable PDE the fiber bundle $\pi:W\to M$. Let $\hat
g_k=0$. Then also the prolonged equations  $(\hat E_k)_{+r}$,
$\forall r\ge 1, \infty$, have their symbols zero: $(\hat
g_k)_{+r}=0$, $\forall r\ge 1, \infty$.
\end{lemma}

\begin{proof}
In fact, from the definition of symbol and prolonged symbols, it
follows that the prolonged symbols coincide with the symbols of the
corresponding prolonged equations.
\end{proof}

1) This follows from Lemma \ref{eq-prolongations} and from the fact
that if $\hat g_k=0$ is also $\hat g_k[s]=0$. This excludes that $\hat E_k[s]$
could have singular solutions. Furthermore, Lemma
\ref{symbol-prolongations} excludes also that there are singular
(nonweak) solutions in the prolonged equations $\hat E_k[s]_{+r}$,
$\forall r\ge 1, \infty$.

2) If $\hat E_k$ is of finite type, with $\hat g_{k+r}=0$, then it is also
$\hat g_{k+r}[s]=0$. Then $\hat E_{k+r}[s]$ cannot have singular (nonweak)
solutions.

3) $\hat E_\infty$ has zero symbol, hence also $\hat E_\infty[s]$ has zero
symbol and cannot have singular (nonweak) solutions.

(So the proof follows the same lines drawn for commutative PDE's.)
\end{proof}

\begin{theorem}{\em(Functional stable solutions and $(k+1)$-connections).}\label{Criterion-fun-stable-sol-conn}
Let $\hat E_k\subset \hat J^k_{m|n}(W)$ be a formally
integrable and completely integrable quantum super PDE.
Let $\rceil$ be a quantum flat $(k+1)$-connection, such that
$\rceil|_{\hat E_k}$ is a $Q^\infty_w$-section of the affine fiber
bundle $\pi_{k+1,k}:(\hat E_k)_{+1}\to \hat E_k$ . Then, the
sub-equation ${}^{\rceil}E_k\subset \hat E_k$ identified, by means
of the ideal $\mathfrak{H}(\rceil)|_{\hat E_k}$, is formally
integrable and completely integrable sub-equation with
zero symbol ${}^{\rceil}\hat g_k$. Then ${}^{\rceil}\hat E_k\subset
\hat E_k$ is functionally stable and Ulam-extended superstable.
Furthermore any regular quantum smooth solution $V\subset
{}^{\rceil}\hat E_k$ is also functionally stable in ${}^{\rceil}\hat
E_k$, with respect to any non weak perturbation.
\end{theorem}

\begin{proof}
In fact, one has the commutative diagram (\ref{functional-stable-solutions-bordism-diagram}) of exact lines.

\begin{equation}\label{functional-stable-solutions-bordism-diagram}
\xymatrix{\Omega_{m-1|n-1}^{(\hat E_k)_{+s}}\ar[d]
\ar[r]&\Omega_{m-1|n-1}((\hat E_k)_{+s})\ar[d]
\ar[r]&0\\
\Omega_{m-1|n-1}^{({}^{\rceil}\hat E_k)_{+s}}\ar[d]
\ar[r]&\Omega_{m-1|n-1}(({}^{\rceil}\hat E_k)_{+s})\ar[d]
\ar[r]&0\\
0&0&\\}
\end{equation}

Furthermore, since ${}^{\rceil}\hat g_k=0$, ${}^{\rceil}\hat E_k$ is of finite
type, hence its smooth regular solutions are functionally stable.
\end{proof}

Taking into account the meaning that connections assume in any
physical theory, we can give the following definition.

\begin{definition}
Let $\hat E_k\subset \hat J^k_{m|n}(W)$ be a formally integrable and completely
integrable quantum super PDE. Let $\rceil$ be a flat quantum $(k+1)$-connection, such that
$\rceil|_{\hat E_k}$ is a $Q^\infty_w$-section of the affine fiber bundle
$\pi_{k+1,k}:(\hat E_k)_{+1}\to \hat E_k$ . We call the couple $(\hat E_k,\rceil)$
a {\em polarized quantum super PDE}. We call also {\em polarized quantum super PDE}, a couple
$(\hat E_k,{}^{\rceil}\hat E_k)$, where ${}^{\rceil}\hat E_k\subset \hat E_k$, is
defined in Theorem \ref{Criterion-fun-stable-sol-conn}. We call
${}^{\rceil}\hat E_k$ a {\em polarization} of $\hat E_k$.
\end{definition}

\begin{cor}
Any quantum smooth regular solutions of a polarization of a polarized couple
$(\hat E_k,{}^{\rceil}\hat E_k)$, is functionally stable, with respect to any
non-weak perturbation.
\end{cor}

\begin{theorem}{\em(Finite stable quantum extended crystal super PDE's).}\label{finite-stable-extended-crystal-PDE}
Let $\hat E_k\subset J\hat D^k(W)$ be a formally integrable and
completely integrable quantum super PDE, such that the center $Z(A)$ of the quantum superalgebra $A$, model for $M$, is Noetherian. Then, under suitable {\em finite
ellipticity conditions}, there exists a stable extended crystal quantum super PDE
${}^{(S)}\hat E_k$ canonically associated to $\hat E_k$, i.e., $\hat E_k$ is a
stabilizable extended crystal.
\end{theorem}

\begin{proof}
In fact, we can use the following lemma.

\begin{lemma}{\em(Finite stability criterion).}\label{finite-stability-criterion}
Let $\hat E_k\subset J\hat D^k(W)$ be a formally integrable and
completely integrable quantum super PDE, such that the center $Z(A)$ of the quantum superalgebra $A$, model for $M$, is Noetherian. Then there exists an integer $s_0$ such
that, under suitable {\em finite ellipticity conditions}, any
regular quantum smooth solution $V\subset (\hat E_k)_{+s_0}$ is functionally
stable.
\end{lemma}

\begin{proof}
On the assumption that $Z(A)$ is Noetherian, the proof can be conduced by following the same lines of the commutative case. (See \cite{PRA26}.)
\end{proof}

Let us, now, use the hypothesis that $\hat E_k$ is formally integrable
and completely integrable. Then all its regular quantum smooth solutions are
all that of $(\hat E_k)_{+s_0}$. In fact, these are all the solutions of
$(\hat E_k)_{+\infty}\subset J\hat D^\infty(W)$. However, even if a
smooth regular solution $V\subset \hat E_k$, and their
$s_0$-prolongations, $V^{(s_0)}\subset (\hat E_k)_{+s_0}$, are equivalent
as solutions, they cannot be considered equivalent from the
stability point of view !!! In fact, $\hat E_k$ can admit singular
solutions, instead for $(\hat E_k)_{+s_0}$ these are forbidden.
Therefore, for $\hat E_k[s]$ singular perturbations are possible, i.e.
are possible infinitesimal vertical symmetries of $\hat E_k$, in a
neighborhood of the solution $s$, having singular points. Instead
for $(\hat E_k)_{+s_0}[s]$ all solutions are without singular points,
hence $s$ considered as solution of  $(\hat E_k)_{+s_0}$ necessitates to
be functionally stable.

By conclusions, $\hat E_k$, under the finite ellipticity conditions is a
stabilizable extended crystal quantum super PDE, and its stable extended crystal
quantum super PDE is ${}^{(S)}\hat E_k=(\hat E_k)_{+s_0}$, for a suitable finite number
$s_0$.
\end{proof}

\begin{remark}
With respect to a quantum frame \cite{PRA15, PRA21, PRA22}, we can
consider the perturbation behaviours of global solutions for
$t\to\infty$, where $t$ is the proper time of the quantum frame.
Then, we can talk about asymptotic stability by reproducing similar
situations for commutative PDE's. (See Refs.\cite{PRA24, PRA29}.) In
particular we can consider the concept of ''averaged stability''
also for solutions of quantum (super) PDE's. With this respect, let
us recall the following definition and properties of quantum
(pseudo)Riemannian supermanifold given in \cite{PRA15, PRA23}.
\end{remark}

\begin{definition}\cite{PRA15, PRA23}
A {\em quantum (pseudo)Riemannian supermanifold} $(M.\widehat{A})$
is a quantum supermanifold $M$ of dimension $(m|n)$ over a quantum
superalgebra $A$, endowed with a $Q^\infty_w$ section
$\widehat{g}:M\to Hom_Z(TM\otimes_ZTM;A)$ such that the induced
homomorphisms $T_pM\to(T_pM)^+$, $\forall p\in M$, are
injective.\end{definition}

\begin{proposition}\cite{PRA15, PRA23}
In quantum coordinates $\widehat{g}(p)$ is represented by a matrix
$\widehat{g}_{\alpha\beta}(p)\in\mathop{\widehat{A}}\limits^2{}_{00}(A)\times\mathop{\widehat{A}}
\limits^2{}_{10}(A)\times\mathop{\widehat{A}}\limits^2{}_{01}(A)\times\mathop{\widehat{A}}\limits^2{}_{11}(A)$.
The corresponding dual quantum metric gives
$\widehat{g}^{\alpha\beta}(p)\in\mathop{\widehat{A}}\limits^2{}^{00}(A)\times\mathop{\widehat{A}}\limits^2{}^{10}(A)\times\mathop{\widehat{A}}\limits^2{}^{01}(A)\times\mathop{\widehat{A}}\limits^2{}^{11}(A)$,
with $\mathop{\widehat{A}}\limits^2{}^{ij}(A)\equiv
Hom_Z(A;A_i\otimes_ZA_j)$, $i,j\in{\mathbb Z}_2$, such that
$\widehat{g}_{\gamma\beta}(p)\widehat{g}^{\alpha\beta}(p)=\delta^\alpha_\gamma\in\widehat{A}$,
$\widehat{g}^{\alpha\beta}(p)\widehat{g}_{\gamma\beta}(p)=\delta^\alpha_\gamma\in
Hom_Z(A\otimes_ZA;A\otimes_ZA)$.\end{proposition}

In fact we have the following definition.

\begin{definition}
Let $\hat E_k\subset J\hat D^k(W)$ be a formally integrable and completely
integrable quantum super PDE on the fiber bundle $\pi:W\to M$, and let
$V=D^ks(M)\subset E_k$ be a regular smooth solution of $\hat E_k$. Let
$\xi:M\to E_k[s]$ be the general solution of $E_k[s]$. Let us assume
that there is an Euclidean structure on the fiber of $E[s]\to M$.
Let $(\psi:\mathbb{R}\times N\to N; i:N\to M)$ be a quantum frame
\cite{PRA15, PRA21, PRA22}. Then, we say that $V$ is {\em average
asymptotic stable}, with respect to the quantum frame, if the
function of time $\mathfrak{p}[i](t)$ defined by the formula:
\begin{equation}\label{average-square-perturbation}
    \mathfrak{p}[i](t)=\frac{1}{2 vol(B_t)}\int_{B_t}i^*\xi^2\hskip 3pt \eta
\end{equation}
has the following behaviour:
$<\mathfrak{p}[i](t)>=<\mathfrak{p}[i](0)>e^{-ct}$ for some real number
$c>0$. Here $B_t\equiv N_t\bigcap supp(i^*\xi^2)$, where
$N=\bigcup_{t\in T}N_t$, is the fiber structure of $N$, over the
proper-time of the quantum frame, and the cuspidated bracket $<,>$ denotes expectation value,
(or evaluation with respect to any quantum state of the corresponding quantum (super)algebra). We call $\tau_0=1/c_0$ the {\em
characteristic stability time} of the solution $V$. If
$\tau_0=\infty$ it means that $V$ is average instable.\footnote{In
the following, if there are not reasons of confusion, we shall call
also stable solution a smooth regular solution of a quantum super PDE $\hat E_k\subset
J\hat D^k(W)$ that is average asymptotic stable. In this paper, as specified in Lemma \ref{quantum-PDE-lower-order-relation-bigraded-bordism}, we shall in general assume
that the fiber bundle $\pi:W\to M$, in the category $\mathfrak{D}_S$, has $\dim_AM=m|n$ and $\dim_BW=(m|n,r|s)$,
with respect to a quantum superalgebra $B=A\times E$, where $E$ is a quantum superalgebra that is also a $Z$-module, with $Z=Z(A)$ the center
of $A$. Therefore, $i^*\xi^2$ is a $E$-valued function on $N$ and the expectation value $<i^*\xi^2>$ is a numerical function on
$N$. Similar remarks hold for $<\mathfrak{p}[i](t)>$ and $<\mathop{\mathfrak{p}}\limits^{\bullet}[i](t)>$.}
\end{definition}

We have the following criterion of average asymptotic stability.
\begin{theorem}{\em(Criterion of average asymptotic stability).}\label{criterion-average-asymptotic-stability}
A regular global smooth solution $s$ of $\hat E_k$ is average stable,
with respect to the quantum frame $(\psi:\mathbb{R}\times N\to N;
i:N\to M)$, if the following conditions are satisfied:
\begin{equation}\label{stability-inequality}
   <\mathop{\mathfrak{p}}\limits^{\bullet}[i](t)>\hskip 3pt\le -c\hskip 3pt<\mathfrak{p}[i](t)>,\quad c\in\mathbb{R}^+, \forall t.
\end{equation}
where

\begin{equation}\label{average-square-perturbation}
    \mathfrak{p}[i](t)=\frac{1}{2\hskip 2pt vol(B_t)}\int_{B_t}i^*\xi^2\eta
\end{equation}
and
\begin{equation}\label{average-square-perturbation-rate}
  \mathop{\mathfrak{p}}\limits^{\bullet}[i](t)=\frac{1}{2\hskip 2pt vol(B_t)}
  \int_{B_t}\left(\frac{\delta i^*\xi^2}{\delta t}\right)\eta
 =\frac{1}{vol(B_t)}\int_{B_t}\left(\frac{\delta i^*\xi}{\delta t}.i^*\xi\right)\hskip 3pt \eta.
\end{equation}
Here $i^*\xi$ represents the integrable general solution of the
linearized equation $\hat E_k[s|i]$ of $\hat E_k$ at the solution $s$, and
with respect to the quantum frame. Let us denote by $c_0$ the
infimum of the positive constants $c$ such that inequality
{\em(\ref{stability-inequality})} is satisfied. Then we call
$\tau_0=1/c_0$ the {\em characteristic stability time} of the
solution $V$. If $\tau_0=\infty$ means that $V$ is
unstable.\footnote{$\tau_0$ has just the physical dimension of a
time.}

Furthermore, Let $s$ be a smooth regular solution of a formally
integrable and completely integrable quantum
super PDE $\hat E_k\subset J\hat{\it D}^k(W)$, where $\pi:W\to M$.
There exists a differential operator $\mathcal{P}[s|i](\xi)$, on
$\bar\pi:\hat E[s|i]\equiv i^*(s^*vTW)\to N$, canonically associated
to the solution $s$, and with respect to the quantum frame, such
that $s$ is average stable in $\hat E_k$, or in some suitable
prolongation $(\hat E_k)_{+h}$, $k+h=2s\ge k$, if the following
conditions are verified:

{\em(i)}  $\mathcal{P}[s|i](\xi)$ is self-adjoint (or symmetric) on
the constraint
\begin{equation}\label{constraint}
    (\hat E_k)_{(+r)}[s|i]\subset J\hat {\it D}^{k+r}(\hat E[s|i]),
\end{equation}
for some $r\ge 0$.

{\em(ii)} The smallest eigenvalue
$\overline{\lambda}_1=\overline{\lambda}_1(t)$ of
$\mathcal{P}[s|i](\xi)$ is positive for any $t\in T$ and lower
bounded: $\overline{\lambda}_1\ge\lambda_1>0$.

Furthermore, average stability can be also translated into a variational problem constrained by $(\hat E_k)_{(+h)}[s]$,
for some $h\ge 0$, such that $k+h=2s$.
\end{theorem}

\begin{proof}
We shall use Theorem \ref{deformation} and the following lemma.

\begin{lemma}{\em(Gr\"onwall's lemma)\cite{GRON}}\label{Gronwall-lemma}
Suppose $f(t)$ is a real function whose derivative is bounded
according to the following inequality: $\frac{df}{dt}\le
g(t)f+h(t)$, for some real functions $g(t)$ and $h(t)$. Then, $f(t)$
is bounded pointwise in time according to $f(t)\le
f(0)e^{G(t)}+\int_{[0,t]}e^{G(t-s)}h(s) ds$, where
$G(t)=\int_{[0,t]}g(r)dr$.
\end{lemma}

Then a sufficient condition for the solution $V$ stability, with
respect to the quantum frame, is that inequality
(\ref{stability-inequality}) should be satisfied. In fact it is
enough to use Lemma \ref{Gronwall-lemma} with $g(t)=-c$ and
$h(t)=0$, to have $<\mathfrak{p}[i](t)>=<\mathfrak{p}[i](0)>e^{-ct}$.

Furthermore, condition (\ref{stability-inequality}) is satisfied iff
\begin{equation}\label{infimum-condition}
I[\xi|i]\equiv\left<-2\int_{B_t}<\frac{\delta i^*\xi}{\delta
t}+ci^*\xi,i^*\xi>\eta\right>\hskip 3pt\ge 0,
\end{equation}
for some constant $c>0$ and for any integrable solution $i^*\xi$ of
$\hat E_k[s|i]$. (The large cuspidated brackets $<,>$ in (\ref{infimum-condition}) denote expectation value.) So the problem is converted to study the
spectrum
of the differential operator, $\mathcal{P}[s|i](\xi)\equiv
\frac{\delta i^*\xi}{\delta t}$, on $\bar\pi:\hat E[s|i]\to N$,
constrained by $(\hat E_k)_{(+r)}[s]$, for some $r\ge 0$, since
$P[s|i](\xi)$ is of order $\ge k$. If this is self-adjoint, (or
symmetric), it follows that it has real spectrum and the stability
of the solution is related to the sign of the smallest
eigenvalue.\footnote{Really it should be enough to require that
$\mathcal{P}[s|i]$ is a symmetric operator in the Hilbert space
$\mathcal{H}_t$, canonically associated to $\hat E[s]|_{B_t}$. In
fact the point spectrum $Sp(A)_p$ of a symmetric linear operator $A$
on $\mathcal{H}_t$ is real $Sp(A)_p\subset \mathbb{R}$. (This is
true also for its continuous spectrum: $Sp(A)_c\subset \mathbb{R}$.)
In our case it is enough that $\mathcal{P}[s|i]$ should symmetric on
the space of $\hat E_k[s|i]$ solutions. However, it is well known in
functional analysis that every symmetric operator has a self-adjoint
extension, on a possibly larger space \cite{DU-SH}.} If such an
eigenvalue $\overline{\lambda}_1(t)$ is positive, $\forall t\in T$,
and $\lambda_1=\inf_{t\in T}\overline{\lambda}_1(t)>0$, then the ratio
$<-\mathop{\mathfrak{p}}\limits^{\bullet}[i](t)>/<\mathfrak{p}[i](t)>$
is higher than a positive constant, hence the solution $s$ is
average stable. In fact, we get

\begin{equation}
\left\{
\begin{array}{ll}
  -\mathop{\mathfrak{p}}\limits^{\bullet}[i](t)-\lambda_1\mathfrak{p}[i](t)& =
    \int_{B_t}[(\mathcal{P}[s|i](\xi).\xi)-\lambda_1i^*\xi^2]\eta\\
  & =\int_{B_t}(\overline{\lambda}_1(t)-\lambda_1)i^*\xi^2\eta=
  (\overline{\lambda}_1(t)-\lambda_1)\int_{B_t}i^*\xi^2\eta\\
\end{array}\right.
\end{equation}
for any $t\in T$. Thus we have also, (for any $\int_{B_t}i^*\xi^2\eta\not=0$),
\begin{equation}
\frac{<-\mathop{\mathfrak{p}}\limits^{\bullet}[i](t)>}{<\mathfrak{p}[i](t)>}-\lambda_1\hskip 3pt
=(\overline{\lambda}_1(t)-\lambda_1)\hskip 3pt\ge
0\hskip 3pt \Rightarrow\quad \frac{<-\mathop{\mathfrak{p}}\limits^{\bullet}[i](t)>}{<\mathfrak{p}[i](t)>}\ge\lambda_1,\quad\forall t\in T.
\end{equation}
So condition (\ref{stability-inequality}) is satisfied, hence the
solution $s$ is average stable. In order to complete the proof of
Theorem \ref{criterion-average-asymptotic-stability}, let us
emphasize that in general $\mathcal{P}[s|i](\xi)-ci^*\xi$ is not
identified with the quantum Euler-Lagrange operator for some quantum Lagrangian. In
fact, in general, the differential order of such an operator does
not necessitate to be even. By the way, since $\hat E_k$ is assumed
formally integrable and completely integrable,
we can identify any smooth solution $V\subset \hat E_k$, with its
$h$-prolongation $V^{(h)}\subset J\hat{\it D}^{k+h}(W)$, such that
$k+h=2s$. Thus the problem of average stability can be translated in
a variational problem, constrained by solutions of $(\hat
E_k)_{+(h)}[s|i]$.
\begin{equation}\label{constrained-Euler-Lagrange-solution-perturbation-operator}
\left\{-\frac{\delta i^*\xi}{\delta t}=2\lambda(t)i^*\xi, \quad
F_\alpha^I[s|i]=0,\quad 0\le|\alpha|\le h,\hskip 2pt
k+h=2s\right\}_{t=const}
\end{equation}
on the fiber bundle $\bar\pi:\hat E[s|i]\equiv i^*(s^*vTW)\to N$.
Here $F^I[s|i]=0$ are the equations encoding $\hat E_k[s|i]$. This
can be made not only locally but also globally. In fact one has the
following lemmas. (See for the terminology \cite{PRA18, PRA32} and
references quoted there.)

\begin{lemma}{\em\cite{PRA32}}\label{constrained-variational-problems}
Let $\pi:W\to M$, a fiber bundle in the category $\mathfrak{Q}_S$,
$\dim_A M=m|n $, $\dim_BW=(m|n,r|s)$. Let $L:\hat
J^k_{m|n}(W)\to\widehat{A}$ be a $k$-order quantum Lagrangian
function and $\theta\equiv L\eta\in\widehat{\Omega}^{m+n}(\hat
J_{m|n}^k(W))$, locally given by
$\theta=L\widehat{dx^1\triangle\cdots\triangle
dx^{m+n}}=l\hat\mu_*\circ dx^1\triangle\cdots\triangle dx^{m+n}$,
where $(x^\alpha,y^j)$ are fibered quantum coordinates on $W$, and
$\hat\mu_*:\dot T^{m+n}_0(A)\to A$ is the $Z$-homomorphism induced
by the product on $A$. Then, extremals for $\theta$, constrained by
$\hat E_k$, are solutions $f:X\to \hat E_k$, with $X$ a quantum
supermanifold of dimension $m|n$ with respect to $A$, such that the
following condition is satisfied:
\begin{equation}\label{variationa-action-integral}
\left<\sum_{1\le j\le r+s}\left[\nu^j\sum_{0\le |i|\le
k}(-1)^{|i|}\partial_i\left({{\partial L}\over{\partial
y^j_i}}\right)\right]\eta,X\right>=0,
\end{equation}

for any $\nu=\nu^j\partial y_j$, solution of the linearized equation
of $\hat E_k$ at the solution $s$. In particular, if $\hat E_k=\hat
J_{m|n}^k(W))$, then extremals are solutions of the following
equation {\em(quantum Euler-Lagrange super equation)}:

\begin{equation}
\hat E[\theta]\subset \hat
J^{2k}_{m|n}(W):\quad\left\{\sum_{0\le|i|\le
k}(-1)^{|i|}\partial_i\left({{\partial L}\over{\partial
y^j_i}}\right)=0\right\}_{1\le j\le r+s}.
\end{equation}
\end{lemma}

This completes the proof. \end{proof}

\section{\bf SUPERGRAVITY YANG-MILLS PDE's IN {\boldmath$\mathfrak{Q}_S$}}
\vskip 0.5cm

In a previous paper we have encoded quantum supergravity as suitable
quantum super Yang-Mills PDE's. Nowadays, there are experimental
evidences that nuclides can be considered as quark-gluon plasmas.
For example, in order to justify spin-nuclides, it is not enough to
consider them as simply made by quarks. In fact, collective effects
appear necessary to justify nuclides properties. (See, e.g., Refs.
\cite{C-N-Z1, C-N-Z2, RON}.) With this respect, we shall encode
nuclear nuclides with suitable quantum supergravity Yang-Mills PDE's. Then stable nuclides are
stable solutions with mass-gap. An existence theorem for solutions with mass-gap is given.
 A quantum super partial differential relation, $\widehat{(Goldstone)}$, ({\em quantum Goldstone-boundary}),
 contained in a quantum super Yang-Mills equation,
having the property to create, or destroy, mass is recognized and characterized. $\widehat{(Goldstone)}$ bounds an open
constraint $\widehat{(Higgs)}\subset\widehat{(YM)}$, where live all the solutions with mass-gap.
A stable quantum super PDE, where all the smooth solutions have mass-gap and are stable in finite times,
is obtained.

Let us introduce some fundamental geometric objects to encode
quantum supergravity. (See also our previous works on this subjects
that formulate quantum supergravity in the framework of our
geometric theory of quantum super PDE's \cite{PRA13, PRA21, PRA22, PRA23,
PRA30, PRA31, PRA32}.) The first geometric object to consider is an affine
$m$-dimensional Minkowsky space-time
$(N,\mathbf{N},\alpha;\underline{g})$, where $\mathbf{N}$ is a
$m$-dimensional $\mathbb{R}$-vector space, endowed with an
hyperbolic metric $\underline{g}\in S^0_2(\mathbf{N})$, with
signature $(+,---\cdots)$. $\alpha:\mathbf{N}\times N\to N$ is the
translation mapping. Then, we consider also a {\em quantum
Riemannian (super)manifold}, $(M,\widehat{g})$, of dimension $m$,
($m|n$), with respect to a quantum algebra $A$, where
$\widehat{g}:M\to Hom_Z(\dot T^2_0M;A)$ is a quantum metric. We
shall assume that $M$ is locally {\em quantum (super) Minkowskian},
i.e., there is a $Z$-isomorphism, {\em(quantum vierbein)}:
\begin{equation}\label{quantum-vierbein}
    \hat\theta(p):T_pM\cong A\otimes_{R}\mathbf{N}, \forall p\in M,
\end{equation}
where $T_pM$ is the tangent space at $p\in M$ to $M$. Equivalently a
quantum vierbein is a section $\hat\theta:M\to
Hom_Z(TM;E)\cong\widehat{E}\otimes_{\widehat{A}}(TM)^+$, where $E$
is the trivial fiber bundle $\bar\pi:E\equiv M\times
A\otimes_{R}\mathbf{N}$. Let us denote by  $\underline{g}$ a
$A$-valued scalar product, $\underline{\hat g}$, on
$A\otimes_{R}\mathbf{N}$, given by $\underline{\hat g}(a\otimes
u,b\otimes v)=ab\hskip
  3pt\underline{g}(u,v)\in A $. By using the canonical splitting $Hom_Z(\dot T^2_0(A\otimes_{\mathbb{R}}\mathbf{N});A)\cong Hom_Z(\dot S^2_0(A\otimes_{\mathbb{R}}\mathbf{N});A)\oplus
  Hom_Z(\dot \Lambda^2_0(A\otimes_{\mathbb{R}}\mathbf{N});A)$, we get also the split representation $\underline{\hat g}=\underline{\hat g}_{(s)}+\underline{\hat g}_{(a)}$. More precisely one has
 \begin{equation}
 \underline{\hat g}_{(s)}(a\otimes u,b\otimes v)=[a,b]_+\underline{g}(u,v),\quad
 \underline{\hat g}_{(a)}(a\otimes u,b\otimes v)=[a,b]_-\underline{g}(u,v).
  \end{equation}
  Furthermore, if $(e_\alpha)$ is a basis in $\mathbf{N}$, and $(e^\beta)$ is its dual, characterized by the conditions $e_\alpha e^\beta=\delta^\beta_\alpha$, let us denote respectively by $(\widehat{1\otimes e_\alpha})$ and
  $((1\otimes e^\beta)^+)$ the induced dual bases on the spaces $\widehat{A\otimes_{\mathbb{R}}\mathbf{N}}$ and $(A\otimes_{\mathbb{R}}\mathbf{N})^+$ respectively. Then one has the following representations
 \begin{equation}\label{basis-representation-quantum-extended-minkowsky-product}
\left\{   \begin{array}{l}
     \underline{\hat g}=\underline{\hat g}_{\alpha\beta}(1\otimes e^\alpha)^+\otimes(1\otimes e^\beta)^+,\quad
     \underline{\hat g}_{\alpha\beta}\in\mathop{\widehat{A}}\limits^{2}\\
     \underline{\hat g}_{(s)}=\underline{\hat g}_{(s)}{}_{\alpha\beta}(1\otimes e^\alpha)^+\bullet(1\otimes e^\beta)^+,\quad
     \underline{\hat g}_{(s)}{}_{\alpha\beta}\in\mathop{\widehat{A}}\limits^{2}\\
     \underline{\hat g}_{(a)}=\underline{\hat g}_{(a)}{}_{\alpha\beta}(1\otimes e^\alpha)^+\triangle(1\otimes e^\beta)^+,\quad
     \underline{\hat g}_{(a)}{}_{\alpha\beta}\in\mathop{\widehat{A}}\limits^{2}.\\
     \end{array}\right.
 \end{equation}

By means of the isomorphism $\hat\theta^{\otimes}$, we can induce on $M$ a quantum
  metric, i.e., the {\em quantum Minkowskian metric} of $M$,
$\widehat{g}=\underline{\hat g}\circ\hat\theta^{\otimes}$.
Conversely any quantum metric $\widehat{g}$ on $M$, induces on the
space $A\otimes_{R}\mathbf{N}$, scalar products, for any $p\in M$:
$\hat g(p)=\widehat{g}(p)\circ(\hat\theta^{\otimes}(p))^{-1}$. As a
by-product, we get that any quantum metric $\widehat{g}$ on $M$,
induces a quantum metric on the fiber bundle $\bar\pi:E\to M$, that
we call the {\em deformed quantum metrics} of $\bar\pi:E\to M$.
Therefore, when we talk about locally Minkowskian quantum manifold
$M$, we mean that on $M$ is defined a Minkowskian quantum metric.
Since
$Hom_Z(TM;A\otimes_{\mathbb{R}}\mathbf{N})\cong\widehat{A\otimes_{\mathbb{R}}\mathbf{N}}\otimes_{\widehat{A}}(TM)^+$,
we can locally represent a quantum vierbein in the following form:
\begin{equation}\label{local-quantum-vierbein}
    \hat\theta=\widehat{1\otimes e_\beta}\bigotimes
    \hat\theta^\beta_\alpha\hskip
  3pt dx^\alpha,
\end{equation}
where $\widehat{1\otimes e_\beta}\in
Hom_Z(A;A\otimes_{\mathbb{R}}\mathbf{N})$, is the full quantum
extension of a basis $(e_\alpha)_{0\le\alpha\le m-1}$ of
$\mathbf{N}$, i.e., $\widehat{1\otimes e_\beta}(a)=a\otimes
e_\beta$. Furthermore, $\hat\theta^\beta_\alpha(p)\in \widehat{A}$.
Then, if $\zeta:M\to \widehat{TM}\equiv Hom_Z(A;TM)$ is a full
quantum vector field on $M$, locally represented by $\zeta=\partial
x_\alpha \zeta^\alpha$, we get that its local representation by
means of quantum vierbein, is given by the following formula:
\begin{equation}\label{quantum-vierbein-full-quantum-vector-field}
    \hat\theta(\zeta)=\widehat{1\otimes e_\beta}\hat\theta^\beta_\alpha\zeta^\alpha,
\end{equation}
where the product is given by composition:
\begin{equation}
\xymatrix@1@C=50pt{A\ar[r]^{\zeta^\alpha}\ar@/_1pc/[rrr]_{\sum_{\alpha,\beta}=\hat\theta(\zeta)}&
A\ar[r]^{\hat\theta_\alpha^\beta}&A\ar[r]^{\widehat{1\otimes
e_\beta}}&A\otimes_Z\mathbf{N}\\}
\end{equation}
(For abuse of notation we
can also denote $\hat\theta(\zeta)$ by $\zeta$ yet.)  Whether
$\widehat{g}=\widehat{g}_{\alpha\beta}dx^\alpha\otimes dx^\beta$, is
the quantum Minkowskian metric of $M$, then its local representation
by means of the quantum vierbein is the following:
\begin{equation}\label{quantum-vierbein-full-quantum-metric}
\left\{\begin{array}{l}
         \widehat{g}=\widehat{g}_{\alpha\omega}dx^\alpha\otimes dx^\omega\\
       \widehat{g}_{\alpha\omega}=\hat\theta^\beta_\alpha\otimes\hat\theta^\gamma_\omega\hskip 3pt\underline{g}_{\beta\gamma}=
   \underline{g}_{\beta\gamma}\hat\theta^\beta_\alpha\otimes\hat\theta^\gamma_\omega,\quad
   \widehat{g}_{\alpha\omega}(p)\in\mathop{\widehat{A}}\limits^{2},\quad \forall p\in M.\\
   \end{array}\right.
   \end{equation}
where $\hat\theta^\beta_\alpha\otimes\hat\theta^\gamma_\omega(p)$,
can be identified with
$\hat\theta^\beta_\alpha\otimes\hat\theta^\gamma_\omega(p)\in
Hom_Z(\dot T^2_0(A);\dot T^2_0(A))$. In fact, one has the following
extension $\hat\theta^\otimes$ of $\hat\theta$:
\begin{equation}\label{tensor-extension-quantum-vierbein}
  \left\{
  \begin{array}{ll}
     \hat\theta^{\otimes}& \in
   Hom_Z(TM\otimes_ZTM;(A\otimes_{\mathbb{R}}\mathbf{N})\otimes_Z(A\otimes_{\mathbb{R}}\mathbf{N})) \\
    & \cong\widehat{(A\otimes_{\mathbb{R}}\mathbf{N})\otimes_Z(A\otimes_{\mathbb{R}}\mathbf{N})}\bigotimes_{\widehat{A}}
   (TM\otimes_ZTM)^+.\\
  \end{array}
\right.
\end{equation}
Locally one can write
\begin{equation}\label{local-tensor-extension-quantum-vierbein}
   \hat\theta^\otimes=\widehat{(1\otimes e_\gamma)\otimes(1\otimes
   e_\omega)}\otimes\hat\theta^\gamma_\alpha\otimes\hat\theta^\omega_\beta\hskip
  3pt
   dx^\alpha\otimes dx^\beta.
\end{equation}
In fact, we have
\begin{equation}\label{calculation-local-tensor-extension-quantum-vierbein}
\widehat{g}(\zeta,\xi) =  \widehat{g}(\widehat{1\otimes e_\beta}\hat\theta^\beta_\alpha\zeta^\alpha,
  \widehat{1\otimes e_\gamma}\hat\theta^\gamma_\omega\xi^\omega)
  =\hat\theta^\beta_\alpha\zeta^\alpha\hat\theta^\gamma_\omega\xi^\omega\hskip
  3pt \underline{g}(e_\beta,e_\gamma)
  =\hat\theta^\beta_\alpha\zeta^\alpha\hat\theta^\gamma_\omega\xi^\omega\hskip
  3pt
  \underline{g}_{\beta\gamma}.
\end{equation}
In the particular case that $(e_\beta)$ is an orthonormal basis,
then we get the following quantum Minkowskian representation for
$\widehat{g}$
\begin{equation}\label{minkowskian-representation-full-quantum-metric}
(\widehat{g}_{\alpha\omega})=\hat\theta^\beta_\alpha\otimes\hat\theta^\gamma_\omega\hskip
  3pt
  \eta_{\beta\gamma},\quad(\eta_{\beta\gamma}) = \left(
                                   \begin{array}{cccc}
                                     1& 0&\cdots&0 \\
                                     0& -1&\cdots&0 \\
                                     \cdots&\cdots &\cdots&\cdots \\
                                     0& 0&\cdots&-1\\
                                   \end{array}
                                 \right).
\end{equation}

The splitting in symmetric and skew-symmetric part of $\widehat{g}$,
i.e.,
\begin{equation}\label{splitting-full-quantum-metric}
\widehat{g}=\widehat{g}_{(s)}+\widehat{g}_{(a)}=\widehat{g}_{\alpha\beta}\hskip
  3ptdx^\alpha\bullet
dx^\beta+\widehat{g}_{\alpha\beta}\hskip
  3ptdx^\alpha\triangle dx^\beta
\end{equation}
can be written in term of quantum vierbein in the following way:
\begin{equation}\label{symmetric-skewsymmetric-vierbein}
\left\{
\begin{array}{l}
  \hat\theta^\otimes=\hat\theta^{\odot}+ \hat\theta^{\wedge}\\
  \hat\theta^{\odot}=\widehat{(1\otimes e_\gamma)\otimes(1\otimes
   e_\omega)}\otimes\hat\theta^\gamma_\alpha\otimes\hat\theta^\omega_\beta\hskip
  3pt
   dx^\alpha\bullet dx^\beta\\
\hat\theta^{\wedge}=\widehat{(1\otimes e_\gamma)\otimes(1\otimes
   e_\omega)}\otimes\hat\theta^\gamma_\alpha\otimes\hat\theta^\omega_\beta\hskip
  3pt
   dx^\alpha\triangle dx^\beta\\
\widehat{g}_{(s)}(\zeta,\xi)=[\hat\theta^\beta_\alpha\zeta^\alpha,\hat\theta^\gamma_\omega\xi^\omega]_+\hskip
  3pt
  \underline{g}_{\beta\gamma},\Rightarrow \widehat{g}_{(s)}{}_{\alpha\omega}=
  \hat\theta^\beta_\alpha\bullet\hat\theta^\gamma_\omega\hskip 3pt\underline{g}_{\beta\gamma}\\
  \widehat{g}_{(a)}(\zeta,\xi)=[\hat\theta^\beta_\alpha\zeta^\alpha,\hat\theta^\gamma_\omega\xi^\omega]_-\hskip
  3pt
  \underline{g}_{\beta\gamma},\Rightarrow
  \widehat{g}_{(a)}{}_{\alpha\omega}=\hat\theta^\beta_\alpha\triangle\hat\theta^\gamma_\omega\hskip
  3pt\underline{g}_{\beta\gamma}.\\
\end{array}
\right.
\end{equation}

Conversely, the local expression of the quantum deformed metrics on
$\bar\pi:E\to M$, induced by a quantum metrics $\widehat{g}$ on $M$,
is given by the following formulas:

\begin{equation}\label{symmetric-skewsymmetric-vierbein}
\left\{
\begin{array}{l}
  (\hat\theta^\otimes)^{-1}=(\hat\theta^{\odot})^{-1}+ (\hat\theta^{\wedge})^{-1}\\
  (\hat\theta^{\odot})^{-1}=\partial x_\alpha\bullet\partial x_\beta\otimes\hat\theta_\gamma^\alpha\otimes
  \hat\theta_\omega^\beta\hskip 3pt
   (1\otimes e^\gamma)^+\bullet(1\otimes e^\omega)^+\\
(\hat\theta^{\wedge})^{-1}=\partial x_\alpha\triangle\partial
x_\beta\otimes\hat\theta_\gamma^\alpha\otimes
  \hat\theta_\omega^\beta\hskip 3pt
   (1\otimes e^\gamma)^+\triangle(1\otimes e^\omega)^+\\
\hat{g}(\zeta^\alpha\otimes e_\alpha,\xi^\beta\otimes
e_\beta)=\widehat{g}_{\gamma\omega}\hat\theta^\gamma_\alpha\zeta^\alpha\otimes\hat\theta^\omega_\beta\xi^\beta,\Rightarrow
\hat{g}_{\alpha\beta}=\widehat{g}_{\gamma\omega}\hat\theta^\gamma_\alpha\otimes\hat\theta^\omega_\beta\\
\hat{g}_{(s)}(\zeta^\alpha\otimes e_\alpha,\xi^\beta\otimes
e_\beta)=\widehat{g}_{(s)}{}_{\gamma\omega}[\hat\theta^\gamma_\alpha\zeta^\alpha,\hat\theta^\omega_\beta\xi^\beta]_+,\Rightarrow
\hat{g}_{(s)}{}_{\alpha\beta}=\widehat{g}_{(s)}{}_{\gamma\omega}\hat\theta^\gamma_\alpha\bullet\hat\theta^\omega_\beta\\
\hat{g}_{(a)}(\zeta^\alpha\otimes e_\alpha,\xi^\beta\otimes
e_\beta)=\widehat{g}_{(s)}{}_{\gamma\omega}[\hat\theta^\gamma_\alpha\zeta^\alpha,\hat\theta^\omega_\beta\xi^\beta]_-,\Rightarrow
\hat{g}_{(a)}{}_{\alpha\beta}=\widehat{g}_{(a)}{}_{\gamma\omega}\hat\theta^\gamma_\alpha\triangle\hat\theta^\omega_\beta.\\
\end{array}
\right.
\end{equation}
In the particular case that $\widehat{g}$ is Minkowskian, then we
can use for $\widehat{g}_{\gamma\omega}$,
$\widehat{g}_{(s)}{}_{\gamma\omega}$ and
$\widehat{g}_{(a)}{}_{\gamma\omega}$ the corresponding expressions
in (\ref{quantum-vierbein-full-quantum-metric}), and by using the
property that
$\hat\theta^\delta_\gamma\hat\theta^\gamma_\alpha=\delta^\delta_\alpha$,
we get $\hat{g}_{\gamma\omega}=\underline{\hat{g}}_{\gamma\omega}$,
$\hat{g}_{(s)}{}_{\gamma\omega}=\underline{\hat{g}}_{(s)}{}_{\gamma\omega}$
and
$\hat{g}_{(a)}{}_{\gamma\omega}=\underline{\hat{g}}_{(a)}{}_{\gamma\omega}$.
The {\em controvariant full quantum metric} $\overline{\widehat{g}}$
of $\widehat{g}:M\to Hom_Z(\dot T^2_0M;A)$ is a section
$\overline{\widehat{g}}:M\to Hom_Z(A;\dot T^2_0M)$ such that the following conditions are satisfied:
\begin{equation}\label{local-coordinates-representation-full-quantum-metric}
\scalebox{0.9}{$\left\{
\begin{array}{l}
 \overline{\widehat{g}}=\partial x_\alpha\otimes\partial
x_\beta\hat{g}^{\alpha\beta},\quad
\widehat{g}=\hat{g}_{\gamma\omega}dx^\gamma\otimes dx^\omega,\quad
 \hat{g}^{\alpha\beta}(p)\in Hom_Z(A;A\otimes_Z A),\quad  \hat{g}_{\alpha\beta}(p)\in \mathop{\widehat{A}}\limits^{2}\\
\hat{g}_{\gamma\omega}(p)\hat{g}^{\gamma\beta}(p)=\delta^\beta_\omega \in\mathbb{R}\subset\widehat{A},\quad
\hat{g}^{\gamma\beta}(p)\hat{g}_{\gamma\omega}(p)=\delta^\beta_\omega \in\mathbb{R}\subset Hom_Z(A\otimes_Z A;A\otimes_Z A).\\
\end{array}
\right.$}
\end{equation}
The products in (\ref{local-coordinates-representation-full-quantum-metric}) are meant by composition:
\begin{equation}\label{composition-local-coordinates-representation-full-quantum-metric}
\xymatrix@1@C=50pt{A\ar[r]^{\hat g{}^{\alpha\beta}(p)}\ar@/_1pc/[rr]_{\delta^\beta_\gamma}&A\otimes_ZA\ar[r]^{\hat g{}_{\alpha\gamma}(p)}&A\\}
\quad
\xymatrix@1@C=50pt{A\otimes_ZA\ar[r]^{\hat g{}_{\alpha\gamma}(p)}\ar@/_1pc/[rr]_{\delta^\beta_\gamma}&A\ar[r]^{\hat g{}^{\alpha\beta}(p)}&A\otimes_ZA.\\}
\end{equation}
In the commutative diagram (\ref{commutative-diagram-pairing}) it is shown the pairing working between the fiber bundles $(\dot T^2_0M)^+$ and $\widehat{\dot T^2_0M}$ over $M$.
\begin{equation}\label{commutative-diagram-pairing}
\xymatrix{\widehat{\dot T^2_0M}\times_M(\dot T^2_0M)^+\ar[d]\ar[r]^{<,>}&\widehat{A}\\
Hom_Z(\dot T^2_0M;\dot T^2_0M)\cong \widehat{\dot T^2_0M}\bigotimes_{\widehat{A}}(\dot T^2_0M)^+\ar[r]_(0.77){\TR}&M\times\widehat{A}\ar[u]^{pr_2}\\}
\end{equation}
In particular, one has:
\begin{equation}\label{trace-pairing}
\frac{1}{s}<\overline{\widehat{g}},\widehat{g}>=\frac{1}{s}\hat g{}^{\alpha\beta}\hat g{}_{\alpha\beta}
=\frac{1}{s}\delta^\beta_\beta 1_{\widehat{A}}=1_{\widehat{A}},
\quad s=\left\{\begin{array}{l}
                 m,\hskip 3pt  \dim_AM=m\\
                 m+n,\hskip 3pt  \dim_AM=m|n.\\
               \end{array}\right.
\end{equation}

It is direct to verify that
$\widehat{g}^{\alpha\beta}=\hat\theta^\alpha_\omega\otimes\hat\theta^\beta_\epsilon\underline{\hat{g}}^{\omega\epsilon}$
is the controvariant expression of the full quantum metric
$\widehat{g}_{\alpha\beta}=\hat\theta_\alpha^\gamma\otimes\hat\theta^\delta_\beta\underline{\hat{g}}_{\gamma\delta}$,
when $\underline{\hat{g}}^{\omega\epsilon}$ is the controvariant one
of $\underline{\hat{g}}_{\gamma\delta}$. In other words if
$\underline{\hat{g}}_{\omega\delta}\underline{\hat{g}}^{\omega\epsilon}=\delta^\epsilon_\delta$,
then
$\widehat{g}_{\alpha\gamma}\widehat{g}^{\alpha\beta}=\delta^\beta_\gamma$.
This means that the full quantum metric $\underline{\hat{g}}$,
induced on $A\otimes_{\mathbb{R}}\mathbf{N}$ by $\underline{g}$, is
not degenerate, i.e. one has the following short exact sequence:
\begin{equation}\label{short-sequence-non-degeneration}
\xymatrix{0\ar[r]&A\otimes_{\mathbb{R}}\mathbf{N}\ar[r]^{{}'\underline{\hat
g}}&(A\otimes_{\mathbb{R}}\mathbf{N})^+.\\}
\end{equation}
In fact, one can see that $\ker({}'\underline{\hat{g}})=\{0\}$.
Really, ${}'\underline{\hat{g}}(a\otimes v)(b\otimes
u)=ab\underline{g}(v,u)=0$, for all $b\in A$ and $u\in \mathbf{N}$
iff $a=0$ or $v=0$. In fact we can take $b=1$ and $u$ any vector of
$\mathbf{N}$. So, since $\underline{g}$ is not degenerate, it
follows that cannot be  $\underline{g}(v,u)=0$, for a non zero $v$,
and $\forall u\in \mathbf{N}$. The nondegeneration of
$\underline{\hat{g}}$ induces also the following isomorphism
$\widehat{A\otimes_{\mathbb{R}}\mathbf{N}}\cong(A\otimes_{\mathbb{R}}\mathbf{N})^+$.

\begin{definition}{\em(Quantum SG-Yang-Mills PDE's).}\label{quantum-sg-ym}
A {\em quantum supergravity Yang-Mills PDE}, {\em(quantum SG-Yang-Mills PDE)}, is a quantum super
Yang-Mills PDE where the quantum super Lie algebra $\mathfrak{g}$ in
the configuration bundle $\pi:W\equiv Hom_Z(TM;\mathfrak{g})\to M$
is a quantum superextension of the Poincar\'e Lie algebra and admits
the following splitting of vector spaces:
\begin{equation}\label{split-quantum-algebra}
\mathfrak{g}=\mathfrak{g}_{\circledR}+\mathfrak{g}_{\copyright}+\mathfrak{g}_{\maltese}
\end{equation}
where $\mathfrak{g}_{\circledR}=A\otimes_{\mathbb{R}}\mathbf{N}$,
(resp. $\mathfrak{g}_{\copyright}$ is the quantum superextension of
the Lorentz part of the Poincar\'e algebra). Here $A$ is a quantum
(super)algebra on which is modeled the quantum (super)manifold $M$,
and $\mathbf{N}$ is the $4$-dimensional Minkowsky vector space.
Furthermore, one assumes that there exists a non-degenerate metric
$\underline{g}$ on $\mathfrak{g}$. Taking into account the canonical
splitting:
\begin{equation}\label{split-quantum-lie-algebra}
Hom_Z(TM;\mathfrak{g})\cong Hom_Z(TM;\mathfrak{g}_{\circledR})\times
Hom_Z(TM;\mathfrak{g}_{\copyright})\times
Hom_Z(TM;\mathfrak{g}_{\maltese})
\end{equation}
we get that the fundamental field $\hat\mu:M\to W$, in a quantum
supergravity Yang-Mills PDE, admits the following canonical
splitting:\footnote{We shall use also the following notation
$\hat\mu={}_{\circledR}\hat\mu+{}_{\copyright}\hat\mu+{}_{\maltese}\hat\mu$,
that can be useful when one must add some indexes, e.g.
${}_{\circledR}\hat\mu^K_A$.}
\begin{equation}\label{splitted-fundamental-quantum-field}
\hat
\mu=\hat\mu_{\circledR}+\hat\mu_{\copyright}+\hat\mu_{\maltese}.
\end{equation}
\end{definition}

\begin{definition}
We say that $\hat\mu$ is {\em non-degenerate} if
$\hat\mu_{\circledR}$ identifies, for any $p\in M$, an isomorphism
$\hat\mu_{\circledR}(p):T_pM\cong A\otimes_{\mathbb{R}}\mathbf{N}$,
hence $\hat\mu_{\circledR}$ can be identified with a quantum vierbein
on $M$: $\hat\mu_{\circledR}\equiv\hat\theta$. Then we define
$\hat\mu_{\circledR}$, (resp. $\hat\mu_{\copyright}$, resp.
$\hat\mu_{\maltese}$), the {\em vierbein-component}, (resp. {\em
Lorentz-component}, resp. {\em deviatory-component}), of $\hat\mu$.
\end{definition}

\begin{definition}
Similarly to the quantum connection $\hat\mu$, i.e., the fundamental quantum field, we get
the following splitting of the quantum curvature:
\begin{equation}\label{split-quantum-curvature}
\hat R={}_{\circledR}\hat R+{}_{\copyright}\hat R+{}_{\maltese}\hat R.
\end{equation}
We call also {\em quantum torsion} the component ${}_{\circledR}\hat
R$ of the quantum curvature. With this respect the
${}_{\circledR}\hat\mu$-component of the quantum field equation is
also called {\em quantum torsion equation}.
\end{definition}

\begin{theorem}{\em(SG-Yang-Mills PDE's)}\label{dynamic-equation}
The dynamic equation $\widehat{(YM)}$ for a second order
SG-Yang-Mills PDE assumes in quantum coordinates the expression
reported in Tab.1. In Tab.2 there is also its unified
expression.\footnote{For example for the case of quantum
gravity-Yang-Mills PDE, corresponding to systems considered in
Example \ref{quantum-gravity-Yang-Mills}, one obtains a quantum
gravity with quantum torsion.}
\end{theorem}
$$\begin{tabular}{|c|l|} \hline \multicolumn{2}{|c|}{\bsmall Tab.1 - Quantum Dynamic Equation {\boldmath$\scriptstyle
\widehat{(YM)}\subset J\hat D^{2}(W)$} and Quantum Bianchi identity.}\\
\hline\hline {\sevenSl Quantum fields equation}&
                   $\scriptstyle  (\partial{}_{\copyright}\hat\mu^A_K.L)-
                  \partial_B(\partial
                  {}_{\copyright}\hat\mu^{A
                  B}_K.L)=0$\\
                  &\hskip 3pt\hbox{  ({\sevenSl quantum curvature-Lorentz equation})}
                  \\  &$\scriptstyle  (\partial{}_{\circledR}\hat\mu^A_K.L)
                  -\partial_B(\partial{}_{\circledR}\hat\mu^{AB}_K.L)=0$\\
                 & \hskip 3pt\hbox{  ({\sevenSl quantum curvature-verbein equation})}
                  \\  $\scriptstyle (\hat E_2)$&
                  $\scriptstyle  (\partial{}_{\maltese}\hat\mu^A_K.L)
                  -\partial_B(\partial{}_{\maltese}\hat\mu^{A
                  B}_K.L)=0$\\
                  &\hskip 3pt\hbox{  ({\sevenSl quantum curvature-deviatory equation})}\\
                  \hline{\sevenSl Quantum Bianchi identity}&
                            $\scriptstyle (\partial x_{H}.{}_{\copyright}\hat R^{K}_{AB})
                            +\frac{1}{2}{}_{\copyright}\hat C^K_{IJ}[\hat\mu
                            ^I_H,\hat R^J_{AB}]_+=0$\\
                              &$\scriptstyle (\partial x_{H}.{}_{\circledR}\hat R^{K}_{AB})
                            +\frac{1}{2}{}_{\circledR}\hat C^K_{IJ}[\hat\mu ^I_H,\hat R^J_{AB}]_+=0$
                            \\  &$\scriptstyle (\partial x_{H}.{}_{\maltese}\hat R^{K}_{AB})
                            +\frac{1}{2}{}_{\maltese}\hat C^K_{IJ}[\hat\mu ^I_H,\hat R^J_{AB}]_+=0$\\
                            \hline{\sevenSl Quamtum fields}&
                   $\scriptstyle {}_{\circledR}\hat R^K_{BA}=(\partial x_B.{}_{\circledR}\hat
\mu^K_A)+{}_{\circledR}\hat C^K_{IJ}[\hat\mu^I_B,\hat\mu^J_A]_+$\\
                  & \hskip 3pt\hbox{\rsmall({\sevenSl quantum verbein-curvature})}\\
                   &$\scriptstyle {}_{\copyright}\hat R^K_{BA}=(\partial x_B.
{}_{\copyright}\hat\mu^K_A)+{}_{\copyright}C^K_{IJ}[\hat\mu^I_B,\hat\mu^J_A]_+$\\
                   &\hskip 3pt\hbox{\rsmall({\sevenSl quantum Lorentz-curvature})}\\
                   &$\scriptstyle {}_{\maltese}\hat R^K_{BA}=(\partial x_B.{}_{\maltese}
\hat\mu^K_A)+{}_{\maltese}\hat C^K_{IJ}[\hat\mu^I_B,\hat\mu^J_A]_+$\\
                   &\hskip 3pt\hbox{\rsmall({\sevenSl quantum deviatory-curvature})}\\
                   \hline\end{tabular}$$

\begin{proof}
Let $(Z_K\in Hom_Z(A;\mathfrak{g}))$ be the basis for the quantum
extension $\widehat{\mathfrak{g}}$ of $\mathfrak{g}$. Let us denote
$(Z_K)=({}_{\circledR}Z_R,{}_{\copyright}Z_S,{}_{\maltese}Z_T)$ the
split induced by the one in (\ref{split-quantum-lie-algebra}).
Similarly we get an induced notation on the quantum structure
constants:
\begin{equation}{\label{quantum-structure-constants}}
[Z_I,Z_J]={}_{\circledR}\hat
C^K_{IJ}{}_{\circledR}Z_K+{}_{\copyright}\hat
C^K_{IJ}{}_{\copyright}Z_K+{}_{\maltese}\hat
C^K_{IJ}{}_{\maltese}Z_K.
\end{equation}
 Then this property is represented, in local quantum
coordinates, by the fact that in the following formula
\begin{equation}
    \hat\mu=Z_K\otimes\hat\mu^K_A dx^A={}_{\circledR}Z_R\otimes{}_{\circledR}\hat\mu^R_Adx^A+
    {}_{\copyright}Z_S\otimes{}_{\copyright}\hat\mu^S_Adx^A+{}_{\maltese}Z_T\otimes{}_{\maltese}\hat\mu^T_Adx^A
\end{equation}
one has $({}_{\circledR}\hat\mu^K_A(p))\in GL(\widehat{A};4)$.

 The
curvature, corresponding to $\hat\mu$, can be locally written in the
form: $\hat R=Z_K\otimes \hat R^K_{AB}dx^A\triangle dx^B$, with
$\hat R^K_{BA}=(\partial x_B \hat\mu^K_A)+\hat
C^K_{IJ}[\hat\mu^I_B,\hat\mu^J_A]_+$. The quantum curvature also
admits the splitting induced by the quantum lie algebra, as well as
the corresponding Bianchi identities. Furthermore, we shall assume a
first order quantum Lagrangian $L:J\hat D(W)\to\widehat{A}$, $L\circ
Ds=\frac{1}{2}\hat R^K_{AB}\hat R_K^{AB}$, $\forall s\in
Q^\infty_w(W)$,\footnote{The rising and lowering of indexes is
obtained by means of the fullquantum metrics $\widehat{g}$ on $M$
and $\underline{g}$ on ${\frak g}$ respectively.} The local
expression of $\widehat{(YM)}$ is given in Tab.2. Note that the
quantum super Yang-Mills equation is now
$(\partial\hat\mu^A_K.L)-(\partial_B(\partial\hat\mu^{AB}_K.L))=0$.
Furthermore, it results
${(\partial\hat\mu^A_K.L)=[\widehat{C}^H_{KR}\hat\mu^R_C,\hat
R^{[AC]}_H]_+}$ and $(\partial\hat\mu^{AB}_K.L)=\hat R^{BA}_K$. (For
more details on quantum gauge theories see also Refs.\cite{PRA9,
PRA15, PRA23, PRA30, PRA32}.)
\end{proof}

\begin{remark}
So in a quantum SG-Yang-Mills PDE, the quantum Riemannian metric
$\widehat{g}$ is not a fundamental field, but a secondary field,
obtained by means of the quantum vierbein
$\hat\theta=\hat\mu_{\circledR}$, that, instead is a fundamental
dynamic field.\footnote{Another suitable name for $\hat\theta$ could
be {\em quantum dynamical fundamental solder form}. In fact, it
solders the quantum Minkowsky vector space
$A\otimes_{\mathbb{R}}\mathbf{N}$ at all the points $p\in M$. But
the previous name is more handable.} Of course since there is a
relation one-to-one between quantum vierbein and quantum metric, on a
locally Minkowskian quantum (super)manifold, one can choice also
quantum metric as a fundamental field, instead of the quantum
vierbein. However, in a quantum SG-Yang-Mills PDE it is more natural
to adopt quantum vierbein as independent field, since it is just
enclosed in the fundamental field $\hat\mu$. The dynamic equation
are resumed in Tab.2.\end{remark}

$$\begin{tabular}{|l|c|} \hline
\multicolumn {2}{|c|}{\bsmall Tab.2 - Local expression of
{\boldmath$\scriptstyle \widehat{(YM)}\subset J\hat
D^2(W)$} and Bianchi identity {\boldmath$\scriptstyle(B)\subset J\hat D^2(W)$}.}\\
\hline \hline $\scriptstyle\hbox{\rsmall(Field equations)}\hskip 2pt
E^A_{K}\equiv-(\partial_{B}.\hat R^{BA}_K)+[\widehat
C^H_{KR}\hat\mu^R_{C},\hat
R^{[AC]}_H]_+=0$&$\scriptstyle \widehat{(YM)}$\\
\hline $\scriptstyle\hbox{\rsmall(Fields)}\hskip 2pt \hat
F^K_{A_1A_2}\equiv \hat R^K_{A_1A_2}-\left[(\partial X_{
A_1}.\hat\mu^K_{A_2})+\frac{1}{2}\widehat{C}{}^K_{IJ}[\hat\mu^I_{A_1},\hat\mu^J_{A_2}]_+\right]=0$&\\
 $\scriptstyle\hbox{\rsmall(Bianchi
identities)}\hskip 2pt B^K_{HA_1A_2}\equiv(\partial X_{H}.\hat
R^K_{A_1A_2})+\frac{1}{2} \widehat{C}{}^K_{IJ}[\bar\mu^I_{H},\hat
R^J_{A_1A_2}]_+=0$&$\scriptstyle (B)$\\
\hline \multicolumn {2}{l}{$\scriptstyle
F^K_{A_1A_2}:\Omega_1\subset J\hat
D(W)\to\mathop{\widehat{A}}\limits^2;\quad
B^K_{HA_1A_2}:\Omega_2\subset J\hat
D^2(W)\to\mathop{\widehat{A}}\limits^3;\quad E^A_{K}:\Omega_2\subset
J\hat D^2(W)\to\mathop{\widehat{A}}\limits^{3}.$}\\ \end{tabular}$$
\begin{definition}\label{quantum-graviton}
We call {\em quantum graviton} a quantum metric $\widehat{g}$
obtained by a solution $\hat\mu$ of $\widehat{(YM)}$, via the
corresponding quantum vierbein.
\end{definition}

\begin{definition}
In relation to the splitting {\em(\ref{split-quantum-curvature})},
and with respect to the possible triviality of such quantum
curvatures, we can classify solutions of $\hat E_2$, as reported in
Tab.3.

$$\begin{tabular}{|l|c|} \hline
\multicolumn {2}{|c|}{\bsmall Tab.3 - Local quantum-curvature
clasification of {\boldmath$\scriptstyle \widehat{(YM)}$} solutions.}\\
\hline \hline $\scriptstyle\hbox{\rsmall Definition}$&$\scriptstyle \hbox{\rsmall Name}$\\
\hline $\scriptstyle \hat R^K_{AB}=0$&$\scriptstyle \hbox{\rsmall quantum full-flat}$\\
\hline $\scriptstyle {}_{\circledR}\hat R^K_{AB}=0$&$\scriptstyle \hbox{\rsmall quantum torsion-free}$\\
\hline $\scriptstyle {}_{\copyright}\hat R^K_{AB}=0$&$\scriptstyle \hbox{\rsmall quantum Lorentz-flat}$\\
\hline $\scriptstyle {}_{\maltese}\hat R^K_{AB}=0$&$\scriptstyle \hbox{\rsmall quantum deviatory-flat}$\\
\hline
\end{tabular}$$
\end{definition}

\begin{theorem}{\em(Quantum Cartan geometry).}\label{sg-ym-pde-cartan-geometry}
Any non-degenerate solution $\hat\mu$ of a quantum SG-Yang-Mills PDE, identifies on the base quantum
supermanifold $M$ a quantum {\em Cartan supergeometry}, i.e., a
Cartan geometry in the category $\mathfrak{Q}_S$.
\end{theorem}

\begin{proof}
A {\em quantum Cartan supergeometry} on a the quantum supermanifold
$M$, is the natural extension, in the category $\mathfrak{Q}_S$, of {\em Cartan geometry} in the
category of smooth finite dimensional manifolds \cite{SHA}. More
precisely it is a principal fiber bundle $\pi:G\to M$, with
structure group $H$, where $G$ is a group in the category $\mathfrak{Q}_S$, such that the
following conditions are satisfied: (i) $M$ admits as quantum model
the quantum Klein geometry $(G,H)$, i.e.,
$T_xM\cong\mathfrak{g}/\mathfrak{h}$, for any $x\in M$, and $(G,H)$
is a quantum Klein model in $\mathfrak{Q}_S$, i.e., $G/H$ is an homogeneus
space in $\mathfrak{Q}_S$, with $G$ containing the subgroup $H$; (ii) there exists a section $\omega:M\to Hom_Z(TG;\mathfrak{g})$, of class $Q^\infty_w$, such that $\omega(p)$ is an isomorphism
$\omega(p):T_pG\to\mathfrak{g}$, for all $p\in G$; (iii)
$\omega(p)|_{vT_pG}:vT_pP\cong\mathfrak{h}$; (iv)
$(R_h)^*\omega=Ad(h^{-1})\omega$, $\forall h\in H$, where $R_h$
denotes right multiplication translation for $h$. Then a quantum SG-Yang-Mills PDE identifies the following quantum
Cartan supergeometry : $\pi:G\to M$, where $G$ is any quantum superextension Lie group of the Poincar\'e group, such that its quantum super Lie algebra is just $\mathfrak{g}$, and containing a subgroup $H$, with quantum super Lie algebra $\mathfrak{h}={}_{\copyright}\mathfrak{g}\oplus{}_{\maltese}\mathfrak{g}$. Then $M$ admits as model the quantum Klein geometry $(G,H)$, since one has the isomorphism
$T_xM\cong{}_{\circledR}\mathfrak{g}\cong
\mathfrak{g}/({}_{\copyright}\mathfrak{g}\oplus{}_{\maltese}\mathfrak{g})$,
$\forall x\in M$. Furthermore, any quantum fundamental field
$\hat\mu$ comes from a quantum principal connection on such a principal
bundle, as it results from the commutative
diagram (\ref{commutative-diagram-quantum-pseudo-connection}).\footnote{Let us emphasize that a section $\omega$,
considered in the above point (ii), is just a {\em quantum
pseudoconnection} in the sense introduced in Refs.\cite{PRA15,
PRA22} . (See also Refs.\cite{PRA8, PRA-REGGE} for superclassical
analogous ones.) Then, one can see that $\omega$ is just a principal
connection, ({\em Ehresmann connection} \cite{EHR}), on the
$\mathfrak{g}$-principal fiber bundle $P\equiv
G\times\mathfrak{g}\to G$. (The proof can be copied from an
intrinsic previous one given in \cite{PRA-REGGE} for the
superclassical case.) Recall that in a Cartan geometry $\pi:G\to M$, the {\em
torsion} is obtained by composition $T=\widetilde{\pi}\circ
R:\Lambda^2_0G\to\mathfrak{g}\to\mathfrak{g}/\mathfrak{h}$, where
$R$ is the curvature associated to the connection $\omega$ and
$\widetilde{\pi}$ is the canonical projection.}
\begin{equation}\label{commutative-diagram-quantum-pseudo-connection}
    \xymatrix@C=3cm{G\ar[d]_{\pi}\ar[r]^(0.4){\omega}&Hom_Z(TG;\mathfrak{g})\\
    M\ar[r]_(0.4){\hat\mu}&Hom_Z(TM;\mathfrak{g})\ar[u]^{\pi_*}\\}
\end{equation}

\end{proof}

\begin{example}{\em(Quantum gravity-Yang-Mills PDE's in
$D=4$).}\label{quantum-gravity-Yang-Mills} This is the most simple
situation where
$\mathfrak{g}={}_{\circledR}\mathfrak{g}\oplus{}_{\copyright}\mathfrak{g}$.
In such a case the quantum Klein geometry is $(P(\hat
N),SO(A\otimes_{\mathbb{R}}\mathbf{N}))$, where
$A\otimes_{\mathbb{R}}\mathbf{N}$ is the $4$-dimensional quantum
Minkowsky vector space, extension of the Minkowsky vector space
$\mathbf{N}$, with respect to the quantum algebra $A$, endowed with
the quantum metric $\widehat{\underline{g}}\equiv
1\otimes\underline{g}$, natural extension of the Minkowsky metric
$\underline{g}$ on $\mathbf{N}$.
$SO(A\otimes_{\mathbb{R}}\mathbf{N})$, is the symmetry group of
$(A\otimes_{\mathbb{R}}\mathbf{N},\widehat{\underline{g}})$. One can
see that this is just isomorphic to the classic Lorentz group
$SO(A\otimes_{\mathbb{R}}\mathbf{N})\cong SO(1,3)$. In fact, any
element $\hat\Lambda\in SO(A\otimes_{\mathbb{R}}\mathbf{N})$ is
necessarily of the type $\hat\Lambda=1_A\otimes\Lambda$, with
$\Lambda\in SO(\mathbf{N})$. In fact, by the condition
$\underline{\hat g}(a\otimes u,b\otimes
v)=\underline{\hat g}(\hat\Lambda(a\otimes
u),\hat\Lambda(b\otimes v))$, we get

\begin{equation}\label{lorentz-group1}
\left\{
\begin{array}{ll}
  \underline{\hat g}(a\otimes u,b\otimes
v)&=\underline{\hat g}(a^\alpha 1\otimes e_\alpha,b^\beta 1\otimes e_b)=
a^\alpha b^\beta \underline{g}(e_\alpha,e_\beta)
=a^\alpha b^\beta g_{\alpha\beta}\\
  &=\underline{\hat g}(\hat\Lambda(a\otimes
u),\hat\Lambda(b\otimes v))=
\underline{\hat g}(\hat\Lambda(a^\alpha 1\otimes
e_\alpha),\hat\Lambda(b^\beta 1\otimes e_b))\\
&=\underline{\hat g}(\hat\Lambda^\gamma_\alpha(a^\alpha)
1\otimes e_\gamma,\hat\Lambda^\delta_\beta(b^\beta) 1\otimes
e_\delta)\\
&=\hat\Lambda^\gamma_\alpha(a^\alpha)\hat\Lambda^\delta_\beta(b^\beta)g_{\gamma\delta}\\
\end{array}
\right.
\end{equation}
with $a^\alpha,b^\beta\in A$,
$\hat\Lambda^\gamma_\alpha\in\widehat{A}$,
$g_{\alpha\beta}\in\mathbb{R}$. Therefore we get
\begin{equation}\label{lorentz-group2}
a^\alpha b^\beta
g_{\alpha\beta}=\hat\Lambda^\gamma_\alpha(a^\alpha)\hat\Lambda^\delta_\beta(b^\beta)g_{\gamma\delta}.
\end{equation}
For the arbitrariness of $a^\alpha$ and $b^\beta$, we get that must
be also
\begin{equation}\label{lorentz-group3}
g_{\alpha\beta}=\hat\Lambda^\gamma_\alpha\hat\Lambda^\delta_\beta
g_{\gamma\delta}.
\end{equation}
Since $g_{\alpha\beta}\in\mathbb{R}$, must necessarily be
$\hat\Lambda^\gamma_\alpha\in\mathbb{R}$. This means that it is
$\hat\Lambda=1_A\otimes\Lambda$, with $\Lambda\in SO(\mathbf{N})$.
Therefore one has the following isomorphisms:
\begin{equation}\label{lorentz-group4}
SO(A\otimes_{\mathbb{R}}\mathbf{N})\cong SO(\mathbf{N})\cong SO(1,3)
\Rightarrow (\hat\Lambda^\gamma_\alpha)=(\Lambda^\gamma_\alpha).
\end{equation}
$P(\hat N)$ is the symmetry group of the $4$-dimensional affine
quantum Minkowsky space-time $(\hat
N,A\otimes_{\mathbb{R}}\mathbf{N},\underline{\widehat{g}})$. One has
the following isomorphisms:
\begin{equation}\label{quantum-poincare-group}
P(\hat N)\cong A\otimes_{\mathbb{R}}\mathbf{N}\rtimes
SO(A\otimes_{\mathbb{R}}\mathbf{N})\cong
A\otimes_{\mathbb{R}}\mathbf{N}\rtimes SO(\mathbf{N})\cong
A\otimes_{\mathbb{R}}\mathbb{R}^{1,3}\rtimes SO(1,3).
\end{equation}
One has the following short exact sequence:
\begin{equation}\label{quantum-poincare-group-extension}
\xymatrix{0\ar[r]&A\otimes_{\mathbb{R}}\mathbf{N}\ar@<0.3ex>[r]&A\otimes_{\mathbb{R}}\mathbf{N}\rtimes
SO(A\otimes_{\mathbb{R}}\mathbf{N})\ar@<0.3ex>[r]&SO(A\otimes_{\mathbb{R}}\mathbf{N})\ar[r]&0\\}
\end{equation}
The semidirect product means that the product in $P(\hat N)$ is
given by the following:
\begin{equation}\label{semidirect-product}
\left\{
\begin{array}{ll}
 (a\otimes u,\hat\Lambda).(b\otimes v,\hat\Lambda')&=(a\otimes u+\hat\Lambda(b\otimes
 v),\hat\Lambda\hat\Lambda')\\
 &=(a\otimes u+b\otimes\Lambda(v),(1\otimes\Lambda)(1\otimes\Lambda'))\\
 &=(a\otimes u+b\otimes\Lambda(v),1\otimes\Lambda\Lambda').\\
\end{array}
\right.
\end{equation}
The quantum Cartan geometry is given by the
$SO(A\otimes_{\mathbb{R}}\mathbf{N})$-principal group, in the
category $\mathfrak{Q}$, $\pi:P(\hat N)\to M$, where $M$ is a
$4$-dimensional quantum manifold, with respect to the quantum
algebra $A$, whose tangent spaces $T_pM\cong
A\otimes_{\mathbb{R}}\mathbf{N}\cong P(\hat
N)/SO(A\otimes_{\mathbb{R}}\mathbf{N})$. Therefore $M$ is locally
quantum Minkowskian. The quantum Lie algebra $\mathfrak{g}$ of
$P(\hat N)$ has the following splitting:
\begin{equation}\label{quantum-poincare-lie-algebra}
\mathfrak{g}\cong
A\otimes_{\mathbb{R}}\mathbf{N}\oplus\mathfrak{s}\mathfrak{o}(1,3)
\equiv{}_{\circledR}\mathfrak{g}\oplus{}_{\copyright}\mathfrak{g}.
\end{equation}
Let us denote by $\hat P_\mu=1\otimes P_\mu$ the generators of
${}_{\circledR}\mathfrak{g}$, where $P_\mu$ are the corresponding
translation-generators in the Poincar\'e algebra. The can see that
one has
\begin{equation}\label{quantum-translation-commutator1}
[\hat P_\mu,\hat P_\nu]=1\otimes[P_\mu,P_\nu]=1\otimes
C^\alpha_{\mu\nu}P_\alpha=C^\alpha_{\mu\nu}1\otimes
P_\alpha=C^\alpha_{\mu\nu}\hat P_\alpha
\end{equation}
with $C^\alpha_{\mu\nu}\in\mathbb{R}$. In fact, one can put on
$A\otimes_{\mathbb{R}}\mathbf{N}$ quantum coordinates $\hat
x^A:A\otimes_{\mathbb{R}}\mathbf{N}\to A$, adapted to the structure
$A\otimes_{\mathbb{R}}\mathbf{N}$, i.e., $\hat x^A(a\otimes u)=a
x^A(u)=a u^A\in A$, where $x^A:N\to \mathbb{R}$ are coordinates on
the $4$-dimensional affine Minkowsky space-time $N$. Then, for any
function $f:A\otimes_{\mathbb{R}}\mathbf{N}\to A$ of class $Q^1_w$,
one has:
\begin{equation}\label{quantum-translation-commutator2}
\left\{
\begin{array}{ll}
P_A.f=&(\partial x_A.f)=(\partial \hat x_B.f)(\partial x_A.\hat
x^B)\\
&=(\partial \hat x_B.f)(1\otimes\delta^B_A)=(\partial \hat x_A.f)\\
&=(\hat P_A.f)\\
\end{array}
\right.
\end{equation}
since $(\partial x_A.\hat x^B)=(\partial x_A.(1\otimes
x^B))=(\partial x_A.1)\otimes x^B+1\otimes(\partial
x_A.x^B)=1\otimes\delta^B_A$.

By resuming all the quantum structure constants, for this quantum
gravity Yang-Mills PDE's, are real numbers. So all the generators of
the quantum Poincar\'e algebra just coincide with the ones of the
Poincar\'e algebra. The situation is summarized in Tab.4.
$$\begin{tabular}{|l|} \hline \multicolumn {1}{|c|}{\bsmall Tab.4 -
Quantum Poincar\'e algebra in {\boldmath$\scriptstyle D=4$}.}\\
\hline\hline $\scriptstyle [J_{\alpha\beta},J_{\gamma\delta}]=\eta
_{\beta\gamma} J_{\alpha\delta}+\eta
_{\alpha\delta}J_{\beta\gamma}-\eta
_{\alpha\gamma}J_{\beta\delta}-\eta_{\beta\delta}J_{\alpha\gamma}$\\
 $\scriptstyle [P_\alpha,P_\beta]=0,\quad
[J_{\alpha\beta},P_\gamma]=\eta _{\beta\gamma}P_\alpha-\eta
_{\alpha\gamma}P_\beta$\\
\hline \multicolumn {1}{l}{\rsmall Boosts: $\scriptstyle
K_k=J_{0k}$; Rotations: $\scriptstyle J_k=\epsilon_{ijk}J^{ij}$,
$\scriptstyle i,j,k\in\{1,2,3\}$.}\\
\end{tabular}$$

\end{example}

\begin{example}{\em(Quantum $N$-superextensions of the Poincar\'e algebra in
$D=4$).}\label{quantum-n-superextensions-poincare-algebra} In $D=4$,
the usual $N$-supersymmetric extension $\mathfrak{g}$ of the
Poincar\'e algebra $\mathfrak{p}=\mathfrak{s}\mathfrak{o}(1,3)\oplus
\mathfrak{t}$, is a $\mathbb{Z}_2$-graded vector space
$\mathfrak{g}=\mathfrak{g}_0\oplus\mathfrak{g}_1$, with a graded Lie
bracket, such that $\mathfrak{g}_0=\mathfrak{p}\oplus \mathfrak{b}$,
where $\mathfrak{b}$ is a reductive Lie algebra, such that its
self-adjoint part is the tangent space to a real compact Lie
group.\footnote{A {\em reductive} Lie algebra is the sum of a
semisimple and an abelian Lie algebra. Since a {\em semisimple} Lie
algebra is the direct sum of simple algebras, i.e., non-abelian Lie
algebras, $\mathfrak{l}_i$, where the only ideals are $\{0\}$ and
$\{\mathfrak{l}_i\}$, it follows that $\mathfrak{b}$ can be
represented in the form
$\mathfrak{b}=\mathfrak{a}\oplus\sum_i\mathfrak{l}_i$.} Furthermore
$\mathfrak{g}_1=(\frac{1}{2},0)\otimes
\mathfrak{s}\oplus(0,\frac{1}{2})\otimes\mathfrak{s}^*$, where
$(\frac{1}{2},0)$ and $(0,\frac{1}{2})$ are specific representations
of the Poincar\'e algebra. Both components are conjugate to each
other under the $*$ conjugation. $\mathfrak{s}$ is a $N$-dimensional
complex representation of $\mathfrak{b}$ and $\mathfrak{s}^*$ its
dual representation.\footnote{If $\rho:\mathfrak{g}\to L(V)$ is a
representation of Lie algebra, its dual $\bar\rho:\mathfrak{g}\to
L(\bar V)$, working on the dual space $\bar V$, is defined by
$\bar\rho(u)=\overline{-\rho(u)}$, $\forall u\in\mathfrak{g}$.} Note
also that the Lie bracket for the odd part is usually denoted by
$\{,\}$ in theoretical physics. Then with such a notation one has
\begin{equation}\label{bracket-odd-part}
\{Q^i_\alpha,Q^j_\beta\}=\delta^{ij}(\gamma^\mu
C)_{\alpha\beta}P_\mu+U^{ij}(C)_{\alpha\beta}+V^{ij}(C\gamma_5)_{\alpha\beta}
\end{equation}
where $U^{ij}=-U^{ji}$, $V^{ij}=-V^{ji}$ are the $(N-1)N$ central
charges, $C$ is the (antisymmetric) charge conjugation matrix,
$(Q^i_\alpha)_{i=1,\dots,N}$, are the $N$ Majorana spinor
supersymmetry charge generators. The dynamical components
$\hat\mu^i$, $i=1,\dots,N$, of the quantum fundamental field,
corresponding to the generators $Q^i$, are called {\em quantum
gravitinos}. So in a quantum $N$-SG-Yang-Mills PDE, one
distinguishes $N$ quantum gravitino types, (and $(N-1)N$ central
charges).\footnote{Since the central charges in (\ref{bracket-odd-part})
have physical dimension of mass, they cannot be carried by massless solutions.}

The more simple case are ones with $N=1$, and $N=2$. More precisely,
for $N=1$, with $\mathfrak{b}=\mathfrak{u}(1)$ and $\mathfrak{s}$
the $1D$ representation of $\mathfrak{u}(1)$. In such a case one has
an {\em electric charge}, (i.e., $\mathfrak{u}(1)$-charge), but
there are not central charges. In Tab.5 are reported the brackets in
the case $N=1$ and $N=2$.\footnote{For the case $N=2$, and with respect to equation (\ref{bracket-odd-part}), one should also have $[Q_{\beta
i},Q_{\mu j}]=(C\gamma^\alpha)_{\beta\mu}\delta_{ij}P_\alpha+C
_{\beta\mu}\epsilon_{ij}\overline{Z}+\epsilon_{ij}(C\gamma_5)_{\beta\mu}\overline{Z}'$.
But the term $\epsilon_{ij}(C\gamma_5)_{\beta\mu}\overline{Z}'$ can
always be rotated into $C _{\beta\mu}\epsilon_{ij}\overline{Z}$ by a
chiral transformations, and therefore does not represent a further
charge.} Then a quantum superextension of $\mathfrak{g}$ is
 $A\otimes_{\mathbb{R}}\mathfrak{g}$, where $A$ is a quantum superalgebra.
 This can be taken $A\subseteq L(\mathcal{H})$, where $\mathcal{H}$ is a super-Hilbert space.
 (See also Refs.\cite{PRA23, PRA32}.) With respect to the splitting
{\em(\ref{splitted-fundamental-quantum-field})} we get:
\begin{equation}\label{splitted-fundamental-quantum-field-example1}
  (N=1):  \left\{
\begin{array}{l}
{}_{\circledR}\hat\mu=P_\alpha\hat\theta^\alpha_\gamma dx^\gamma\\
{}_{\copyright}\hat\mu=J_{\alpha\beta}\hat\omega^{\alpha\beta}_\gamma dx^\gamma\\
{}_{\maltese}\hat\mu=Q_{\alpha}\phi^{\alpha}_\gamma dx^\gamma.\\
\end{array}
    \right\};\quad(N=2):  \left\{
\begin{array}{l}
{}_{\circledR}\hat\mu=P_\alpha\hat\theta^\alpha_\gamma dx^\gamma\\
{}_{\copyright}\hat\mu=J_{\alpha\beta}\hat\omega^{\alpha\beta}_\gamma dx^\gamma\\
{}_{\maltese}\hat\mu=[\overline{Z}\hat A_\gamma +Q_{\alpha i}\phi^{\alpha i}_\gamma] dx^\gamma.\\
\end{array}
    \right\}.
\end{equation}
$$\begin{tabular}{|l|} \hline \multicolumn {1}{|c|}{\bsmall Tab.5 - {\boldmath$\scriptstyle N=1, 2$}
Super Poincar\'e algebra in {\boldmath$\scriptstyle D=4$}.}\\
\hline\hline
\multicolumn {1}{|c|}{\bsmall {\boldmath$\scriptstyle N=1$}}\\
\hline\hline $\scriptstyle [J_{\alpha\beta},J_{\gamma\delta}]=\eta
_{\beta\gamma} J_{\alpha\delta}+\eta
_{\alpha\delta}J_{\beta\gamma}-\eta
_{\alpha\gamma}J_{\beta\delta}-\eta_{\beta\delta}J_{\alpha\gamma}$\\
 $\scriptstyle [P_\alpha,P_\beta]=0,\quad
[J_{\alpha\beta},P_\gamma]=\eta _{\beta\gamma}P_\alpha-\eta
_{\alpha\gamma}P_\beta,\quad [P_\alpha,Q_{\gamma}]=0$\\
$\scriptstyle [J_{\alpha\beta},Q_{\gamma
}]=(\sigma_{\alpha\beta})^{\mu }_\gamma Q_{\mu },\quad [Q_{\beta
},Q_{\mu }]=(C\gamma^\alpha)_{\beta\mu}P_\alpha$\\
\hline\multicolumn {1}{|c|}{\bsmall {\boldmath$\scriptstyle N=2$}}\\
\hline $\scriptstyle [J_{\alpha\beta},J_{\gamma\delta}]=\eta
_{\beta\gamma} J_{\alpha\delta}+\eta
_{\alpha\delta}J_{\beta\gamma}-\eta
_{\alpha\gamma}J_{\beta\delta}-\eta_{\beta\delta}J_{\alpha\gamma}$\\
 $\scriptstyle [P_\alpha,P_\beta]=0,\quad
[J_{\alpha\beta},P_\gamma]=\eta _{\beta\gamma}P_\alpha-\eta
_{\alpha\gamma}P_\beta,\quad [P_\alpha,Q_{\gamma i}]=0$\\
$\scriptstyle [J_{\alpha\beta},Q_{\gamma
i}]=(\sigma_{\alpha\beta})^{\mu }_\gamma Q_{\mu i},\quad [Q_{\beta
i},Q_{\mu j}]=(C\gamma^\alpha)_{\beta\mu}\delta_{ij}P_\alpha+C
_{\beta\mu}\epsilon_{ij}\overline{Z},\quad [\overline{Z}.\cdot]=0$\\
\hline\hline \multicolumn {1}{|c|}{\bsmall Tab.6 -
Supersymmetric semi-simple tensor extension Poincar\'e algebra in {\boldmath$\scriptstyle D=4$}.}\\
\hline\hline
$\scriptstyle [J_{\alpha\beta},J_{\gamma\delta}]=\eta
_{\beta\gamma} J_{\alpha\delta}+\eta
_{\alpha\delta}J_{\beta\gamma}-\eta
_{\alpha\gamma}J_{\beta\delta}-\eta_{\beta\delta}J_{\alpha\gamma},\quad [P_\alpha,P_\beta]=cZ_{\alpha\beta}$\\
$\scriptstyle [J_{\alpha\beta},P_\gamma]=\eta_{\beta\gamma}P_\alpha-\eta_{\alpha\gamma}P_\beta,\quad
[J_{\alpha\beta},Z_{\gamma\delta}]=\eta_{\alpha\delta}Z_{\beta\gamma}+\eta_{\beta\gamma}Z_{\alpha\delta}-
\eta_{\alpha\gamma}Z_{\beta\delta}-\eta_{\beta\delta}Z_{\alpha\gamma}$\\
$\scriptstyle [Z_{\alpha\beta},P_\gamma]=\frac{4a^2}{c}(\eta_{\beta\gamma}P_\alpha-\eta_{\alpha\gamma}P_\beta),\quad
[Z_{\alpha\beta},Z_{\gamma\delta}]=\frac{4a^2}{c}(\eta_{\alpha\delta}Z_{\beta\gamma}+\eta_{\beta\gamma}Z_{\alpha\delta}-
\eta_{\alpha\gamma}Z_{\beta\delta}-\eta_{\beta\delta}Z_{\alpha\gamma})$\\
$\scriptstyle[J_{\alpha\beta},Q_\gamma]=-(\sigma_{\alpha\beta}Q)_\gamma,\quad
[P_\alpha,Q_\gamma]=a(\gamma_\alpha Q)_\gamma,\quad
[Z_{\alpha\beta},Q_\gamma]=-\frac{4a^2}{c}(\sigma_{\alpha\beta}Q_\gamma)$\\
$\scriptstyle
[Q_\alpha,Q_\beta]=-b[\frac{2a}{c}(\gamma^\delta C)_{\alpha\beta}P_\delta
+(\sigma^{\gamma\delta}C)_{\alpha\beta}Z_{\gamma\delta}]$\\
\hline
\end{tabular}$$
In Tab.6 are reported supersymmetric semi-simple tensor extensions of Poincar\'e algebra in $D=4$ too.
There $a$, $b$ and $c$ are constants. This algebra admits the following splitting:
$\mathfrak{s}\mathfrak{o}(3,1)\oplus \mathfrak{o}\mathfrak{s}\mathfrak{p}(1,4)$, where
$\mathfrak{s}\mathfrak{o}(3,1)$ is the $4$-dimensional Lorentz algebra and $\mathfrak{o}\mathfrak{s}\mathfrak{p}(1,4)$
is the orthosymplectic algebra. Then, by considering the quantum superextension
 $A\otimes_{\mathbb{R}}[\mathfrak{s}\mathfrak{o}(3,1)\oplus \mathfrak{o}\mathfrak{s}\mathfrak{p}(1,4)]$,
 where $A$ is a quantum superalgebra, and with respect to the splitting
{\em(\ref{splitted-fundamental-quantum-field})} we get:
\begin{equation}\label{splitted-fundamental-quantum-field-example}
    \left\{
\begin{array}{l}
{}_{\circledR}\hat\mu=P_\alpha\hat\theta^\alpha_\gamma dx^\gamma\\
{}_{\copyright}\hat\mu=J_{\alpha\beta}\hat\omega^{\alpha\beta}_\gamma dx^\gamma\\
{}_{\maltese}\hat\mu=[\overline{Z}_{\alpha\beta}\hat A_\gamma^{\alpha\beta}+Q_{\alpha i}\phi^{\alpha i}_\gamma ]dx^\gamma.\\
\end{array}
    \right.
\end{equation}
\end{example}

\begin{theorem}{\em(Quantum Levi-Civita connection and quantum Higgs fields in $\widehat{(YM)})$.}\label{quantum-Levi-Civita-Higgs-ym}
Quantum Levi-Civita connections $\hat\omega$, belonging to a
solution $\hat\mu$ of $\widehat{(YM)}$, identify covariant
derivative on the quantum metric $\underline{g}$, corresponding to
$\hat\mu$, such that ${}^{\hat\omega}\nabla\widehat{g}=0$, when
$\hat\omega$ comes from a quantum Higgs-symmetry breaking mechanism,
where $\widehat{g}$ can be identified with a quantum Higgs field.
\end{theorem}
\begin{proof}
Quantum Higgs fields and symmetry breaking are two aspects of an
unique mathematical phenomenon: reduction of a $G$-principal fiber
bundle to a closed subgroup $H\subset G$, considered in the category
$\mathfrak{Q}_S$. When this happens one has the commutative diagram (\ref{reduction-diagram}) of fiber bundles.

\begin{equation}\label{reduction-diagram}
\xymatrix{{\hat h}^*P\equiv\pi^{-1}_H(\hat
h(M))\ar[d]\ar@{=}[r]^(.7){\sim}&{}^{\hat h}P\hskip
2pt\ar[d]_{\pi_{\hat h}}\ar@{^{(}->}[r]&
P\ar[dl]_{\pi}\ar[d]_{\pi_H}\\
M\ar@{=}[r]&M\ar@/_1pc/[r]_{\hat h}&P/H\ar[l]_{\pi_{/H}}\\}
\end{equation}
$\pi_H:P\to P/H$ is a principal bundle with structure group $H$ and
$\pi_{/H}:P/H\to M$ is a fiber bundle associated to $\pi:P\to M$,
with the natural action of $G$ on the fiber type $G/H$. $\pi_{\hat
h}:{}^{\hat h}P\to M$ is a $H$-principal bundle, reduction of $P$.
Any of such reduction is identified with a global section $\hat
h:M\to P/H$ of $\pi_{/H}$, such that ${}^{\hat h}P=\pi^{-1}_H(\hat
h(M))\cong h^*P$. Any principal connection ${}^{\hat h}\hat\omega$
on ${}^{\hat h}P$ identifies a principal connection on $P$, hence a
covariant derivative ${}^{{}^{\hat h}\hat\omega}\nabla$ on sections
of $\pi_{/H}$, such that ${}^{{}^{\hat h}\hat\omega}\nabla \hat
h=0$. Conversely a principal connection $\hat\omega$ on $P$ is
projected onto ${}^{\hat h}P$ iff ${}^{{}^{\hat h}\hat\omega}\nabla
\hat h=0$. Furthermore, if the quantum Lie (super)algebra
$\mathfrak{g}$ of $G$, splits into
$\mathfrak{g}=\mathfrak{h}\bigoplus\mathfrak{a}$, where
$\mathfrak{h}$ is associated to $H$ and $\mathfrak{a}$ is a subspace
of $\mathfrak{g}$ on which $G$ acts for adjointness, then $i_{\hat
h}^*\hat\omega_{\mathfrak{h}}:{}^{\hat h}P\to Hom_Z(T{}^{\hat
h}P;\mathfrak{h})$ is a principal connection on ${}^{\hat h}P$, as
defined by the commutative diagram (\ref{connection-reduction}).
\begin{equation}\label{connection-reduction}
\xymatrix{Hom_Z(TP;\mathfrak{g})\ar[r]&Hom_Z(TP;\mathfrak{h})\ar[r]&Hom_Z(T{}^{\hat h}P;\mathfrak{h})\\
P\ar[u]_{\hat\omega}\ar@{=}[r]&P\ar[u]_{\hat\omega_{\mathfrak{h}}}&
{}^{\hat h}P\ar[l]_{i_{\hat h}}\ar[u]^{i_{\hat  h}^*\hat\omega_{\mathfrak{h}}}\\
}
\end{equation}
In the case that $P\equiv\mathcal{E}(M)$ is the principal bundle of
linear frames on $M$, with structure group $GL(4,A)$, that is
reducible to $SO(1,3)$, then global sections of
$\mathcal{E}(M)/SO(1,3)\to M$ are related to quantum metrics on $M$,
as it is shown by the commutative and exact diagram (\ref{quantum-metric-Higgs1}).
\begin{equation}\label{quantum-metric-Higgs1}
\xymatrix{0\ar[r]&\mathcal{E}(M)/SO(1,3)\ar[r]^{i}&Hom_Z(\dot T^2_0M;A)\\
&M\ar[u]_{\hat h}\ar@{=}[r]&M\ar[u]_{\widehat{g}} \\
&0\ar[u]&0\ar[u]\\ }
\end{equation}
such that
\begin{equation}\label{quantum-metric-Higgs2}
i\circ \hat h=\widehat{g}=g_{\alpha\beta}dx^\alpha\otimes
dx^\beta=\hat\theta^a_\alpha\hat\theta^b_\beta\eta_{ab}dx^\alpha\otimes
dx^\beta.
\end{equation}

 So, in this case, a quantum Higgs field is identified
with a locally Minkowskian quantum metric. Since any principal
connection ${}^{\hat h}\hat\omega$ on ${}^{\hat
h}\mathcal{E}(M)\subset\mathcal{E}(M)$ is extendable to a principal
connection $\hat\omega$ on $\mathcal{E}(M)$, and identifies a linear
quantum connection on $TM$, and on other associated vector bundles,
such that ${}^{\hat\omega}\nabla\widehat{g}=0$, when $\widehat{g}$
is just the quantum metric-Higgs field defined in
(\ref{quantum-metric-Higgs2}). This proves that quantum Levi-Civita
connections, corresponding to solutions of $\widehat{(YM)}$, such
that condition ${}^{\hat\omega}\nabla\widehat{g}=0$, are particular
cases, related to the quantum Higgs-symmetry breaking mechanisms.
\end{proof}

\begin{theorem}{\em(Quantum crystal structure of $\widehat{(YM)})$.}\label{quantum-crystal-structure-ym}
If $H_3(M;\mathbb{K})=0$ the dynamic equation $\widehat{(YM)}$ is a quantum extended crystal super PDE. Moreover, under the full admissibility hypothesis, it becomes a quantum $0$-crystal PDE.
\end{theorem}
\begin{proof}
In Refs.\cite{PRA13, PRA21} it is proved that $\widehat{(YM)}\subset
J\hat D^2(W)$ is formally integrable and also
completely integrable.\footnote{We shall assume that the $A$ is a quantum (super)algebra, over $\mathbb{K}=\mathbb{R}$, or $\mathbb{K}=\mathbb{C}$, with Noetherian center $Z\equiv Z(A)$. In general $A$ is a subalgebra of $L(\mathcal{H})$, where $\mathcal{H}$ is a (super)Hilbert space. Then $A$ is Noetherian since  $L(\mathcal{H})$ is so. In fact, $L(\mathcal{H})$ is Morita equivalent to $\mathbb{K}$. (Two rings $R$ and $S$ are {\em Morita equivalent}, $R\mathop{\thicksim}\limits^{M}S$, if there is a $R$-module $W_R$, ({\em progenerator}), such that $S\equiv End(W_R)$.) If $R$ is Noetherian, then $S\mathop{\thicksim}\limits^{M}R$ is Noetherian too. As a by-product, it follows that also the center $Z\subset A$ is a Noetherian ring. Note also that the derived quantum algebra $\widehat{A}\equiv Hom_Z(A;A)$ is a Noetherian ring. In fact, in this case $\widehat{A}\mathop{\thicksim}\limits^{M}Z$, with progenerator the $Z$-module $A$.} That proof works well also in
this situation, since it is of local nature, and remains valid also
for quantum supermanifolds that are only locally quantum
super-Minkowskian ones. Then, by using Theorem \ref{main4} we get

$\Omega_{3|3,w}^{\widehat{(YM)}}\cong\Omega_{3|3,s}^{\widehat{(YM)}}=A_0\otimes_{\mathbb{K}}H_3(W;\mathbb{K})\bigoplus
A_1\otimes_{\mathbb{K}}H_3(W;\mathbb{K})$. Since the fiber of $W$ is
contractible, we have
$\Omega_{3|3,w}^{\widehat{(YM)}}\cong\Omega_{3|3,s}^{\widehat{(YM)}}=A_0\otimes_{\mathbb{K}}H_3(M;\mathbb{K})\bigoplus
A_1\otimes_{\mathbb{K}}H_3(M;\mathbb{K})$. Thus, under the condition
that $H_3(M;\mathbb{K})=0$, one has
$\Omega_{3|3,w}^{\widehat{(YM)}}\cong\Omega_{3|3,s}^{\widehat{(YM)}}=0$,
hence $\widehat{(YM)}$ becomes a quantum extended crystal super PDE.
This is surely the case when $M$ is globally quantum super
Minkowskian. (See Refs.\cite{PRA14, PRA16, PRA21, PRA22, PRA32}.) In
such a case one has
$\Omega_{3|3}^{\widehat{(YM)}}=K_{3|3}\widehat{(YM)}$, where
\begin{equation}
K_{3|3}\widehat{(YM)}\equiv\left\{[N]_{\overline{\widehat{(YM)}}}\in\Omega_{3|3}\widehat{(YM)}
\left|\begin{array}{l}
{N=\partial V,\hskip 2pt\hbox{\rm for some (singular)}}\\
\hbox{\rm $(4|4)$-dimensional quantum}\\
\hbox{\rm supermanifold $V\subset W$ }\\
\end{array}\right.\right\}.
\end{equation}

So $\widehat{(YM)}$ is not a quantum $0$-crystal super PDE. However,
if we consider admissible only integral boundary manifolds, with
orientable classic limit, and with zero characteristic quantum
supernumbers, ({\em full admissibility hypothesis}), one has:
$\Omega_{3|3}^{\widehat{(YM)}}=0$, and $\widehat{(YM)}$ becomes a
quantum $0$-crystal super PDE. Hence we get existence of global
$Q^\infty_w$ solutions for any boundary condition of class
$Q^\infty_w$.

With respect to the commutative exact diagram in
(\ref{Reinhart-bordism-groups-relation}) we get the exact
commutative diagram (\ref{Yang-Mills-Reinhart-bordism-groups-relation}).

\begin{equation}\label{Yang-Mills-Reinhart-bordism-groups-relation}
\xymatrix{0\ar[r]&K^{\widehat{(YM)}}_{3|3;2}\ar[r]&\Omega_{3|3}^{\widehat{(YM)}}\ar[r]&
\mathop{\Omega}\limits_c{}_{6}^{\widehat{(YM)}}\ar[d]\ar[r]\ar[dr]&0&\\
&0\ar[r]&K^\uparrow_{6}\ar[r]&\Omega^\uparrow_{6}\ar[r]&
\Omega_{6}\ar[r]& 0\\}
\end{equation}
Taking into account the result by Thom on the unoriented cobordism
groups \cite{THO1}, we can calculate
$\Omega_6\cong\mathbb{Z}_2\bigoplus\mathbb{Z}_2\bigoplus\mathbb{Z}_2$.
Then, we can represent $\Omega_6$ as a subgroup of a $3$-dimensional
crystallographic group type $[G(3)]$. In fact, we can consider the
amalgamated subgroup $D_2\times\mathbb{Z}_2\star_{D_2}D_4$, and
monomorphism $\Omega_6\to D_2\times\mathbb{Z}_2\star_{D_2}D_4$,
given by $(a,b,c)\mapsto(a,b,b,c)$. Alternatively we can consider
also $\Omega_6\to D_4\star_{D_2}D_4$.  (See Appendix C in
\cite{PRA25} for amalgamated subgroups of $[G(3)]$.) In any case the
crystallographic dimension of $\widehat{(YM)}$ is $3$ and the
crystallographic space group type are $D_{2d}$ or $D_{4h}$ belonging
to the tetragonal syngony. (See Tab.6 in \cite{PRA25} and, for
further informations, \cite{HAH}.)
\end{proof}

\begin{example}
If $M$ is the quantum superextension, with respect to the quantum
superalgebra $A$, of the $4$-dimensional affine Minkowsky
space-time, then $\widehat{(YM)}$ is a quantum extended crystal PDE.
\end{example}

\begin{theorem}{\em(Quantum crystal structure of $\widehat{(YM)}[i]$).}\label{quantum-crystal-structure-ym[i]}
The observed dynamic equation $\widehat{(YM)}[i]$, by means of a quantum relativistic frame, is a quantum extended
crystal super PDE. Moreover, under the full admissibility hypothesis, it becomes a quantum $0$-crystal super PDE.
\end{theorem}
\begin{proof}
The evaluation of $\widehat{(YM)}$ on a macroscopic shell
$i(M_C)\subset M$ is given by the equations reported in Tab.7.

$$\begin{tabular}{|l|c|} \hline
\multicolumn {2}{|c|}{\bsmall Tab.7 - Local expression of
{\boldmath$\scriptstyle \widehat{(YM)}[i]\subset J\hat D^2(i^*W)$}
and Bianchi idenity {\boldmath$\scriptstyle  (B)[i]\subset
J\hat D^2(i^*W)$}.}\\
\hline\hline $\scriptstyle \hbox{\rsmall(Field equations)}\hskip 2pt
(\partial_{\alpha}.\tilde R^{K\alpha\beta})+[\widehat
C^K_{IJ}\tilde\mu^I_{\alpha},\tilde R^{J\alpha\beta}]_+
=0$&$\scriptstyle \widehat{(YM)}[i]$ \\
\hline $\scriptstyle \hbox{\rsmall(Fields)}\hskip 2pt \bar
R^K_{\alpha_1\alpha_2}=(\partial
\xi_{[\alpha_1}.\tilde\mu^K_{\alpha_2]})+\frac{1}{2}\widehat{C}{}^K_{IJ}\tilde\mu^I_{[\alpha_2}\tilde\mu^J_{\alpha_1]}
$&{}\\
$\scriptstyle \hbox{\rsmall(Bianchi identities)}\hskip 2pt (\partial
\xi_{[\gamma}.\tilde R^K_{\alpha_1\alpha_2]})+\frac{1}{2}
\widehat{C}{}^K_{IJ}\tilde\mu^I_{[\gamma}\tilde
R^J_{\alpha_1\alpha_2]}=0$&$\scriptstyle (B)[i]$\\
\hline
\end{tabular}$$

\vskip 0.5cm This equation is also formally integrable
and completely integrable. Furthermore, the
$3$-dimensional integral bordism group of $\widehat{(YM)}[i]$ and
its infinity prolongation $\widehat{(YM)}[i]_ +\infty$ are trivial,
under the full admissibility hypothesis:
$\Omega_3^{\widehat{(YM)}[i]}\cong\Omega_3^{\widehat{(YM)}[i]_
+\infty}\cong 0$. So equation $\widehat{(YM)}[i]\subset J\hat
D^2(i^*W)$ becomes a quantum $0$-crystal super PDE and it admits global
(smooth) solutions for any fixed time-like $3$-dimensional (smooth)
boundary conditions.
\end{proof}

\begin{proposition}
The quantum vierbein curvature ${}_{\circledR}\hat R$ identifies, by
means of the quantum vierbein $\hat\theta$ a quantum field $\hat
S:M\to Hom_Z(\dot\Lambda^2_0M;TM)$, that we call {\em quantum
torsion}, associated to $\hat\mu$. In quantum coordinates one can
write
\begin{equation}\label{quantum-torsion}
    \hat S=\partial x_C\otimes\hat S^C_{AB}dx^A\triangle dx^B,\quad \hat
    S^C_{AB}=\hat\theta^C_K{}_{\circledR}\hat R^K_{AB}.
\end{equation}

Furthermore, with respect to a quantum relativistic frame $i:N\to
M$, the quantum torsion $\hat S$ identifies a $A$-valued
$(1,2)$-tensor field on $N$, $\widetilde{S}\equiv i^*\hat S:N\to
A\otimes_{\mathbb{R}}\Lambda^0_2N\otimes_{\mathbb{R}}TN$, that we
call {\em quantum torsion of the observed solution}.
\end{proposition}

\begin{proof}
In fact $\hat S=\hat\theta^{-1}_*\circ {}_{\circledR}\hat R$, i.e.,
the diagram (\ref{commutative-diagram-quantum-torsion}) is commutative.
\begin{equation}\label{commutative-diagram-quantum-torsion}
\xymatrix@1@C=50pt{&Hom_Z(\dot\Lambda^2_0M;{}_{\circledR}\mathfrak{g})\ar[d]_{\hat\theta^{-1}_*}\\
M\ar[ur]^{{}_{\circledR}\hat R}\ar[r]_(.3){\hat
S}&Hom_Z(\dot\Lambda^2_0M;TM)\\}
\end{equation}
where $\hat\theta^{-1}_{*}(p)\equiv Hom_Z(1_{\dot\Lambda^2_0(T_pM)};\hat\theta^{-1}(p))$, $\forall p\in M$.
\end{proof}

\begin{definition}\label{spin-observed-solution}
Furthermore, we say that an observed solution has a {\em quantum
spin}, if the observed solution has an observed torsion
\begin{equation}\label{observed-quantum-torsion}
\widetilde{S}\equiv i^*\hat S=\partial
x_\gamma\otimes\sum_{0\le\alpha<\beta\le
3}\widetilde{S}^\gamma_{\alpha\beta}dx^\alpha\wedge dx^\beta:N\to
A\otimes_{\mathbb{R}}\mathbf{N}\otimes_{\mathbb{R}}\Lambda^0_2(N)\cong
      A\otimes_{\mathbb{R}}\Lambda^0_2(N)\otimes_{\mathbb{R}}TN
\end{equation}
with
$\widetilde{S}^\gamma_{\alpha\beta}(p)=-\widetilde{S}^\gamma_{\beta\alpha}(p)\in
A$, $p\in N$, that satisfies the following conditions,
{\em(quantum-spin-conditions)}:
\begin{equation}\label{quantum-spin-conditions}
\begin{array}{l}
 \left\{
    \begin{array}{l}
      \widetilde{S}=\widetilde{s}\otimes\dot\psi\\
      \widetilde{s}=\sum_{0\le\alpha<\beta\le
3}\widetilde{s}_{\alpha\beta}dx^\alpha\wedge dx^\beta:N\to
A\otimes_{\mathbb{R}}\Lambda^0_2N,\\
\hskip 1cm\widetilde{s}_{\alpha\beta}(p)=-\widetilde{s}_{\beta\alpha}(p)\in A, \hskip 2pt p\in N,\\
\dot\psi\rfloor \widetilde{S} =0\\
    \end{array}
    \right\}\\
    \Downarrow\\
    \left\{
    \begin{array}{l}
      \widetilde{S}^\lambda_{\alpha\beta}=\widetilde{s}_{\alpha\beta}\dot\psi^\lambda\\
      \widetilde{S}^\lambda_{\alpha\beta}\dot\psi^\alpha =0\\
    \end{array}
    \right\}.\\
 \end{array}
\end{equation}
where $\dot\psi$ is the velocity field on $N$ of the time-like foliation representing the quantum relativistic frame on $N$. When conditions {\em(\ref{quantum-spin-conditions})}
are satisfied, we say that the solution considered  admits a {\em quantum spin-structure}, with respect to the quantum relativistic frame. We call $\widetilde{s}$ the {\em quantum $2$-form spin} of the observed solution. Let $\{\xi^\alpha\}_{0\le\alpha\le 3}$ be coordinates on $N$, adapted to the quantum relativistic frame. Then one has the following local representations:
  \begin{equation}\label{local-representations-observed-quantum-spin-observed-quantum-torsion}
 \left\{
    \begin{array}{l}
      \widetilde{s}=\widetilde{s}_{ij}d\xi^i\wedge d\xi^j\\
    \end{array}
    \right\}.
\end{equation}
We define {\em quantum spin-vector-field} of the observed solution
\begin{equation}\label{quantum-spin-vector-field}
\widetilde{\underline{s}}=<\epsilon,\widetilde{S}>=[\epsilon_{\mu\nu\lambda\rho}\dot\psi^\mu\widetilde{s}^{\nu\lambda}]d\xi^\rho
\equiv \widetilde{s}_kd\xi^k\hskip
2pt\Rightarrow
\widetilde{s}=\partial\xi_i\widetilde{s}_kg^{ki}=\partial\xi_i\widetilde{s}^i
\end{equation}
where
$\epsilon_{\mu\nu\lambda\rho}=\sqrt{|g|}\delta_{\mu\nu\lambda\rho}^{0123}$
is the completely antisymmetric tensor density on $N$. One has
$\widetilde{s}^\rho(p)\in A$, $p\in N$. The classification of the
observed solution on the ground of the spectrum of
$|\widetilde{s}|^2\equiv\widetilde{s}^\rho\widetilde{s}_\rho$, and
its ({\em quantum helicity}), i.e., component $\widetilde{s}_z$, is
reported in Tab.8.\footnote{In particle physics with the term {\em
helicity} one usually means the component of the spin along the
momentum. In this way one should obtain only the contribution by the
''intrinsic spin'', decoupled by the angular momentum. However, in
our geometric formulation one talks of spin of an observed solution,
since the macroscopic model may be inadequate, and also
undesirable.}
\end{definition}

$$\begin{tabular}{|l|l|} \hline
\multicolumn {2}{|c|}{\bsmall Tab.8 - Local quantum
spectral-spin-classification of
{\boldmath$\scriptstyle \widehat{(YM)}[i]$} solutions.}\\
\hline\hline
$\scriptstyle \hbox{\rsmall Definition}$ &$\scriptstyle \hbox{\rsmall Name}$ \\
\hline $\scriptstyle
Sp(|\widetilde{s}(p)|^2)\subset\mathfrak{b}\equiv\{\hbar^2s(s+1)|
s\in\mathbb{N}\equiv\{0,1,2,\dots\}\},\quad
Sp(\widetilde{s}_z(p))\subset
 \mathfrak{c}$&$\scriptstyle \hbox{\rsmall bosonic-polarized}$\\
$\scriptstyle
Sp(|\widetilde{s}(p)|^2)\subset\mathfrak{f}\equiv\{\hbar^2s(s+1)|
s=\frac{2n+1}{2},n\in\mathbb{N}\equiv\{0,1,2,\dots\}\}, \quad
Sp(\widetilde{s}_z(p))\subset
 \mathfrak{c}$&$\scriptstyle \hbox{\rsmall fermionic-polarized}$\\
$\scriptstyle Sp(|\widetilde{s}(p)|^2)\cap\mathfrak{b}=Sp(|\widetilde{s}(p)|^2)\cap\mathfrak{f}=\varnothing$&$\scriptstyle \hbox{\rsmall unpolarized}$\\
$\scriptstyle Sp(|\widetilde{s}(p)|^2)\cap\mathfrak{b}\not=\varnothing, \hbox{\rsmall  and/or } Sp(|\widetilde{s}(p)|^2)\cap\mathfrak{f}\not=\varnothing$&$\scriptstyle \hbox{\rsmall mixt-polarized}$\\
\hline \multicolumn {2}{l}{\rsmall
$\scriptstyle|\widetilde{s}(p)|^2\equiv\widetilde{s}^\rho(p)\widetilde{s}_\rho(p),\hskip
2pt p\in N$. $\scriptstyle \mathfrak{c}\equiv\{\hbar
m_s|m_s=-s,-s+1,\cdots,s-1,s\}$. $\scriptstyle \widetilde{s}_z$ quantum helicity}\\
\multicolumn {2}{l}{\footnotesize  $\scriptstyle s=$ spin quantum
number; $\scriptstyle m_s=$ spin orientation quantum number.}\\
\end{tabular}$$

\begin{remark}
May be useful to emphasize that since here the quantum spin content
of an observed solution is a geometric object of local nature, does
not necessitate that its property, in relation to the classification
in Tab.7, should be constant in any part of the solution. In fact,
e.g., for solutions representing particles reactions, can happen
that the spin-content changes during the reaction process. For
example the spins of composite particles, such as protons and atomic
nuclides, are not just the sum of their constituent particles. The
usual justification is ascribed to the contribution of the total
angular momenta, in some mechanical simulations. On the other hand
this game does not work well. For example for the proton there are
experimental evidences that such mechanical models are inadequate.
(See e.g. \cite{RON}.) Some approaches similar to graviton-quark-gluon plasma
appear more appropriate. But it is unknown if gluons have spin...
So, instead to insist with naif and unjustified mechanical models,
it is more appropriate describe quantum particles like $p$-chains
solutions of suitable quantum Yang-Mills PDE's.\footnote{With a
language nearer to physicists, $p$-chains can be called {\em
$p$-dimensional extendons}.}
\end{remark}

In some previous works we have proved the following important result.

\begin{theorem}{\em(Obstruction-mass-gap-existence).\cite{PRA16, PRA23, PRA32}}\label{obstruction-mass-gap}
A quantum full-flat solution cannot have mass-gap. In order that an
observed solution admits mass-gap it is enough that the following
conditions should be satisfied:
\begin{equation}\label{mass-gap}
\left\{
\begin{array}{l}
\TR(\widetilde{R}^K_{\beta\alpha}(p)\widetilde{R}^{\beta\alpha}_H(p)-
\widetilde{\mu}^K_{\alpha\beta}(p)\widetilde{R}^{\beta\alpha}_H)(p)\not=0;\\
         \exists\frac{1}{\TR(\widetilde{R}^K_{\beta\alpha}(p)\widetilde{R}^{\beta\alpha}_H(p)-
         \widetilde{\mu}^K_{\alpha\beta}(p)\widetilde{R}^{\beta\alpha}_H)(p)}\in\widehat{A}.\\
       \end{array}\right\}_{p\in N}
\end{equation}
\end{theorem}

\begin{theorem}{\em(Local mass-formula).}\label{local-mass-formula-theorem}
If an observed solution of $\widehat{(YM)}$ is with mass-gap, one has the following local {\em mass-formula}:
\begin{equation}\label{local-mass-formula}
m(p)={}_{\circledR}m(p)+{}_{\copyright}m(p)+{}_{\maltese}m(p),\hskip 2pt p\in N.
\end{equation}
We call ${}_{\circledR}m(p)$, (resp. ${}_{\copyright}m(p)$, resp. ${}_{\maltese}m(p)$)), the {\em local torsion mass}, (resp. {\em local Lorentz-mass}, resp. {\em local deviatory-mass}), of the observed solution.\footnote{Note that the formula (\ref{local-mass-formula}) interprets the local mass $m(p)$ as a local-charge, emphasizing the role playied by the different charge-components of the systems.}
\end{theorem}

\begin{proof}
In fact, the splitting on the quantum Lie superalgebra $\mathfrak{g}$,
induces the following splitting on the quantum Hamiltonian observed
by the quantum relativistic frame:

\begin{equation}\label{split-hamiltonian}
H={}_{\circledR}H+{}_{\copyright}H+{}_{\maltese}H \left\{
\begin{array}{l}
  {}_{\circledR}H={}_{\circledR}\widetilde{R}^K_{\beta\alpha}{}_{\circledR}\widetilde{R}^{\beta\alpha}_K
  -{}_{\circledR}\widetilde{\mu}^K_{\alpha\beta}{}_{\circledR}\widetilde{R}^{\beta\alpha}_K\\
{}_{\copyright}H=
{}_{\copyright}\widetilde{R}^K_{\beta\alpha}{}_{\copyright}\widetilde{R}^{\beta\alpha}_K-
{}_{\copyright}\widetilde{\mu}^K_{\alpha\beta}{}_{\copyright}\widetilde{R}^{\beta\alpha}_K\\
 {}_{\maltese}H={}_{\maltese}\widetilde{R}^K_{\beta\alpha}{}_{\maltese}\widetilde{R}^{\beta\alpha}_K-
 {}_{\maltese}\widetilde{\mu}^K_{\alpha\beta}{}_{\maltese}\widetilde{R}^{\beta\alpha}_K.\\
 \end{array}\right.
\end{equation}
Therefore, if the observed solution is with mass gap, we get that spectrum of the hamiltonian has a splitting induced by the corresponding hamiltonian splitting (\ref{split-hamiltonian}).
\end{proof}

\begin{theorem}{\em(Observed objects and splitting formulas).}\label{observed-objects-and-splitting-formulas-theorem}
To any observed solution of $\widehat{(YM)}$ one can associate the following geometric objects:
\begin{equation}\label{associated-observed-geometric-objects}
\left\{
\begin{array}{l}
\hbox{\rm (quantum-electric-charge-field):}\\
\hat E=\dot\psi\rfloor \tilde R=Z_K\otimes \hat E^K_\alpha
dx^\alpha=Z_K\otimes\tilde R^K_{\alpha\beta}\dot\psi^\beta
dx^\alpha:N\to
\mathfrak{g}\bigotimes_{\mathbb{R}}T^*N\\
\\
  \hbox{\rm (quantum-magnetic-charge-field):}\\
  \hat B=\dot\psi\rfloor(\star\tilde R)=Z_K\otimes\hat H^K_\mu
  dx^\mu=Z_K\otimes
\epsilon^{\kappa\lambda}_{\mu\nu}\tilde R^K_{\kappa\lambda}\dot\psi^\nu dx^\mu:N\to
\mathfrak{g}\bigotimes_{\mathbb{R}}T^*N. \\
\end{array}
\right.
\end{equation}
In adapted coordinates $\{\xi^\alpha\}_{0\le\alpha\le 3}$, to the
relativistic quantum frame, $\hat E$ and $\hat B$ have the following
representations:
\begin{equation}\label{frame-adapted-coordinates-associated-observed-geometric-objects}
\left\{
\begin{array}{l}
\hat E=Z_K\otimes \hat E^K_i d\xi^i=Z_K\otimes\tilde
R^K_{0i}d\xi^i,\hskip 2pt \hat E_i^K(p)\in A,\hskip 2pt \forall p\in N\\
\\
  \hat B=Z_K\otimes\hat B^K_i
  d\xi^i=Z_K\otimes
\epsilon^{rs}_{m0}\tilde
R^K_{rs}d\xi^m,\hskip 2pt \hat B_i^K(p)\in A,\hskip 2pt \forall p\in N. \\
\end{array}
\right.
\end{equation}
This means that in the quantum relativistic frame system $\hat E$
and $\hat B$ are $A$-valued space-like objects, whose components belong to the quantum superalgebra $A$. Therefore, their quantum content is given by the spectra of $\hat E_i^K(p)$ and $\hat B_i^K(p)$, for $p\in N$.

One has the following splittings:
\begin{equation}\label{splittings-associated-observed-geometric-objects}
\left\{
\begin{array}{l}
\hat E={}_{\circledR}\hat E+{}_{\copyright}\hat E+{}_{\maltese}\hat E\\
\hat B={}_{\circledR}\hat B+{}_{\copyright}\hat B+{}_{\maltese}\hat B.\\
\end{array}
\right.
\end{equation}
We call ${}_{\circledR}\hat E$, (resp. ${}_{\copyright}\hat E $,
resp. ${}_{\maltese}\hat E$)), the {\em torsion-quantum
electric-charge-field}, (resp. {\em Lorentz-quantum
electric-charge-field}, resp. {\em deviatory-quantum
electric-charge-field}), of the observed solution. Similarly, we
call ${}_{\circledR}\hat B$, (resp. ${}_{\copyright}\hat B$, resp.
${}_{\maltese}\hat B$)), the {\em torsion-quantum
magnetic-charge-field}, (resp. {\em Lorentz-quantum
magnetic-charge-field}, resp. {\em deviatory-quantum
magnetic-charge-field}), of the observed solution.

If the solution has a spin-structure, then ${}_{\circledR}\hat E=0$
and ${}_{\circledR}\hat B$ is related to the quantum spin-vector
field $\widetilde{s}$ by the following formula:
\begin{equation}\label{spin-vector-field-vs-magnetic-charge-solution-spin-structure}
\hat\theta^\gamma_K{}_{\circledR}\hat
B^K_\mu=-\dot\psi^\gamma\widetilde{s}_\mu\hskip 2pt\Rightarrow
\hbox{\rm (in frame-adapted coordinates):  }
\hat\theta^0_K{}_{\circledR}\hat B^K_i=-\widetilde{s}_i.
\end{equation}
Therefore, in an observed solution with spin-structure, it is
recognized that the opposite {\em torsion-quantum
magnetic-charge-field} identifies a distribution of quantum
spin-vector fields, $\widetilde{s}^k=g^{ki}\widetilde{s}_i$, with
$\widetilde{s}_i$  given in
{\em(\ref{spin-vector-field-vs-magnetic-charge-solution-spin-structure})}.
Thus, ${}_{\circledR}\hat B$ determines the
spectral-spin-classification of the solution, according to Tab.7.
\end{theorem}

\begin{proof}
It follows directly from the definitions and previous results. Let
us only to emphasize that in coordinates on $N$, adapted to the
frame $\dot\psi=\partial\xi_0$, and that the metric
$g=g_{\alpha\beta}d^\alpha\otimes dx^\beta$ and its controvariant
form $\bar g=g^{\alpha\beta}\partial x_\alpha\otimes
\partial x_\beta$ assume respectively the following forms:
\begin{equation}\label{metric-adapted-coordinate}
(g_{\alpha\beta})=\left(
  \begin{array}{cccc}
    1 & 0 & 0 & 0 \\
    0 & g_{11} & g_{12} & g_{13} \\
    0 & g_{21} & g_{22} & g_{23} \\
    0 & g_{31} & g_{32} & g_{33} \\
  \end{array}
\right)\quad (g^{\alpha\beta})=\left(
  \begin{array}{cccc}
    1 & 0 & 0 & 0 \\
    0 & g^{11} & g^{12} & g^{13} \\
    0 & g^{21} & g^{22} & g^{23} \\
    0 & g^{31} & g^{32} & g^{33} \\
  \end{array}
\right),\quad g^{ij}g_{jk}=\delta^i_k.
\end{equation}
Furthermore taking into account that
$\epsilon^{\kappa\lambda}_{\mu\nu}=g^{\kappa\alpha}g^{\lambda\beta}\epsilon_{\alpha\beta\mu\nu}$,
with
$\epsilon_{\alpha\beta\mu\nu}=\sqrt{|g|}\delta^{0123}_{\alpha\beta\mu\nu}$,
the formulas
(\ref{frame-adapted-coordinates-associated-observed-geometric-objects})
follow directly. When the observed solution admits spin-structure,
then, since
$\widetilde{S}^\lambda_{\alpha\beta}=\hat\theta^\lambda_K{}_{\circledR}\widetilde{R}^K_{\alpha\beta}$,
we get also
$0=\widetilde{S}^\lambda_{\alpha\beta}\dot\psi^\alpha=\hat\theta^\lambda_K{}_{\circledR}\hat
E^K_\alpha$. Thus we have
$0=\hat\theta^H_\lambda\hat\theta^\lambda_K{}_{\circledR}\hat
E^K_\alpha=\delta^H_K{}_{\circledR}\hat E^K_\alpha$. Therefore must
necessarily be ${}_{\circledR}\hat E^H_\alpha=0$.

Finally we can write $\hat\theta^\gamma_K{}_{\circledR}\hat
B^K_\mu=\epsilon^{\kappa\lambda}_{\mu\nu}\hat\theta^\gamma_K{}_{\circledR}\widetilde{R}^K_{\kappa\lambda}\dot\psi^\nu$.
If the observed solution admits a spin-structure, then
$\hat\theta^\gamma_K{}_{\circledR}\widetilde{R}^K_{\kappa\lambda}=\widetilde{s}_{\kappa\lambda}\dot\psi^\gamma$.
Therefore, we can write
$$\left\{
\begin{array}{ll}
  \hat\theta^\gamma_K{}_{\circledR}\hat
B^K_\mu&=\epsilon^{\kappa\lambda}_{\mu\nu}\widetilde{s}_{\kappa\lambda}\dot\psi^\gamma\dot\psi^\nu
  =g^{\alpha\kappa}g^{\beta\lambda} \epsilon_{\alpha\beta\mu\nu}
  \widetilde{s}_{\kappa\lambda}\dot\psi^\gamma\dot\psi^\nu\\
  & =-\epsilon_{\nu\alpha\beta\mu}
  \widetilde{s}^{\alpha\beta}\dot\psi^\nu\dot\psi^\gamma=-\widetilde{s}_\mu\dot\psi^\gamma.\\
\end{array}
\right.$$
\end{proof}

\begin{theorem}{\em(Stability properties of $\widehat{(YM)}$).}\label{stability-properties-ym}
$\widehat{(YM)}$ is a functional stable quantum super PDE. In general a global solution $V\subset\widehat{(YM)}$ is unstable and the corresponding observed solution, by means of a quantum relativistic frame, is unstable at finite times. However, $\widehat{(YM)}$ admits a stable quantum extended crystal PDE. There all the observed smooth solutions are stable at finite times. Furthermore, to study the asymptotic stability of a global solution $V\subset\widehat{(YM)}$, with respect to a quantum relativistic frame, we can apply Theorem \ref{criterion-average-asymptotic-stability}.
\end{theorem}

\begin{proof}
$\widehat{(YM)}$ is a functional stable quantum super PDE since it is completely integrable and formally integrable.
(See Theorem \ref{criteria-fun-stab}). For the same reason it admits
$\widehat{(YM)}_{+\infty}\subset J\hat{\it D}^\infty(W)$ like stable quantum extended crystal PDE. Furthermore,
since its symbol $\hat g_2$ is not trivial, any global solution $V\subset\widehat{(YM)}$ can be unstable, and the
corresponding observed solution, can appear unstable in finite times. However, global smooth solution, result stable
in finite times in $\widehat{(YM)}_{+\infty}$. Finally the asymptotic stability study of global solutions of
$\widehat{(YM)}$, with respect to a quantum relativistic frame, can be performed by means of
Theorem \ref{criterion-average-asymptotic-stability}, since, for any section $s:M\to W$, on the fibers of
$\hat E[s]\to M$ there exists a non-degenerate scalar product.
In fact, $\hat E[s]\cong W$, as $W$ is a vector bundle over $M$. Furthermore, for any section $s$, we can identify on
$M$ a non degenerate metric $\widehat{g}$,\footnote{It should be more precise to denote $\widehat{g}$ with the
symbol $\widehat{g}[s]$, since it is identified by means of the section $s$.} that beside the rigid
metric $\underline{g}$ on $\mathfrak{g}$, identifies a non-degenerate metric on each fiber $\hat E[s]_p\cong W_p=Hom_Z(T_pM;\mathfrak{g})$, $\forall p\in M$. In fact we get $\hat\xi(p)\cdot\hat\xi(p)'=\underline{g}_{KH}\widehat{g}^{AB}(p)\xi^K_A(p)\otimes\xi'{}^H_B(p)\in A$.
\end{proof}

\begin{example}{\em(Stable nuclear-charged plasmas and nuclides).}\label{plasmas-nuclides}
Of particular relevance are solutions of $\widehat{(YM)}$ that
encode nuclear-charged plasmas, or nuclides, dynamics. These are
described by solutions that, when observed by means of a quantum
relativistic frame have at any $t\in T$, i.e., frame-proper time,
compact sectional support $B_t\subset N$. The {\em global mass} at
the time $t$, i.e. the evaluation
$m_t=\int_{B_t}m(t,\xi^k)\sqrt{\det(g_{ij})}d\xi^1\wedge
d\xi^2\wedge d\xi^3$ of such mass on the space-like section $B_t$,
gives the global mass-contents of the nuclear-plasmas or nuclides,
in their ground-eigen-states, at the proper time $t$. Whether such
solutions are asymptotically stable, then they interpret the meaning
of  stable nuclear-plasmas or nuclides.
\end{example}

The following theorem shows how thermodynamic functions can be associated to solutions of $\widehat{(YM)}$.

\begin{definition}{\em(Thermodynamic states).}\label{observed-quantum-system-thermodynamics}
A {\em thermodynamic state} is characterized by a set of independent
parameters $(s,c_\alpha)_{1\le\alpha\le n}$, that identify the
internal energy: $e=e(s,c_\alpha,a^i)$, where $a^i$ identifies the
$i$-th material particle, $s$ is the {\em specific entropy}, (with
physical dimension energy /temperature mass), and $c_\alpha$ are
electric or mechanical parameters. The system is said to be {\em
thermodynamically homogeneous} if the following equations are
satisfied: $(\partial a_\alpha\cdot e)=0$.\footnote{One has
$\theta=\theta(s,c^\alpha,a^i)$,
$t^\alpha=t^\alpha(s,c^\alpha,a^i)$. Of course one can choose
$\theta$ and $c^\alpha$ as independent thermodynamic functions: $s=
s(\theta,c^\alpha,a^i)$, $t^\alpha=t^\alpha(\theta,c^\alpha,a^i)$
and $e=e(\theta,c^\alpha,a^i)$.} Important thermodynamic functions
and thermodynamic equations are given in Tab.9.\footnote{With
respect to the quantum relativistic frame it is identified on $N$ a
time-like vector field $v$, such that for any scalar
$\mathbb{R}$-valued thermodynamic function $f$ one has $\frac{\delta
f}{\delta t}=(\partial t.f)+v^k(\partial x_k.f)$.}
\end{definition}

$$\begin{tabular}{|c|c|}\hline
\multicolumn{2}{|c|}{\bsmall Tab.9 - Distinguished thermodynamic functions and equations.}\\
\hline\hline {\rsmall Name}&{\rsmall Definition}\\
\hline {\rsmall temperature}&$\scriptstyle\theta\equiv(\partial s\cdot e)$\\
\hline {\rsmall thermodynamic stresses}&$ \scriptstyle t^\alpha\equiv(\partial c^\alpha\cdot e)$\\
\hline {\rsmall thermodynamic pressure}&$ \scriptstyle  p\equiv t^1=-(\partial({{1}\over{\rho }})\cdot e) \hskip 2pt(*)$\\
\hline {\rsmall chemical potential}&$ \scriptstyle t^\alpha \equiv
p^\alpha\equiv(\partial e^\alpha\cdot e)\hskip 2pt(**) $\\
\hline {\rsmall specific heats}&$ \scriptstyle C \equiv{{\delta
q}\over{\delta\theta}}={{1}\over{\dot\theta}}(\dot e-t^\alpha{\dot c}_\alpha )$\\
\hline {\rsmall latent heats}&$ \scriptstyle {\widehat e}^\alpha
\equiv{{\delta q}\over{\delta t^\alpha}}={{1}\over{{\dot t}^\alpha}}
(\dot e-t^\beta{\dot c}_\beta) $\\
&$ \scriptstyle{\widehat m}_\alpha\equiv{{\delta q}\over{\delta
C_\alpha }} = {{1}\over{{\dot c}_\alpha }} ( \dot e-t^\beta{\dot c}_\beta)$\\
\hline {\rsmall specific heat at $\scriptstyle c^\alpha=$ cost.}&$
\scriptstyle C_{c}\equiv(\partial\theta\cdot e)_{{c^\alpha }} $\\
\hline {\rsmall specific heat at $\scriptstyle t^\alpha=$ cost.}&$
\scriptstyle C_t=C_c+[( \partial c_\beta\cdot e)_{\theta
,{c_\alpha}}-t^\beta]( \partial\theta\cdot c_\beta)_t $\\
&$\scriptstyle \gamma\equiv C_t/C_{c}$\\
\hline {\rsmall free energy density (Helmholtz density)}&$
\scriptstyle f=e-\theta s $\\
\hline
{\rsmall hentalpy density}&$\scriptstyle h=e-t^\alpha c_\alpha$\\
\hline
{\rsmall free hentalpy density (Gibbs function)}&$\scriptstyle g=h-\theta s$\\
\hline
{\rsmall Gibbs-equations of first type}&
$\scriptstyle de = \theta  ds + t^\alpha  dc_\alpha,\hskip 3pt \hbox{\rsmall fixed $\scriptstyle a$.} $\\
\hline {\rsmall Gibbs-equations of second type}& $\scriptstyle
{{\delta e}\over{\delta t}}=\theta {{\delta
s}\over{\delta t}}+t^\alpha{{\delta c_\alpha }\over{\delta t}}$\\
\hline
{\rsmall Thermodynamic evolution along thermodynamic curve}&
$\scriptstyle C={{1}\over{d\theta }}(de-t^\beta  dc_\beta)=\theta(\partial\theta\cdot s) $\\
&$\scriptstyle {\widehat e}_\alpha={{1}\over{dt^\alpha }}(de-t^\beta dc_\beta)=\theta(\partial t_\alpha\cdot s)$\\
&$\scriptstyle {\widehat m}_\alpha={{1}\over{dc_\alpha }}(de-t^\beta dc_\beta)=\theta(\partial c_\alpha\cdot s)$\\
\hline
\multicolumn{2}{l}{\rsmall $\scriptstyle (*)$\hskip 2pt One denotes $\scriptstyle -t^1$ if $\scriptstyle c_1 = {{1}\over{\rho }}$.}\\
\multicolumn{2}{l}{\rsmall $\scriptstyle (**)$\hskip 2pt
$\scriptstyle e_\alpha=c_\alpha$ is the concentration of the
$\scriptstyle \alpha$-component in a system with different components.}\\
\multicolumn{2}{l}{\rsmall An observed thermodynamic curve,
$\scriptstyle \lambda=\lambda(t)$ is identified at fixed $\scriptstyle a^\alpha$.}\\
\end{tabular}$$

\begin{theorem}{\em(Thermodynamics of $\widehat{(YM})$).}\label{thermodynamics-ym}
There exists a canonical way to characterize thermodynamic states of
observed solutions of $\widehat{(YM})$. Thermodynamic functions and
equations of observed solutions of $\widehat{(YM})$ can be encoded
as scalar-valued differential operator on the extended fiber bundle
$W[i]\times_NT^0_0N\cong
W[i]\times\mathbb{R}\equiv\widetilde{W[i]}$, over $N$, whose
sections $(\widetilde{\mu},\beta)$ over $N$, represent an {\em observed
quantum fundamental field}, $\widetilde{\mu}$, and a function $\beta:N\to \mathbb{R}$,
{\em thermal function}. If $\beta=\frac{1}{\kappa_B\theta}$,
where $\kappa_B$ is the Boltzmann constant and $\theta$ is the
temperature, then the observed solution encodes a system in
equilibrium with a heat bath.
\end{theorem}

\begin{proof}
Let $\widetilde{H}$ be the Hamiltonian corresponding to an observed
solution of $\widehat{(YM)}$.\footnote{Let us recall that $\widetilde{H}$ is a $\widehat{A}$-valued
function on the $4$-dimensional space-time $N$, considered in the quantum relativistic frame.} Let us denote by $E\in
Sp(\widetilde{H})$. If $N(E)=\TR\delta(E-\widetilde{H})$ denotes the
degeneracy of $E$, let us define {\em local partition function} of
the observed solution the Laplace transform of the degeneracy
$N(E)$, with respect the spectrum $Sp(\widetilde{H})$ of
$\widetilde{H}$. We get
\begin{equation}\label{partition-function1}
    \left\{
\begin{array}{ll}
  Z(\beta)&= \int_{Sp(\widetilde{H})}e^{-\beta E}N(E) dE\\
  & = \int_{Sp(\widetilde{H})} e^{-\beta E}\TR\delta(E-\widetilde{H}) dE \\
  &=\TR e^{-\beta \widetilde{H} }.\\
\end{array}
    \right.
\end{equation}
So we get the following formula
\begin{equation}\label{partition-function2}
Z(\beta)=\TR e^{-\beta \widetilde{H} }
\end{equation}
where $\beta$ is the Laplace transform variable and it does not
necessitate to be interpreted as the ''inverse temperature'', i.e.,
$\beta=\frac{1}{\kappa_B\theta}$, where $\kappa_B$ is the
Boltzmann's constant.\footnote{If $\beta=\frac{1}{\kappa_B\theta}$
then the system encoded by the observed solution of
$\widehat{(YM)}$, is in equilibrium with a heat bath (canonical
system).} Note that all above objects are local functions on the
space-time $N$. The same holds for $\beta$. We can interpret $Z(\beta)$ as a normalization factor
for the local probability density
\begin{equation}\label{probability-density}
P(E)=\frac{1}{Z}N(E)e^{-\beta E}
\end{equation}
that
the system, encoded by the observed solution, should assume the
local energy $E$, with degeneration $N(E)$. In fact we have:
\begin{equation}\label{partition-function-normalization-factor}
1=\int_{Sp(\widetilde{H})} P(E)
dE=\frac{1}{Z}\int_{Sp(\widetilde{H})} N(E)e^{-\beta
E}dE=\frac{Z}{Z}.
\end{equation}
As a by-product we get that the {\em local average energy}
$<E>\equiv e$ can be written, by means of the partition function, in
the following way:

\begin{equation}\label{partition-function-average-energy1}
e=-(\partial\beta\ln Z).
\end{equation}
In fact, one has
\begin{equation}\label{partition-function-average-energy2}
    \left\{
\begin{array}{ll}
  e&= \int_{Sp(\widetilde{H})}E P(E) dE=\frac{1}{Z}\int_{Sp(\widetilde{H})}EN(E)e^{-\beta E}dE\\
  & = \frac{1}{Z}\int_{Sp(\widetilde{H})} E\TR\delta(E-\widetilde{H})e^{-\beta E} dE \\
  &=\frac{1}{Z}\TR(\widetilde{H} e^{-\beta \widetilde{H}})=-\frac{1}{Z}(\partial\beta.Z)=-(\partial\beta.\ln Z).\\
\end{array}
    \right.
\end{equation}

 When we can interpret
$\beta=\frac{1}{\kappa_B\theta}$, then one can write
\begin{equation}\label{partition-function-average-energy3}
e=\kappa_B\theta^2(\partial\theta.\ln Z).
\end{equation}
Then we get also that the local {\em energy fluctuation} is
expressed by means of the variance of $e$:
\begin{equation}\label{partition-function-energy-fluctuation}
<(\triangle E)^2>\equiv <(E-e)^2>=(\partial\beta\partial\beta.\ln
Z).
\end{equation}
Furthermore, we get for the {\em local heat capacity} $C_v$ the following
formula:
\begin{equation}\label{heat-capacity-v}
C_v=(\partial\theta.e)=\frac{1}{\kappa_B\theta^2}<(\triangle E)^2>.
\end{equation}
We can define the {\em local entropy} by means of the following
formula:
\begin{equation}\label{entropy1}
s=-\kappa_B\int_{Sp(\widetilde{H})}P(E) \ln P(E) dE.
\end{equation}
In fact one can prove that one has the usual relation by means of
the energy. (See Tab.9.) Really we get:

\begin{equation}\label{entropy2}
\left\{
\begin{array}{ll}
  s&  =-\kappa_B\int_{Sp(\widetilde{H})}P(E) \ln P(E) dE\\
  & =\kappa_B(\ln Z+\beta e)=(\partial\theta.(\kappa_B\theta \ln Z)).\\
\end{array}
\right.
\end{equation}
Then from the relation $s=\kappa_B(\ln Z+\beta e)$ we get
$e=\theta s-\kappa_B\ln Z$, hence also
$(\partial s.e)=\theta$. This justifies the definition of entropy
given in (\ref{entropy1}). Furthermore, from (\ref{entropy2}) we get also $\frac{1}{\beta}\ln Z=e-\theta s=f$,
 where $f$is the Helmoltz free
energy. It follows the following
expression of the local Helmoltz free energy, by means of the local
partition function $Z$:
\begin{equation}\label{free-energy-partition-function}
f\equiv e-\theta s=-\kappa_B\theta\ln Z.
\end{equation}
Conversely, from (\ref{free-energy-partition-function}) it follows that the partition function can be expressed by means
of the local Helmoltz free energy
\begin{equation}\label{partition-function-free-energy}
Z=e^{-\beta f}.
\end{equation}

So we see that the local thermodynamic functions, can be expressed as
scalar-valued differential operators on the fiber bundle
$W[i]\times_NT^0_0N\to N$. The situation is resumed in Tab.10.
$$\begin{tabular}{|l|l|l|l|} \hline
\multicolumn {4}{|c|}{\bsmall Tab.10 - Local thermodynamics functions of
{\boldmath$\scriptstyle
\widehat{(YM)}[i]$} solutions.}\\
\hline\hline $\scriptstyle \hbox{\rsmall Name}$&$\scriptstyle \hbox{\rsmall Definition}$ &
$\scriptstyle \hbox{\rsmall Remark}$&$\scriptstyle \hbox{\rsmall Order}$\\
\hline $\scriptstyle \hbox{\rsmall partition
function}$&$\scriptstyle Z=\TRS(e^{-\beta\widetilde{H}})$&
$\scriptstyle Z=Z(\beta,\widetilde{\mu}^K_\alpha)$&$\scriptstyle 1$\\
\hline $\scriptstyle \hbox{\rsmall interior energy}$&$\scriptstyle
e=-(\partial\beta.\ln Z)$&
$\scriptstyle e=\kappa_B\theta^2(\partial\theta.\ln Z)$&$\scriptstyle 1$\\
\hline $\scriptstyle \hbox{\rsmall fluctuation interior
energy}$&$\scriptstyle <(\triangle E)^2>=<(E-e)^2>$&
$\scriptstyle <(\triangle E)^2>=(\partial\beta\partial \beta.\ln Z)$&$\scriptstyle 2$\\
\hline $\scriptstyle \hbox{\rsmall entropy}$&$\scriptstyle
s=\kappa_B(\ln Z+\beta e)$&
$\scriptstyle s=\partial\theta.(\kappa_B\ln Z)$&$\scriptstyle 1$\\
\hline $\scriptstyle \hbox{\rsmall free energy}$&$\scriptstyle
f=e-\theta s$&$\scriptstyle f=-\kappa_B\theta\ln Z$&$\scriptstyle 1$\\
\hline
\multicolumn {3}{l}{\rsmall It is assumed a Lagrangian of first derivation order.}\\
\end{tabular}$$

The following lemmas relate spectral measures identified by $\widetilde{H}(p)$, $p\in N$, and the local partition
function introduced in (\ref{partition-function2}). It is useful to recall some definitions and results about quantum states.
More precisely a {\em quantum state} of a quantum algebra $A$, is a function $S:A\to\mathbb{C}$, that satisfies the following
properties: (i) $S$ is $\mathbb{C}$-linear; (ii) $S$ is self-adjoint; (iii) normalized by the constraint $\sup_{\|a\|\le 1}S(a^*a)=1$.
a {\em quantum pure state} is one which is not a linear combination with positive coefficients of two other states, otherwise
it is called {\em mixed}. A mixed state is described by its associated {\em density operator} $\rho=\sum_sp_s|\psi_s><\psi_s|$,
where $p_s$ is the fraction of the set in each pure state $|\psi_s>$. A criterion to see whether a density operator is describing
a pure or mixed state is that $\TR(\rho^2)=1$ for pure state, and $\TR(\rho^2)<1$ for mixed state.

\begin{lemma}{\em(Gelfand-Naimark-Segal construction).}\label{gns-construction}
If $A$ is a $C^*$-algebra, then every state on $A$ is of the following type $a\mapsto<\xi,\pi(a)(\xi)>$, where $\pi:A\to L(\mathcal{H})$
is a representation of $A$ in an Hilbert space $\mathcal{H}$, and $\xi\in\mathcal{H}$ is a {\em cyclic vector} for $\pi$, i.e.,
$\pi(A)(\xi)$ is norm dense in $\mathcal{H}$, hence $\pi$ is a {\em cyclic representation}.
\end{lemma}

A set of quantum states of $A$ is called {\em complete} if the only element of $A$ which vanishes in every state of the set is
zero. The {\em variance} of a self-adjoint element $a\in A$, in a quantum state $S$, is defined by $S(a^2)-S(a)^2$. $a$ is said
to {\em have the exact value} $S(a)$, in the state $S$, in the case its variance vanishes in this state. A {\em quantum value}
of $\widetilde{H}(p)$ is a point $\lambda(p)\in Sp(\widetilde{H}(p))$. The {\em probability to find a quantum value} of
$\widetilde{H}(p)$ in the Borel set $U\subset\mathbb{R}$, if the system is in the quantum state $S$, is given by the following
formula
\begin{equation}\label{probability-spectral-measure1}
    p(\widetilde{H}(p);S,U)=\TR(E_{\widetilde{H}(p)}(U)S)=\TR(SE_{\widetilde{H}(p)}(U)),
\end{equation}
where $E_{\widetilde{H}(p)}$ is the spectral measure of $E_{\widetilde{H}(p)}$. If $S=\psi\otimes\psi$, we have
\begin{equation}\label{probability-spectral-measure2}
    p(\widetilde{H}(p);S,U)=\int_Ud(E_{\widetilde{H}(p)})_\psi(\lambda),
\end{equation}
where $(E_{\widetilde{H}(p)})_\psi:\mathcal{B}(\mathbb{R})\to\mathbb{R}$ is the measure on the $\sigma$-algebra of Borel subsets
of $\mathbb{R}$, given by $(E_{\widetilde{H}(p)}(U))_\psi=<E_{\widetilde{H}(p)}(U)(\psi),\psi>$. Then the {\em mean value}, or {\em expectation value},
of $\widetilde{H}(p)$ in the state $S$ is given by the following formula:
\begin{equation}\label{mean-value1}
\left\{
\begin{array}{ll}
  <\widetilde{H}(p)>_S& =\int_{Sp(\widetilde{H}(p))}\lambda dE_{\widetilde{H}(p)}S\\
  & =\TR(\widetilde{H}(p)S)=\TR(S\widetilde{H}(p)).\\
\end{array}
\right.
\end{equation}
If $S=\psi\otimes\psi$ we have
\begin{equation}\label{mean-value2}
<\widetilde{H}(p)>_S=\int_{Sp(\widetilde{H}(p))}\lambda d(E_{\widetilde{H}(p)}(\lambda))_\psi.
\end{equation}
The corresponding {\em variance}, in the state $S$, is given by the following:
\begin{equation}\label{variance1}
\left\{
\begin{array}{ll}
\sigma^2(\widetilde{H}(p))_S&=\int_{Sp(\widetilde{H}(p))}(\lambda -<\widetilde{H}(p)>)^2dp(\widetilde{H}(p);S,\lambda)\\
&=\TR(\widetilde{H}(p)^2S)-\TR(\widetilde{H}(p)).\\
\end{array}
\right.
\end{equation}
If $S=\psi\otimes\psi$ we have

\begin{equation}\label{variance2}
\left\{
\begin{array}{ll}
\sigma^2(\widetilde{H}(p))_S&=\int_{Sp(\widetilde{H}(p))}(\lambda -<\widetilde{H}(p)>)^2d(E_{\widetilde{H}(p)})_\psi(\lambda)\\
&=\|\widetilde{H}(p)(\psi)\|^2-<\widetilde{H}(p)(\psi)|\psi>^2.\\
\end{array}
\right.
\end{equation}

\begin{lemma}{\em(Gibbs canonical quantum state).}\label{gibbs-canonical-state}
The local partition function given in {\em(\ref{partition-function1})} is normalization factor of a quantum state called
{\em Gibbs canonical quantum state}.
\end{lemma}

\begin{proof}
Let $E_{\widetilde{H}(p)}:(\mathbb{R},\mathcal{B})\SRA\widehat{A}$ be the spectral measure on the
$\sigma$-algebra $\mathcal{B}\equiv\mathcal{B}(\mathbb{R})$ of Borel subsets of $\mathbb{R}$, uniquely identified by $\widetilde{H}(p)$, $p\in N$. Then,
given a state $S$, the distribution of $\widetilde{H}(p)$ under $S$ is the probability measure on $\mathcal{B}$, given by
\begin{equation}\label{distribution}
    D_{\widetilde{H}(p)}(U)=\TR(E_{\widetilde{H}(p)}(U)S)
\end{equation}
and the expected value, in the state $S$, of $\widetilde{H}(p)$, is
\begin{equation}\label{expectation}
    <\widetilde{H}(p)>_S=\int_{\mathbb{R}}\lambda dD_{\widetilde{H}(p)}(\lambda).
\end{equation}
One has $<\widetilde{H}(p)>_S=\TR(\widetilde{H}(p)S)=\TR(S\widetilde{H}(p))$. If $S$ is a pure state corresponding to the vector $\psi$, then $<\widetilde{H}(p)>_S=<\psi|\widetilde{H}(p)|\psi>$. Furthermore, let $Sp(\widetilde{H}(p))=Sp(\widetilde{H}(p))_p$, i.e., let us assume that $\widetilde{H}(p)$ has a pure point-spectrum with eigenvalues $E_n$, that go to $+\infty$ as sufficiently fast, then $e^{-\beta(p)\widetilde{H}(p)}$ will be a non-negative trace-class operator for every positive $\beta(p)\in\mathbb{R}$. One defines {\em Gibbs canonical quantum state}
\begin{equation}\label{gibbs-canonical-quantum-state}
   S=\frac{e^{-\beta\widetilde{H}}}{\TR(e^{-\beta\widetilde{H}})}=\frac{e^{-\beta\widetilde{H}}}{\sum_ne^{-\beta E_n}}.
\end{equation}
Then, $\TR(S\widetilde{H})=e$, as obtained in (\ref{partition-function2}). Furthermore, the definition of entropy given in (\ref{entropy2}) just coincides with the von Neumann entropy of the state $S$:
$s_{von-Neumann}=\TR(S\ln S)$ \cite{NEU}.\footnote{The state $S$ can be diagonalized and one can write
$s_{von-Neumann}=-\sum_i\lambda_i\ln\lambda_i$, with the convention $0.\ln0=0$. This is a real number belonging
to $[0,+\infty]$. $s_{von-Neumann}$ measures the amount of randomness in the state $S$. (Larger entropy corresponds
to more dispersed eigenvalues.) The quantum state $S$ is called {\em maximally mixed quantum state} if it
maximalizes $s_{von-Neumann}$. $s_{von-Neumann}(S)=0$ iff $S$ is a pure state, i.e., $S=|\psi><\psi|$.}
\end{proof}
From above results it follows that the way to define local thermodynamic functions, by means of local partition function,
coincides with the expectation value of energy in a Gibbs canonical state, when the spectrum of the observed Hamiltonian
is only a point spectrum.
\end{proof}

\begin{remark}
The approach given here to formulate the local thermodynamics of quantum systems, differs from one actually adopted in the literature.
In fact, this last, usually necessitates to quantize classical systems and then characterizes
associated thermodynamic functions by means of the energy-momentum tensor anomaly $<\widetilde{T}^\mu_\mu>\equiv\epsilon$. (See, e.g., Refs.\cite{C-N-Z1, C-N-Z2}.)
Really the trace of the classic energy-momentum tensor is zero, but in the process of quantization this conservation is
not more assured and the trace of the quantized classic energy-momentum tensor can take a non-zero value:
$\epsilon\not=0$. This discrepancy between classic and quantum formulation is to ascribe to the fact that the process of
quantization is performed by means of a linearization of the dynamic in the neighborhood of the classic solution, hence
it discards non-linear effects. However, in our non-commutative framework, there is not discrepancy in the covariant
description of quantum systems, with respect to eventual classic or superclassic analogues. Therefore, we can directly implement thermodynamic function on the observed quantum solutions, characterizing their spectral properties.
\end{remark}

The concept of quantum states can be also related to a proof for existence of solutions with mass-gap. In fact, we have the following theorem, that completes Theorem \ref{obstruction-mass-gap}, and recognizes an interior constraint in $\widehat{(YM)}$, where live solutions with mass gap.

\begin{theorem}{\em(Existence of $\widehat{(YM)}$ solutions with mass-gap.)}\label{existence-mass-gap-solutions}
Equation $\widehat{(YM)}$ admits local and global solutions with mass-gap. These are contained into a sub-equation,
{\em(Higgs-quantum super PDE)},
$\widehat{(Higgs)}\subset\widehat{(YM)}$, that is formally integrable and completely integrable, and also
 a stable quantum super PDE. If $H_3(M;\mathbb{K})=0$, $\widehat{(Higgs)}$ is also a quantum extended crystal super PDE.
 In general
 solutions contained in $\widehat{(Higgs)}$ are not stable in finite times. However there exists an associated stabilized
 quantum super PDE, (resp. quantum extended crystal super PDE), where all global smooth solutions are stable in finite
 times.
 Furthermore, there exists a quantum super partial differential relation, {\em(quantum Goldstone-boundary)},
 $\widehat{(Goldstone)}\subset\widehat{(YM)}$, bounding $\widehat{(Higgs)}$, such that any global solution of
 $\widehat{(YM)}$,
 loses/acquires mass, by crossing $\widehat{(Goldstone)}$.\footnote{Since the symmetries properties of the
 Higgs-quantum super PDE $\widehat{(Higgs)}$ are different from the ones of the quantum super Yang-Mills equation
 $\widehat{(YM)}$, we can consider this theorem like a quantum dynamic generalization of the usual
 Higgs-breaking-symmetry mechanism. This justifies the name given, and the symbol, used to denote the constraint in $\widehat{(YM)}$ where live the solutions with mass-gap.}
\end{theorem}
\begin{proof}
The Hamiltonian of $\widehat{(YM)}$ can be identified with a $Q^\omega_w$-function
$\hat H:J\hat{\it D}(W)\to\widehat{A}$. Therefore, the set of solutions, with mass-gap, of the trivial equation
$J\hat{\it D}(W)\subseteq J\hat{\it D}(W)$ can be identified with the subset
$\widehat{H}_{(Higgs)}\equiv\hat H^{-1}(G(\widehat{A}))\subset J\hat{\it D}(W)$, where $G(\widehat{A})$ is the abelian group
of units of $\widehat{A}$. In general one has $G(\widehat{A})\subseteq \widehat{A}$, where the equality happens iff $\widehat{A}$ is
a division algebra. It is well known that the only associative division algebras are the real $\mathbb{R}$, complex $\mathbb{C}$,
and quaternion $\mathbb{H}$ numbers. Let us exclude these cases.\footnote{If $A$ is an associative division algebra,
 identified with $\mathbb{K}\equiv\mathbb{R}$ or $\mathbb{K}\equiv\mathbb{C}$, the constraint
$\widehat{H}_{(Higgs)}$ is trivially an open set
$\widehat{H}_{(Higgs)}\subset J\hat{\it D}(W)$. In fact, when
$A=\mathbb{K}$, one has $Z(A)=\mathbb{K}$ and
$\widehat{\mathbb{K}}=\mathbb{K}=A$, hence $G(\widehat{A})$ is the
open set $G(\widehat{A})=A\setminus\{0\}$. Therefore,
$\widehat{H}_{(Higgs)}\equiv H^{-1}(A\setminus\{0\})$, is
necessarily an open quantum submanifold of $J\hat{\it D}(W)$.} Therefore
solutions of $\widehat{(YM)}$ with mass-gap are all the solutions of
the following augmented partial differential relations:
\begin{equation}\label{augmented-differential-relation-mass-gap}
   \left\{ \widehat{(Higgs)}\equiv \widehat{(YM)}\bigcap  \widehat{H}_{(Higgs)}\subset J\hat{\it D}^2(W)\right\}.
\end{equation}
In order to study the topological-differential structure of $\widehat{H}_{(Higgs)}$, and hence that of $ \widehat{(Higgs)}$, let us consider the quantum states.
In fact, between quantum states, useful to characterize energetic contents of $\widehat{(YM)}$ solutions,
there are ones coming from characters of quantum algebras. More precisely we give the following definition.
\begin{definition}
We define {\em character-quantum state} a character of the quantum algebra $\widehat{A}$, i.e.,
a unitary multiplicative linear function
$\chi:\widehat{A}\to\mathbb{C}$. We denote by $Ch(\widehat{A})$ the set of such quantum states.\footnote{We call
$Ch(\widehat{A})$
also {\em spectrum} of $\widehat{A}$ and we write $Sp(\widehat{A})=Ch(\widehat{A})$.}
\end{definition}
One has the following lemma that gives an alternative way to identify solutions of $\widehat{(YM)}$ with mass-gap.

\begin{lemma}{\em(Character-quantum states and mass-gap)}.\label{character-quantum-states-mass-gap}
A solution of $\widehat{(YM)}$ has mass-gap in $p\in M$, iff for any character-quantum-state
$\chi\in Ch(\widehat{A})$, one has $\chi(H(p))\not=0$.
\end{lemma}

\begin{proof}
It is a direct consequence of the following lemma.

\begin{lemma}\label{characters-group-units}
Let $B$ be a ${\mathbb K}$-algebra. Then $ b\in G(B)$ iff
$\chi(b)\not=0$, $\forall \chi\in Ch(B)$, where $ Ch(B)$ is the
set of characters of $ B$. (One has
also $\chi(b)\in Sp(b)$, $\forall b\in B$.)\end{lemma}

\begin{proof}
It is standard. (See e.g., Ref.\cite{BOU}.)
\end{proof}
In fact we have $H(p)\in\widehat{A}$, for any $p\in N$.
\end{proof}

\begin{lemma}{\em(Topology of $G(\widehat{A})$).}\label{topology-units-group}
$G(\widehat{A})$ is an open set in $\widehat{A}$.
\end{lemma}
\begin{proof}
Let us emphasize that the group of units of a topological ring may not be a topological group using the subspace topology, as
inversion on the unit group need not be continuous with the subspace topology. However, we have the following lemma.
\begin{lemma}\label{inversion-quantum-algebra}
The operation of taking an inverse is continuous on quantum algebras.
\end{lemma}
\begin{proof}
In fact, a quantum (super)algebra is a Fr\'echet algebra, and for such algebras it is well known that the operation of taking an inverse is continuous \cite{HAZ}.
\end{proof}
Let us use, now, the following lemma.
\begin{lemma}
The operation of taking an inverse is continuous for a general $F$-algebra $A$ iff the group $G(A)$ of its invertible
elements is a $G_\delta$-set.\footnote{In a topological space, a $G_\delta$-set is one countable intersection of open sets.}
\end{lemma}
\begin{proof}
See, e.g., \cite{HAZ} and references quoted there.
\end{proof}
Then taking into account that a quantum (super)algebra is a Fr\'echet algebra, hence must necessarily $G(A)$ be a $G_\delta$-set,
therefore an open set in $A$. Note that this result could directly follow from Lemma \ref{characters-group-units} if we could
state that all characters of $\widehat{A}$ are continuous functions. But this is not assured for a topological algebra.
({\em Michael-Mazur's problem} \cite{MICH}.) By the way, we can state that there exists a countable subset
$Ch(\widehat{A})^0_c\subseteq Ch(\widehat{A})$ of continuous characters such that
$G(\widehat{A})=\bigcap_{\chi\in Ch(\widehat{A})^0_c}G(\widehat{A})_\chi$,
where $G(\widehat{A})_\chi\equiv \chi^{-1}(\mathbb{C}\setminus\{0\})$ is an open set
in $\widehat{A}$. This open set cannot be empty, since it surely contains the identity map $1_A\in \widehat{A}$.
\end{proof}
\begin{lemma}{\em(Topology of $\widehat{H}_{(Higgs)}$).}\label{topology-mass-gap-constraint}
$\widehat{H}_{(Higgs)}$ is an open quantum submanifold of $J\hat{\it D}(W)$.
\end{lemma}
\begin{proof}
From above lemma it follows that $\widehat{H}_{(Higgs)}$ is an open set in $J\hat{\it D}(W)$ since $H$ is a continuous
map as it is of class $Q^\omega_w$.
\end{proof}
Let us, now, remark that the restriction $\bar\pi_{2,1}$ to $\widehat{(YM)}$ of the canonical projection
$\pi_{2,1}:J\hat{\it D}^2(W)\to J\hat{\it D}(W)$, gives a surjective mapping $\bar\pi_{2,1}:\widehat{(YM)}\to J\hat{\it D}(W)$
too. Therefore, we can identify the
following constraint $\widehat{(Higgs)}\equiv \bar\pi_{2,1}^{-1}(\widehat{H}_{(Higgs)})\subset\widehat{(YM)}$. One has the
following commutative diagram, with vertical exact lines.

\begin{equation}\label{commutative-diagram-mass-gap-constraint}
    \xymatrix@1@C=60pt{\widehat{(Higgs)}\ar[d]^{\tilde\pi_{2,1}\equiv\bar\pi_{2,1}|_{\widehat{(Higgs)}}}\hskip 3pt\ar@{^{(}->}[r]&\widehat{(YM)}\ar[d]^{\bar\pi_{2,1}\equiv\pi_{2,1}|_{\widehat{(YM)}}}\hskip 3pt\ar@{^{(}->}[r]&
    J\hat{\it D}^2(W)\ar[d]^{\pi_{2,1}}\\
\widehat{H}_{(Higgs)}\ar[d]\hskip 3pt\ar@{^{(}->}[r]&J\hat{\it D}(W)\ar[d]\ar@{=}[r]&J\hat{\it D}(W)\ar[d]\\
0&0&0\\}
\end{equation}
Taking into account that $\bar\pi_{2,1}$ is a continuous mapping and that $\widehat{H}_{(Higgs)}$ is an open subset of $J\hat{\it D}(W)$,
it follows that also $\widehat{(Higgs)}$ is an open quantum submanifold of $\widehat{(YM)}$ over the open quantum submanifold
$\widehat{H}_{(Higgs)}$ of
$J\hat{\it D}(W)$. Since $\widehat{(YM)}$ is formally integrable and completely integrable, it follows that also
$\widehat{(Higgs)}$ is so. By conclusion, the situs of solutions of $\widehat{(YM)}$ with mass-gap is the
formally integrable and completely integrable quantum super PDE $\widehat{(Higgs)}\subset \widehat{(YM)}$. For any point
$q\in \widehat{(Higgs)}$ there exists a solution of $\widehat{(YM)}$, having mass-gap. The characterization of global solutions
is made by means of integral bordism groups of $\widehat{(Higgs)}$. Since this last equation is an open quantum submanifold of
$\widehat{(YM)}$ its integral bordism groups coincide with the ones of $\widehat{(YM)}$. (See Theorem \ref{quantum-crystal-structure-ym}.)
Therefore, if $H_3(M;\mathbb{K})=0$, equation $\widehat{(Higgs)}$ is a quantum extended crystal super PDE. Moreover, under the full
admissibility hypothesis, it becomes a quantum $0$-crystal super PDE. From Theorem \ref{criteria-fun-stab} it follows that
$\widehat{(Higgs)}$ is a stable quantum extended crystal super PDE. Since its symbol is not trivial, we get also that in general
such solution with mass-gap are not stable in finite times. However, we get also that for the infinity prolongation of
such an equation, one has $\widehat{(Higgs)}_{+\infty}\subset \widehat{(YM)}_{+\infty}$. Hence we can state that
$\widehat{(Higgs)}$ is a stabilizable quantum extended crystal super PDE, with stable quantum extended crystal super PDE
just $\widehat{(Higgs)}_{+\infty}$. There all global smooth solutions have mass-gap, and are stable in
finite times. In order to fix ideas, let us, before to complete the proof of Theorem \ref{existence-mass-gap-solutions},
consider the following example containing also some useful definitions.

\begin{example}
We call a global solution $V\subset\widehat{(YM)}$ a {\em quantum matter-solution} when
$V\cap\widehat{(Higgs)}\neq\varnothing$, otherwise we say that $V$ is a {\em quantum matter-free-solution}. Of the first type
are solutions describing, e.g., the following particles or nuclear reactions: $\gamma+p\to\pi^{\mathop{\pm}\limits^\circ}+p$,
$\pi^+\to\mu^++\nu_\mu$, $Li+H\to 2 He+22.4 Mev$. Instead examples of quantum matter-free-solutions are ones encoding
electro-magnetic fields or gluons fields. A generic global solution can be made by pieces that are completely
inside $\widehat{(Higgs)}$, {\em quantum pure-matter-solution}, and other ones that are outside $\widehat{(Higgs)}$,
{\em quantum matter-free-solution}. The case of photoproduction of pions
is an example of solutions of this type. Another one is a monopole-vortex chain in $SU(2)$ gauge theory.
Then Theorem 2.19 in \cite{PRA31} gives a decomposition of such solutions as union of a finite number of (integral)
bordisms between elementary bordisms. Some of these are quantum pure-matter-bordisms (resp. quantum matter-free-bordisms), since completely
inside, (resp. outside) $\widehat{(Higgs)}$. (See Fig.\ref{quantum-matter-solution}.) Let $V\subset \widehat{(YM)}$ be such a quantum matter-solution,
such that $\partial V=N_0\bigcup P\bigcup N_1$, with $N_0\subset \widehat{(Higgs)}$ and
$N_1\subsetneq\widehat{(Higgs)}$.
Set ${}_0V\equiv V\bigcap\widehat{(Higgs)}$ and ${}_\bullet V\equiv V\setminus {}_0V\NOTSUBSET \widehat{(Higgs)}$.
Then we can write $V={}_0V\bigcup {}_\bullet V$. Since $\widehat{(Higgs)}$ is open in $\widehat{(YM)}$, it follows that
${}_0V\bigcap (\overline{\widehat{(Higgs)}}\setminus\widehat{(Higgs)})=\varnothing$ and
${}_\bullet V\bigcap (\overline{\widehat{(Higgs)}}\setminus\widehat{(Higgs)})\equiv V_{\partial}\NOTSUBSET \widehat{(Higgs)}$.
Furthermore one has $V_{\partial}\in[N_1]=[N_0]\in Bor_{m-1|n-1}(\widehat{(YM)};A)$.

\begin{figure}
\centerline{\includegraphics[height=6cm]{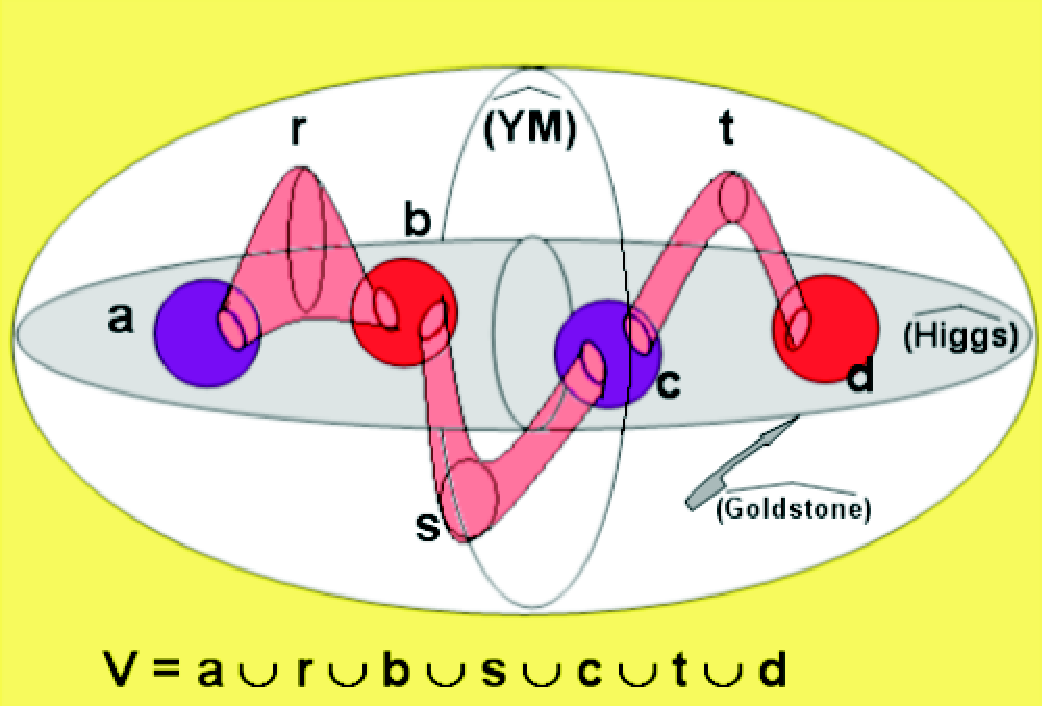}}
\caption{Quantum matter-solution and elementary bordism-decompositions. $a$, $b$, $c$ and $d$ are quantum-matter solutions.
$r$, $s$ and $t$ are quantum-matter-free solutions. All the solution $V\equiv a\cup r\cup b\cup s\cup t\cup c\cup d$ is
a quantum-matter-solution.\label{quantum-matter-solution}}
\end{figure}
\end{example}

Let us conclude the proof of the theorem by considering the boundary $\widehat{(Goldstone)}\equiv\partial\widehat{(Higgs)}$ of $\widehat{(Higgs)}$,
i.e., $\widehat{(Goldstone)}\equiv \overline{\widehat{(Higgs)}}\setminus\widehat{(Higgs)}$. Since
$\widehat{(Goldstone)}\subset\widehat{(YM)}$, for any point $q\in\widehat{(Goldstone)}$ there exists a solution
$r\subset\widehat{(YM)}$ of $\widehat{(YM)}$. $r$ can be a quantum matter-free-solution. However, since
$\widehat{(Goldstone)}$ is dense in $\widehat{(Higgs)}$,  $q$ is limit point of a sequences of points
$q_i\in\widehat{(Higgs)}$, $i\in J$. To each point $q_i$ there corresponds a quantum matter-solution
$s_i\subset \widehat{(Higgs)}$. This means that for any
quantum matter-free-condition $q\in\widehat{(Goldstone)}$, we can always find quantum matter-solutions
$s_i\subset\widehat{(Higgs)}$ converging to $q$. Then by a surgery technique we can prolong such
quantum matter-solutions to a solution $V$, across
$\widehat{(Goldstone)}$, soldering with the quantum matter-free solution $r$. Since $\widehat{(YM)}$ and $\widehat{(Higgs)}$
are formally integrable and completely integrable, we can repeat this surgery technique to the infinity prolongations
of these equations, i.e., by considering the sequence:
$\widehat{(Higgs)}_{+\infty}\subset\widehat{(Goldstone)}_{+\infty}\subset\widehat{(YM)}_{+\infty}$, by obtaining also
quantum smooth solutions that passing across $\widehat{(Goldstone)}_{+\infty}$ acquire/lose mass.
Such solutions, are therefore stable in finite times.
\end{proof}

\begin{definition}\label{goldstone-piece-solution}
Let $V\subset\widehat{(YM)}$ be a global solution crossing the quantum Goldstone-boundary $\widehat{(Goldstone)}$. Let
us call $V_{\partial}\equiv V\bigcap \widehat{(Goldstone)}$ the {\em Goldstone-piece} of $V$.
\end{definition}

\begin{cor}{\em(Goldstone-piece characterization of $\widehat{(YM)}$ solutions).}\label{goldstone-piece-characterization}
In any connected global solution $V\subset\underline{(YM)}$ bording Cauchy data
$N_0\subset \widehat{(YM)}\setminus\overline{\widehat{(Higgs)}}$ and $N_1\subset\widehat{(Higgs)}$, there exists a
Goldstone piece.\footnote{Theorem \ref{existence-mass-gap-solutions} and Corollary \ref{goldstone-piece-solution},
thanks to our geometric theory of quantum (super)PDE's, give a full quantum dynamical justification to mass
production/destruction mechanism in quantum solutions of $\widehat{(YM)}$. These results can be considered generalizations, in the geometric theory of quantum PDE's,
of the well-known Goldstone's theorem, and ''Higgs symmetry-breaking mechanism'' conjectured, in the second middle of the
last century, in order to justify mass in a gauge theory. (See Refs.\cite{GOLDST, GOLDST-SAL-WEIN, HIGGS1, HIGGS2}.)}
\end{cor}.

\end{document}